\documentclass[11pt,reqno]{amsart}

\usepackage[x11names]{xcolor}
\usepackage[linkcolor=Blue2,citecolor=black,colorlinks]{hyperref}

\usepackage{color, bm, amscd, tikz-cd}

\usepackage{mathdots}

\usepackage{amsmath,amsthm,amssymb, mathrsfs, bbm}

\allowdisplaybreaks

\theoremstyle{plain}
\newtheorem{theorem}{Theorem}[section]
\newtheorem{corollary}[theorem]{Corollary}
\newtheorem{lemma}[theorem]{Lemma}
\newtheorem{proposition}[theorem]{Proposition}

\theoremstyle{definition}
\newtheorem{definition}[theorem]{Definition}
\newtheorem{example}[theorem]{Example}

\newtheorem{remark}[theorem]{Remark}

\numberwithin{equation}{section}

\newcommand{\N}{\mathbb{N}}

\newcommand{\R}{\mathbb{R}}
\newcommand{\C}{\mathbb{C}}

\newcommand{\ltt}{\left(}
\newcommand{\rtt}{\right)}

\newcommand{\beas}{\begin{eqnarray*}}
	\newcommand{\eeas}{\end{eqnarray*}}
\newcommand{\bes} {\begin{equation*}}
	\newcommand{\ees} {\end{equation*}}
\newcommand{\be} {\begin{equation}}
	\newcommand{\ee} {\end{equation}}
\newcommand{\bea} {\begin{eqnarray}}
	\newcommand{\eea} {\end{eqnarray}}
\newcommand{\beals} {\begin{align*}}
	\newcommand{\eeals} {\end{align*}}
\newcommand{\beal} {\begin{align}}
	\newcommand{\eeal} {\end{align}}

\newcommand{\rnk} {\operatorname{rank}}

\newcommand{\bydef}{\stackrel{\rm def}{=}}
\usepackage{mismath}

\begin{document}
	\title[Completely monotone functions] {Holomorphic interpolation of multivariate completely monotone functions}
    
	\author[M. Bhowmik]{Mainak Bhowmik }
	\address{Indian Institute of Technology, Bombay, India }
	\email{\tt mainak.bhowmik943@gmail.com, mainak@math.iitb.ac.in}
	
	\author[A. Chatterjee]{Agniva Chatterjee}
	\address{Indian Institute of Technology Palakkad, Palakkad, Kerala, India }
	\email{\tt agnivac@iitpkd.ac.in, agnivachatterjee3@gmail.com}

	\author[M. Putinar]{Mihai Putinar}
	\address{University of California at Santa Barbara, CA} 
	\email{\tt mputinar@math.ucsb.edu}

	
	
	
	
	\begin{abstract}  
    The integral representation of completely monotone functions of several real variables as Laplace or Stieltjes–Fantappié transforms of positive measures opens a Hilbert space path toward their finite-point interpolation by simpler functions. We combine, within a non-commutative Radon transform framework, the matrix pencil realization of the positive semi-definite Hankel kernel associated with the sampling of a completely monotone function with Weyl’s operational calculus and Fantappi\`e's analytic calculus. The interpolation is achieved by finitely determined entire or rational functions, respectively, which are directionally completely monotone. In our relaxation scheme, the original positive measure is approximated by a sequence of specific Wigner distributions, which can also be regarded as analytic functionals. Throughout the interpolation process, tight bounds are enforced on the modulus or the real part of the holomorphic extension to the underlying tube domain.	
	\end{abstract}
	
	\maketitle
    \tableofcontents
	
	\section{Introduction}
	A descriptive title of this article could be {\it ``A non-commutative Radon transform approach to structured holomorphic interpolation of multivariate completely monotone functions via Hankel kernels, Wigner distributions and analytic duality"}. Alluding by ``structured" to one of the classical frameworks: Schur class of holomorphic functions with values in the unit disk, or Herglotz-Nevanlinna class of holomorphic functions with non-negative real part.
	
	Let $\Gamma$ be an acute closed convex solid cone in $\mathbb{R}^d$ and $\Gamma^*= \{ y\in \mathbb{R}^d : y\cdot x \geq 0 \text{ for } x\in \Gamma\}$ be its closed dual cone.
    \begin{definition}[Completely monotone functions]
		A real-valued function $f$ on $\operatorname{int} \Gamma$ is said to be {\em completely monotone function} if $f \in C^\infty(\operatorname{int} \Gamma)$ and 
	\begin{align}\label{Eq:CM-cond}
			(-1)^k D_{\xi^{(1)}} D_{\xi^{(2)}} \dots D_{\xi^{(k)}} F(p) \geq 0,
		\end{align}
		for all $p\in \operatorname{int} \Gamma$ and all $\xi^{(1)}, \dots, \xi^{(k)} \in \Gamma$, where
		$$
		D_{\xi} = \sum_{j=1}^{n} \xi_j \frac{\partial}{\partial p_j}, \ \ \ \xi=(\xi_1, \dots, \xi_n).
		$$
\end{definition}
Let $T_\Gamma = \operatorname{int} \Gamma + i  \ {\mathbb R}^d \subset \mathbb{C}^d$ be the tube domain over the cone $\Gamma$. Every completely monotone function $f$ on $\operatorname{int} \Gamma$ is the restriction of (complexified) Laplace transform of a positive measure $\mu$ supported on $\Gamma^*$, given by
$$ f(z) = \int_{\Gamma^\ast} e^{- z \cdot x} d\mu(x), \ \ z \in T_\Gamma.$$

Let $T_{\Gamma \times \mathbb{R}_+}$ denote the tube domain $(\operatorname{int} \Gamma \times \mathbb{R}_+) + i\ \mathbb{R}^{d+1} \subset \mathbb{C}^{d+1}$. Then, every homogeneous completely monotone function $g(p, p_0)$  of homogeneity $-1$ on the cone $\Gamma \times \mathbb{R}_+$ is the restriction of (complexified) Stieltjes-Fantappi\`e transform of a positive measure $\mu$ on $\Gamma^*$, defined as
$$ 
g(z,z_0) = \int_{\Gamma^\ast} \frac {d\mu(x)}{z_0 + z\cdot x}, \ \ (z, z_0) \in T_{\Gamma \times \mathbb{R}_+}.
$$
These characterizations of completely monotone functions via the integral representations associated to a {\it positive} measure go back to the celebrated memoir of Bernstein \cite{Bernstein} amended by multivariate generalizations by Gilbert, Bochner \cite{Bochner}, Henkin and Shannanin \cite{HS}, Thomas \cite{Thomas}.
The Laplace and the Stieltjes–Fantappi\`e transform of the same measure are connected via the Borel transform. The commutative diagram in Preliminaries well explains this relationship. 

The present work is devoted to an interpolation scheme adapted to the class of holomorphic functions obtained as the Laplace transform or the Stieltjes–Fantappi\`e transform of a positive measure on $\Gamma^*$.
The Laplace transforms and the Stieltjes–Fantappi\`e transforms are strictly contained within the Schur and Herglotz–Nevanlinna classes of holomorphic functions defined on the tube domains $T_\Gamma$ and $T_{\Gamma \times \mathbb{R}_+}$, respectively.

The complications of multivariate Carath\'eodory-Fej\'er or Hermite type holomorphic interpolation and approximation in these larger classes of functions are notorious, cf. \cite{BP2, Guillaume-Huard}. Compared to the classical single variable setting, 
	the lack of contractive divisors and continued fraction decompositions in two or more dimensions rule out an effective free parametrization of functions constrained by bounds on their
modulus or real part. To give a single striking example, the classification of extremal rays of the convex cone of holomorphic functions defined in a polydisk in two or more complex variables and possessing non-negative real part is not known \cite{Knese}.
The major advantage of working with holomorphic extensions of completely monotone functions is their free parametrization by positive Borel measures $\mu$ supported by the dual convex cone $\Gamma^\ast$. 

A consequence of Bernstein's theorem (see Theorem~\ref{Thm:BBG}) is the positive semi-definiteness of the Hankel kernel:
$$  \Xi (\alpha, \beta) = (-1)^{|\alpha + \beta|} \partial_z^{\alpha + \beta} f|_{z=\lambda} = \int_{\Gamma^\ast} x^{\alpha+\beta} e^{- \lambda \cdot x} d\mu(x), \ \ \alpha, \beta \in \mathbb{N}_0^d,$$
associated to the partial derivatives of $f$ at $\lambda \in \operatorname{int} \Gamma$. Similarly, the Hankel kernel associated to $g(z, z_0)$ can be formed.
This is the path we follow to merge Hilbert space methods and interpolation of functions depending on several complex variables.

Hankel kernels are well-known for detecting recognizable formal series from finite data, particularly rational functions given in a series expansion \cite{Fliess}. Not surprisingly, in our context, a finite-rank Hankel kernel as described above provides a certificate for a Laplace transform or a Stieltjes-Fantappi\`e transform of a positive measure to be a finite exponential sum or a finite sum of simple fractions, respectively. This is detailed in Theorem~\ref{Thm-multinode-flat} and Theorem~\ref{th:flat} of this article.
 An early account of exploiting Hankel kernels in the study of completely monotone series and totally positive kernels appeared in Section 7 of Chapter 2 of Karlin's monograph \cite{Karlin}. A moving point study of the jet data of a holomorphic function depending on a complex variable appears in \cite{Lubinsky}.

 The main theme of our article revolves around the question of interpolating and approximating the Laplace transform $f(z)$ by sums of exponentials of the form: 
 $$
 \sum_{j=1}^N b_j e^{-z\cdot \lambda^{(j)}}, \ \exists \ b_j \geq 0, \ \lambda^{(j)} \in \Gamma^*,$$ and the Stieltjes-Fantappi\`e transform $g(z, z_0)$ by sums of simple fractions of the form:
 $$\sum_{j=1}^N \frac{c_j}{z_0 + z \cdot q^{(j)}}, \  \exists \ c_j \geq 0, \ q^{(j)} \in \Gamma^*.$$
 
The problem essentially boils down to solving a multivariate truncated moment problem: from partial data involving the Taylor coefficients of $f(p), \ p \in \operatorname{int} \Gamma,$ (or $g(p, p_0), \ p \in \operatorname{int} \Gamma, p_0 >0$) at a point in the real cone $\Gamma$, produce a sequence of finitely supported positive measures $\mu_n$ which converges to the generating measure $\mu$ in weak-$\ast$ topology and interpolates the data.

A fundamental theorem of Richter and Tchakalov (Theorem~\ref{Thm:RT}) guarantees the existence of such finitely supported measures, and hence, of the interpolants as their Laplace or Stieltjes–Fantappi\`e transforms. However, this result is purely theoretical. A constructive approach to finding such interpolants is central to many applications and has been thoroughly studied for decades \cite{Beylkin, Kammler-2, Kammler, Koyama, Loy-Anderssen, Maergoiz}. The ample article by Henkin and Shananin  \cite{HS} is motivated by a mathematical economics; there one can find a multivariable generalization of Widder approximation method, applicable to both Laplace and Stieltjes-Fantappi\`e settings.  A lucid presentation on the interpolation of univariate completely monotone functions, as derived from solving Hausdorff's moment problem via finite differences, appears in Feller's note \cite{Feller}. In general, the univariate setting is much better understood due to classical by now results referring to the underlying moment problem and probability theory
\cite{Akhiezer, SSV}.

All attempts at solving the multivariate truncated moment problem in a variety of impersonations (cubature formulae, integral transforms inversions, polynomial optimization) via a computationally accessible manner underscore the intrinsic difficulty of the original interpolation problem \cite{Cools, Mnatsakanov, Lasserre}. To state that this problem is ill-posed is an understatement; for instance, the underlying measure $\mu$ can be moment-indeterminate. For general multivariate moment problems, the Hankel kernel method is well documented in recent works such as \cite{Mourrain-cubature, Mourrain, Evelyne}.  We will not, however, pursue our investigation along this traditional path.

Instead, we propose a direct Hilbert space operator factorization of the truncated, positive semi-definite kernel $\Xi$ arising from the prescribed jet data (i.e., the derivatives of $f(z)$ or $g(z, z_0)$ up to a fixed order at a point in the underlying real cone). This yields exact interpolating functions for $f(z)$ and $g(z, z_0)$ that are finitely determined within precise subsets of Schur and Herglotz–Nevanlinna functions, respectively. Along the rays emanating from the vertex of the cone $\Gamma$, these interpolants reduce to finite sums of exponentials with positive coefficients and finite sums of simple fractions with positive coefficients, respectively. This finite determinateness is assured by the data encoded in the Hankel kernel. 
More precisely, in a finite-dimensional Hilbert space setting, the interpolatory entire functions for $f(z)$ take the form
$$ F(z) = \langle \exp (- z_1 A_1 - \ldots - z_d A_d) v, v \rangle,$$
where $A = (A_1, \dots, A_d)$ is a $d$-tuple of \textit{non-commuting} real symmetric matrices whose joint numerical range is contained in the convex cone $\Gamma$. Similarly, for the Stieltjes–Fantappi\`e transform, the interpolatory rational functions take the form
$$ G(z, z_0) = \langle (z_0 + z_1 A_1 + \ldots z_d A_d)^{-1} v, v \rangle.$$
This method is developed in the proofs of Theorem~\ref{th:Lap conv} and Theorem~\ref{Thm:S-F-approximation}. The structure of these algebraic operations on linear pencils of non-commuting symmetric matrices is illuminated by the perspective of Weyl's functional calculus \cite{Anderson} and its Radon transform variant, the Wigner distribution. The latter object has compact semi-algebraic support; via its Laplace and Stieltjes–Fantappi\`e transforms, it elegantly encodes the qualitative features of the entire functions $F(z)$ and the rational functions $G(z, z_0)$, respectively \cite{Cohen, SW}. The preliminaries below contain the precise definitions.

The structure of our simple interpolating functions suggests an underlying connection to the Radon transform (see Definition~\ref{Def:Radon}). From a function-theoretic viewpoint, we investigate below certain configurations of this transform, allowing us to invert the Stieltjes–Fantappié transform using a Radon transform inversion from partially observed directions (see Theorem~\ref{th:inversion fan}).

The transition from commuting real coordinates, viewed as multipliers on certain function spaces, to a system of non-commuting symmetric matrices turns out to be the key technical ingredient in our interpolation scheme. The price to pay is a departure from the original complete monotonicity, while remaining within the larger Schur or Herglotz–Nevanlinna classes. To be precise, we prove below that the power series in $x \in \mathbb{R}^d$ associated to a tuple $A = (A_1, A_2, \ldots, A_d)$ of positive semi-definite matrices,
$$
L(x) = \exp(-x_1 A_1 - \ldots - x_d A_d) = \sum_\alpha C_\alpha x^\alpha,
$$
is completely monotone in the operator sense—that is, $(-1)^{|\alpha|} C_\alpha \geq 0$ for all $\alpha \in \mathbb{N}_0^d$—if and only if the matrices $A_j$ mutually commute. By means of this result we provide a sufficient condition for our interpolants to be completely monotone. See Corollary~\ref{Cor:comm-LT} and Corollary~\ref{Cor:comm-FT}.

To complete the circle and rescind the non-commutative relaxation resulting from our scheme, we recall below that the sharp question of
interpolating completely monotone functions by finite sums of exponentials with positive coefficients is equivalent to finding a commuting dilation of the $d$-tuple of matrices $A$ \cite{Dilcub1,Dilcub2}.
While the construction of such a commuting dilation of the specific tuple of symmetric matrices $A$ is decidable in the logical setting of real closed field theory, the computational difficulties
are comparable to any other multivariate cubature method. In summary, we replace the classical Radon X-ray transform with a non-commutative X-ray transform, aiming at bypassing the non-constructive existence ingredients in the theory of multivariate cubature.

The following {\it generic example} illustrates the key features of our proposed interpolation scheme. Start with a completely monotone function $f(z)$ defined in the cone
$\Gamma \subset {\mathbb R}^d$ and associated to a positive measure $\mu$ supported by the dual cone $\Gamma^\ast$. Fix a point $\lambda \in {\rm int} \ \Gamma$ and a positive integer $n$.
The positive semi-definiteness of Hankel kernel  $[(-1)^{|\alpha + \beta|} \partial_z^{\alpha + \beta} f|_{z=\lambda}]_{\alpha, \beta}$ carries a Gram matrix type factorization $\langle A^\alpha u, A^\beta u\rangle, \ \ |\alpha|, |\beta| \leq n,$ where $A = (A_1, \ldots, A_d)$ is a $d$-tuple of symmetric linear operators on a real Hilbert space of dimension less than or equal to
$\binom{n+d}{d}$. In general, the matrices $A_j$ do not commute; nonetheless, the monomial $A^\alpha$ above can be evaluated in any chosen order to determine its action on the cyclic vector $u=e^{-\lambda \cdot x/2}$. Then, we find a Wigner distribution $\mathcal{W}_n$ with the property that the Taylor coefficients of the entire function
$$ f_n(z) = (\mathcal{W}_n(x), \ e^{-z\cdot x}) = \langle \exp ( - A\cdot z) u, u \rangle$$
match those of $f(z)$ up to order $2n+1$ at $z=\lambda$.
In general, the function $f_n(x)$ is only directionally completely monotonic; that is, $t \mapsto f_n(\lambda + t p), \ \ t \geq 0$ is completely monotonic for every $p \in \Gamma$. In particular,
$$ f_n(\lambda + t p) = \sum_{k=1}^{\binom{n+d}{d}} c_k e^{-t \gamma_k},$$
where the coefficients $c_k$ and the exponents $\gamma_k$ are non-negative and they depend on the direction $p$. The geometric locus of the interpolation nodes $\gamma_k$ is encoded in the determinantal variety $\det [z_0 + z\cdot A] = 0$, which is intrinsically related to the joint numerical range of $A$.
Moreover, $\mathcal{W}_n$ converges in the sense of distribution to the measure $\mu$ and consequently, $f_n(z)$ converges to $f(z)$ in the space of holomorphic functions defined on the open tube $\Omega$. The analysis of the Stieltjes-Fantappi{\`e} transform is very similar, with a Wigner distribution interpreted as an analytic functional carried by the dual convex cone.

 The contents of this paper are organized as follows. The Preliminaries section introduces the Laplace and Stieltjes–Fantappi\`e transforms following the approach of Henkin and Shananin \cite{HS}; it also includes an introduction to the Wigner distribution associated with several non-commuting matrices, alongside a recap of the Fantappi\`e transform as a central tool in the analytic duality of convex domains in several complex variables. The convex cones of holomorphic functions on tube domains, arising from their complete monotonicity on the real part, are briefly described in Section~\ref{sec:convexity}, which also includes a discussion of their invariance under automorphisms of the underlying real cone. Section~\ref{sec:inversion} is devoted to the bijective correspondence between the Laplace and Stieltjes–Fantappi\`e transforms of positive measures. The inversions of the Borel and Stieltjes–Fantappi\`e transforms are presented in full detail, complementing the classical Widder sampling inversion formula. Section~\ref{Section:Finite-Deter} exploits the positivity structure arising from the finite sampling (with multiplicities) of a completely monotone function; here, we isolate the positive semi-definite Hankel kernel encoding these data and exploit its Hilbert space factorization to detect finitely determined completely monotone functions. Section~\ref{Sec:NC} contains the non-commutative operator calculus approach to the original interpolation problem, complemented by a qualitative analysis of the approximation process and a brief discussion of the convergence rate. Section~\ref{Sec:Matrix-CM} is devoted to the analysis of the complete monotonicity of matrix-valued functions. Section~\ref{Sec:Comm-dilation} provides a brief overview of the commutative dilation approach to constructing Gaussian-type cubature formulae in several variables. Finally, in Section~\ref{Sec:Exp}, we elaborate on two relevant examples: a pair of non-commuting $2 \times 2$ matrices and their associated Wigner distribution, and a low-order illustration of the interpolation algorithm on the three-dimensional Lorentz cone.

	\section{Preliminaries}

    \noindent \textbf{Notations:} The following notations are adopted for the rest of this article.
\begin{itemize}
\item[•] $\alpha = (\alpha_1, \dots, \alpha_d) \in \mathbb{N}_0^d \text{ and } |\alpha| = \sum_{j=1}^d \alpha_j$ and $\alpha ! = \prod_{j=1}^d \alpha_j !$.
\item[•] For $x=(x_1, \dots, x_d), \ x^\alpha = \prod_{j=1}^d x_j^{\alpha_j}$ and $\partial^\alpha_x f = \frac{\partial^{|\alpha|}f}{\partial x^\alpha}$.
\item[•] $\mathbb{R}^d_+ = (0, +\infty)^d$.
\item[•] $\mathbb H^+_R$ denotes the right half plane in $\C$.
\item[•] $\mathbb H^+_U$ denotes the upper half plane in $\C$.
\item[•] $\Gamma$ denotes an acute closed convex solid cone in $\mathbb{R}^d$.
\item[•] $\Gamma^*$ denotes the closed dual cone of $\Gamma$.
\item[•] $\mathcal E(\R^d)$ denotes the space of Schwartz class functions on $\R^d$.
\item[•] $\mathcal C^\infty_c(\Omega)$ denotes the space of compactly supported smooth functions on the domain $\Omega\subseteq\R^d$. 
\item[•] The abbreviation $\operatorname{PSD}$ stands for positive semi-definite.
\item[•] $\mathcal{O}(\Omega)$ denote the Frech\'et space of holomorphic function in $\Omega$, a domain $\Omega$ in $\mathbb{C}^d$, endowed with the topology of uniform convergence on compact subsets of $\Omega$. 
\end{itemize}
	\subsection{Completely monotone functions}
	Let $\Gamma$ be an acute closed convex solid cone in $\mathbb{R}^d$ and $\Gamma^*= \{ y\in \mathbb{R}^d : y\cdot x \geq 0 \text{ for } x\in \Gamma\}$ be its closed dual cone. 
	We denote the collection of all complete monotone functions in $\operatorname{int} \Gamma$ by $\operatorname{CM}(\Gamma)$.

\vspace{2mm}
     
    \noindent \textbf{Laplace transform:} We say that a nonnegative measure $\mu$ supported in $\Gamma^*$ is of {\em moderate growth} if $\exp (-p \cdot x) \in L^1(\Gamma^*, \mu)$, for every $p\in \operatorname{int} \Gamma$. The {\em Laplace transform} of a measure of moderate growth $\mu$ is given by
   \begin{align}\label{Eq:LT}
			F(p) \bydef \int_{\Gamma^*} \exp(-p\cdot y)\ d\mu(y), \ \ p \in \operatorname{int} \Gamma.
		\end{align}

       Let us denote the collection of non-negative finite measures supported on $\Gamma^*$ by $\mathcal{M}^+(\Gamma^*)$. Clearly, $\mathcal{M}^+(\Gamma^*)$ is a sub-collection of the measures supported on $\Gamma^*$ of moderate growth. The following theorem by Bernstein, Bochner, and Gilbert completely characterizes the Laplace transforms of these measures in terms of completely monotone functions on $\Gamma$ that are continuous up to the boundary of $\Gamma$.
	\begin{theorem}[Bernstein, Bochner, Gilbert \cite{Bernstein, Bochner}] \label{Thm:BBG}
		Let $\Gamma$ be a closed, acute convex solid cone with $\Gamma^*$ as its dual cone. Then a function $F\in C(\Gamma)$ is completely monotone on $\operatorname{int} \Gamma$ if and only if there exists a $\mu\in \mathcal M^+(\Gamma^*)$ such that 
		\begin{align*}
			F(p)= \int_{\Gamma^*} e^{-p\cdot x} \ d\mu(x), \ \ p \in \operatorname{int} \Gamma.
		\end{align*}
	\end{theorem}
	The continuity of $F$ on $\Gamma$ implies that the total mass $\mu(\Gamma^*)$ is equal to $F(0)$, the value of $F$ at the vertex of $\Gamma$.
Motivated by this theorem, for a closed acute convex cone $\Gamma$, by $\widetilde{\operatorname{CM}}(\Gamma)$, we denote the space $\operatorname{CM}(\Gamma)\cap C(\Gamma)$.

As a corollary of this theorem, Henkin and Shananin showed that {\em a function $F\in\operatorname{CM(\Gamma)}$ has an integral representation 
    \eqref{Eq:LT} if and only if $\mu$ is a measure of moderate growth}. See \cite[Corollary 2.2]{HS}.

\vspace{2mm}
We extend the Laplace transform \eqref{Eq:LT} of a measure with moderate growth, as a holomorphic function on the tube domain $T_\Gamma = \operatorname{int} \Gamma + i \ \mathbb{R}^d$ in $\mathbb{C}^d$ by complexifying the $p$ variable as below:
$$
F(z) = \int_{\Gamma^*} e^{-z\cdot x} d\mu(x), \ \ z\in T_{\Gamma}.
$$
Since the extension is unique, we do not distinguish between $F(p)$ on the real cone and its complexification on the tube domain. 

Note that $F(z)$ is a bounded holomorphic function in $T_{\Gamma}$ whenever $\mu$ is a finite measure. Indeed, 
$$
|F(z)| \leq \int_{\Gamma^*} \exp {(- \operatorname{Re} z \cdot x)}\  d\mu(x) \leq \mu(\Gamma^*)
$$
 as $\operatorname{Re} z \in \operatorname{int} \Gamma$ and $\operatorname{Re} z \cdot x \geq 0$ for $x\in \Gamma^*$.
    
    \vspace{2mm}
    
   \noindent \textbf{Stieltjes-Fantappi\`e transform:} For $\mu \in \mathcal{M}^+(\Gamma^*)$, we define the {\em Stieltjes-Fantappi\`e transform} of $\mu$ by
	\begin{align}\label{Eq:FT}
		\Phi(p, p_0) \bydef \int_{\Gamma^*} \frac{d\mu(x)}{p_0 + p\cdot x}, \ \ \ p\in \Gamma \text{ and } p_0 >0.
	\end{align}
Furthermore, by $\operatorname{CM}_H(\Gamma\times \mathbb{R}_+)$ we denote the space of complete monotone functions in $\operatorname{int} \Gamma \times \mathbb{R}_+$ that are continuous in $\Gamma\times\R_+$ and homogeneous of degree $-1$.

 \vspace{2mm}

\noindent \textbf{Borel transform:}
     The Laplace transform \eqref{Eq:LT} and the Stieltjes-Fantappi\`e transform \eqref{Eq:FT} for measures in $\mathcal M^+(\Gamma^*)$ are related via the following commutative diagram:
     \[
	\begin{tikzcd}
		\mathcal M^+(\Gamma^*) \arrow[rr, "F"] \arrow[dr, "\Phi"'] &     & \widetilde{\operatorname{CM}}(\Gamma) \arrow[dl, "\mathcal B"] \\
		& \operatorname{CM}_H(\Gamma\times \R_+ )&,  
	\end{tikzcd}
	\]
    where $\mathcal B: \widetilde{\operatorname{CM}}(\Gamma)\rightarrow  \operatorname{CM}_H(\Gamma\times \R_+ )$ is the {\em Borel transform} given by
    \bea\label{eq:def borel}
	\mathcal B(F)(p,p_0) \bydef \int_0^\infty F(tp)e^{-tp_0} dt,\ \ \ p\in \Gamma \text{ and } p_0 >0.
	\eea
    By showing that the Borel transform is an isomorphism, Henkin and Shananin \cite[Theorem 2.1$'$]{HS} completely charecterizes the Stieltjes-Fantappi\`e transform of the class $\mathcal M^+(\Gamma^*)$ by $\operatorname{CM}_H(\Gamma\times \mathbb{R}_+)$.

     \begin{theorem} \label{Thm:SF}
		A function $\Phi(p, p_0)$, $p\in \operatorname{int} \Gamma, p_0 > 0$ can be represented in the form \eqref{Eq:FT} with a nonnegative measure supported in $\Gamma^*$ if and only if $\Phi\in \operatorname{CM}_H(\Gamma\times \mathbb{R}_+)$.
	\end{theorem}
    
	In Section \ref{sec:inversion}, we discuss the inversion of the Borel and  Stieltjes-Fantappi\`e transform in more detail.
	
	Note that, $\operatorname{int} \Gamma \times \mathbb{R}_+ \subseteq \mathbb{R}^{d+1}$ is the domain of continuity of the Stieltjes-Fantappi\`e transform $\Phi(p, p_0)$ of the measure $\mu$. 
    Let $T_{\Gamma \times \mathbb{R}_+}$ be the tube domain $(\operatorname{int} \Gamma \times \mathbb{R}_+) + i\ \mathbb{R}^{d+1}$ inside $\mathbb{C}^{d+1}$. A point in $T_{\Gamma \times \mathbb{R}_+}$ is written as $(z, z_0) = (p + i q, p_0 + i q_0) \in \mathbb{C}^d \times \mathbb{C}$, where $p \in \Gamma, p_0>0, q\in \mathbb{R}^d$ and $q_0 \in \mathbb{R}$. 
    
    With this notation we complexify the Stieltjes-Fantappi\`e transform \eqref{Eq:FT} in the following way:
	\begin{align} \label{Eq:complex-FT}
		\Phi(z, z_0) = \int_{\Gamma^*} \frac{\mu(dx)}{z_0 + z\cdot x}, \ \ \ (z, z_0)\in T_{\Gamma \times \mathbb{R}_+}.
	\end{align}
    It is easy to see that $\Phi(z, z_0)$ (as above) is holomorphic in $T_{\Gamma \times \mathbb{R}_+}.$
Since the extension of $\Phi(p, p_0)$ to the tube domain is unique, we keep writing $\Phi$ to denote the Stieltjes-Fantappi\`e transform on the real cone and its complexification.

\begin{definition}
For a domain $\Omega$ in $\mathbb{C}^d$, the {\em Herglotz-Nevanlinna class} consists of holomorphic functions in $\Omega$ having non-negative real parts. 
\end{definition}

Integral representation of Herglotz-Nevanlinna functions in the polydisk can be found in \cite{KP} and for an application see \cite{BP, BP2}. The Stieltjes-Fantappi\`e transform $\Phi(z, z_0)$ of $\mu$ is in the Herglotz-Nevanlinna class of the tube domain $T_{\Gamma \times \mathbb{R}_+}$. Indeed, 
$$
\operatorname{Re} \Phi(z, z_0) = \int_{\Gamma^*} \frac{(\operatorname{Re} z_0 + \operatorname{Re} z \cdot x)}{|z_0 + z\cdot x|^2} d\mu(x) \geq 0,
$$
as $\operatorname{Re} z_0 >0$ and $\operatorname{Re} z \in \operatorname{int} \Gamma$.

\subsection{Symmetric matrix tuples and Wigner distributions} \label{Sub:Wigner}
Let $d \geq 1$ be an integer and let $A = (A_1, A_2, \ldots, A _d)$ be a $d$-tuple of self-adjoint matrices acting on a finite dimensional real Hilbert space $H$.
In general, if the linear operators $A_j$ do not mutually commute, a notion of joint spectrum with all expected functoriality properties does not exist. As a crude substitute, the {\it joint numerical range} is well defined:
$$ W(A) = \{\langle Au, u \rangle \bydef (\langle A_1 u,u\rangle, \ldots, \langle A_d u, u \rangle) \in \mathbb{R}^d : \ \| u \| =1, \ u\in H \}.$$
A classical theorem due to Toeplitz and Hausdorff  asserts that $W(A)$ is convex for $d=2$. The convex hull of $W(A)$ can be defined similarly by replacing the pure states $\langle \cdot, u \rangle u$ by density matrices, and then it is a convex subset of ${\mathbb R}^d$. See \cite{PSW} for details.

A celebrated result of Kippenhahn, recently generalized to any number of dimensions $d$ connects the complex projective hypersurface
$$ \mathcal{V}(A) = \{ (z_0, z_1, \ldots, z_d) \in {\mathbb C}^{d+1}: \ \ \det [ z_0 + z_1 A_1 + \ldots + z_d A_d] = 0\}$$
to the geometry of the joint numerical range. Roughly speaking, the convex hull of the smooth real points in the projective dual variety $\mathcal{V}(A)^\ast$ lying in the affine chart $z_0 =1$
is equal to the convex hull of $W(A)$. We refer to Theorem 1.2 in \cite{PSW} for the exact statement and a complete proof. The complex projective hypersurface $\mathcal{V}(A)$ is known in the operator theory community as the {\it projective spectrum} of the $d$-tuple $A$. A series of non-trivial ramifications of this concept to 
group representations and complex geometry are recorded in the recent monograph \cite{Yang}. 

For a vector $u \in H$ and a function $\phi$ in the Schwartz class ${\mathcal E}({\mathbb R}^d)$, the self-adjoint functional calculus yields a (tempered) distribution ${\mathcal W}_u(A)$ given by
\begin{align}\label{Eq:Wigner}
    ({\mathcal W}_u(A)(\xi), \phi ( x \cdot \xi))  \bydef \langle \phi(x \cdot A)u, u \rangle,
\end{align}
known in quantum optics as {\it Wigner distribution}. The support of  ${\mathcal W}_u(A)$ is compact and it is contained in the convex hull of the numerical range $W(A)$. So, \eqref{Eq:Wigner} remains valid for any smooth function $\phi$ in $\mathbb{R}^d$.
The Fourier transform of $\mathcal{W}_u$
is an essentially bounded function as the operator $\exp (- i \xi \cdot A)$ is unitary for every $\xi \in {\mathbb R}^d$.

This is an instance of Weyl's non-commutative operational calculus, originally defined by Fourier inversion \cite{Anderson}. The resulting operator-valued distribution $T(A)$ is called by Anderson the
{\it joint spectral distribution} of the $d$-tuple $A$ while its support serves as a joint spectrum. Functoriality and other qualitative properties of $T(A)$ are proved in \cite{Anderson}.
In particular, \cite[Lemma~3.8]{Anderson} shows that the order of ${\mathcal W}_u(A)$ is at most $\frac{d}{2}+1$, if $d$ is even, and at most $\frac{d+1}{2}$, if $d$ is odd.  An independent  thorough analysis of the qualitative properties of the Wigner distribution was carried out in \cite{SW}. The recent monograph \cite{Cohen} also has relevant material for Wigner distribution analysis. An example of Wigner distribution has been illustrated from various perspective in Example~\ref{Sub:pair-nc-projections}.

\subsection{Analytic duality in several complex variables}
For a compact set $K\subset\C^d$, let $\mathcal O(K)$ be the space of germs of holomorphic functions defined in neighbourhoods of $K$. 
The space of {\em analytic functionals}, denoted by $\mathcal O'(K)$ is defined as the topological dual of $\mathcal O(K)$, with respect to the inductive limit topology. Equivalently, a linear functional $L : \mathcal O(K)\rightarrow\C$, belongs to $\mathcal O'(K)$ if and only if, for any open neighborhood $U$ of $K$, there exists $C_U>0$ such that
\[
|L(f)|\leq C_U \sup_{z\in U} |f(z)|,\quad \text{for all } f\in \mathcal O(U).
\]
 When $K$ is ($\C$-)convex, a classical theorem due to  Aizenberg and independently Martineau shows that the {\em Fantappi{\`e} transform} on $\mathcal O'(K)$, given by
\beas
\mathfrak{F}\ltt L \rtt(z) \bydef L \ltt\frac{1}{1+\left\langle\cdot,z\right\rangle}\rtt,\quad L \in\mathcal O'(K), z\in K^*,
\eeas
implements a topological duality pairing between ${\mathcal O}'(K)$ and ${\mathcal O}(K^\ast)$. Above
 $\left\langle\cdot,\cdot\right\rangle$ is the non-Hermitian product in $\C^d$ and $K^*=\{z\in\C^d:\left\langle\zeta,z\right\rangle\neq -1,\, \forall \zeta\in K\}$ is the dual complement of $K$. This result has been generalized in the setting of the complex projective space $\C\mathbb P^d$, showing that the Fantappi{\`e}  transform(in its projectivized version) is an isomorphism between the topological dual space of holomorphic line bundles over a $(\C-)$ convex open or compact set $K\subset\C\mathbb P^d$ onto the space of holomorphic line bundles on the dual complement $K^*$. A comprehensive discussion on this topic can be found in \cite{Andersson Passare book}. Also, see \cite{Chatterjee} for further generalization to strongly convex domains.
 
 However, if $K\subsetneq \C^d$ is a compact set, then viewing it as a compact set in $\C\mathbb P^d$ sitting in an affinization, the section of a holomorphic line bundle on $K$ can be viewed as a $k$-homogeneous holomorphic function on $K\subsetneq \C^d$. Moreover, at the level of the space of analytic functionals, the topological dual space of the space of holomorphic line bundles is identified as $\mathcal O'(K)$. Furthermore, the projective dual complement in the sense of \cite[Definition~2.1.1]{Andersson Passare book} is  given by 
 \bea\label{eq:projective dual}
 K^*=\{[z_0:z]\in \C\mathbb P^d: z_0+\langle z,w\rangle\neq 1, \forall w\in K\}
 \eea
can be viewed as an open set in $\C^{d+1}$, and the Fantappi{\`e} transform (projectivized version) constitutes an isomorphism between $\mathcal O'(K)$ and the space of holomorphic functions that are $k$-homogeneous; see \cite[Theorem~3.5.3]{Andersson Passare book}. We elaborate more on this in Subsection~\ref{sub:analytic dual} for the particular situation when the compact set
$K=\operatorname{conv} W(A)$, the convex hull of $W(A)$ for a $d$-tuple real symmetric matrices $A$.
Consequently, the rational Herglotz–Nevanlinna approximants of the Stieltjes-Fantappi{\`e} transforms emerge as the Fantappi\`e transforms of analytic functionals acting on $\mathcal{O}(\operatorname{conv} W(A))$.

The present article makes use of Laplace transforms of Wigner distributions and Fantappi\`e transforms of specific analytic functionals, both canonically derived from Hankel kernel data encoded in a completely monotone function.

	\section{$\operatorname{CM}(\Gamma)$ and $\operatorname{CM}_{H}(\Gamma \times \mathbb{R}_+)$ as convex cones} \label{sec:convexity}
	In this section, we record some properties of the two convex cones of our concern: the Laplace transforms and the Stieltjes-Fantappi\`e transforms. Treating them as subsets of $\mathcal{O}(T_\Gamma)$ and $\mathcal{O}(T_{\Gamma \times \mathbb{R}_+})$ respectively, we mainly focus on their closure under the {\em Fr\'echet topology of uniform convergence on compact subsets of the interior on the respective tube domains}
    and their invariance under the automorphism group of $\Gamma$.
    
	\subsection{Laplace transforms}
	We begin with the observation that the cone $\widetilde {\operatorname{CM}}(\Gamma)$ is not closed under the Fr\'echet topology. Indeed, the Laplace transforms of the finite positive measures $\mu_n$, defined to be the Lebesgue measure restricted to the closed interval $[0, n]$ for each $n$, do not converge in the same cone under the Fr\'echet topology.
    
However, the enlarged cone of the Laplace transforms of measures with moderate growth on $\Gamma^*$, i.e., ${\operatorname{CM}(\Gamma)}$ is closed under the Fr\'echet topology as shown in the next proposition. The comments after Theorem \ref{Thm:BBG} in Preliminaries can be summarized as: {\em $F \in {\operatorname{CM}(\Gamma)}$ if and only if $F \in C^\infty(\operatorname{int} \Gamma)$ and the conditions \eqref{Eq:CM-cond} are satisfied for all $p \in \operatorname{int} \Gamma$, $\xi^{(1)}, \dots, \xi^{(k)} \in \Gamma$} and $k\in \mathbb{N}$.
    
    \begin{proposition}
		The convex cone ${\operatorname{CM}(\Gamma)}$ is closed in $\mathcal{O}(T_\Gamma)$ under the Fr\'echet topology.
	\end{proposition}
	\begin{proof}
		Suppose $\{F_n\}_n$ is a sequence of complexified Laplace transforms of measures with moderate growth on $\Gamma^*$ and $F_n$ converges to $F$ uniformly on compact subsets of $T_\Gamma$. Then $F$ is holomorphic in $T_\Gamma$ and hence, a smooth function when restricted to $\operatorname{int} \Gamma$.
		Moreover, the derivatives of $F_n$ converge to $F$ at every interior point of the tube domain, in particular the interior points of $\Gamma$. Thus, for any $p\in \operatorname{int} \Gamma$ and $\xi^{(1)}, \dots, \xi^{(k)} \in \Gamma$, we have
		$$(-1)^k D_{\xi^{(1)}} D_{\xi^{(2)}} \dots D_{\xi^{(k)}} F_n(p) = \lim_{n\to \infty} (-1)^k D_{\xi^{(1)}} D_{\xi^{(2)}} \dots D_{\xi^{(k)}} F(p) \geq 0,$$ as each $F_n$ satisfies the inequalities \eqref{Eq:CM-cond}. Hence, $F$ belongs to ${\operatorname{CM}(\Gamma)}$.
        \end{proof}
        
  Since the Laplace transform is injective on the space of positive measures of moderate growth, the next simple observation follows.
  \begin{corollary} The extremal rays of the closed cone ${\operatorname{CM}(\Gamma)}$ are the functions $c e^{-z \cdot x}$ with $c \geq 0$ and $x \in \Gamma^\ast$.
   \end{corollary}
    
	\subsection{Stieltjes-Fantappi\`e transforms}
	Let us now consider the convex cone $\operatorname{CM}_H(\Gamma \times \mathbb{R}_+)$
     of Stieltjes-Fantappi\`e transforms of measures in $\mathcal{M}^+(\Gamma^*)$.
	\begin{proposition}
		The convex cone $\operatorname{CM}_H(\Gamma \times \mathbb{R}_+)$ is a closed in $\mathcal{O}(T_{\Gamma \times \mathbb{R}_+})$ under the Fr\'echet topology.
	\end{proposition}
	\begin{proof}
	The proof is similar to the proof of the previous proposition. Note that, the complete monotonicity in $\operatorname{int} \Gamma \times \mathbb{R}_+$ and degree of homogeneity of a sequence of functions is preserved under uniform convergence on compact subsets of $\operatorname{int} \Gamma \times \mathbb{R}_+$. Hence, the proposition is a consequence of Theorem~\ref{Thm:SF}.
	\end{proof}
	
	Similarly, the extremal rays on the convex cone of Stieltjes-Fantappi\`e transforms of positive measures corresponds to point masses.
	
	\subsection{Invariance under the automorphism group}
The automorphism group $G(\Gamma)$ of the cone $\Gamma$ consists of $d \times d$ real invertible matrices mapping $\Gamma$ to itself. It is known that the automorphism group $G(\Gamma^*)$ of the dual cone consists of the adjoints of the elements in $G(\Gamma)$. See \cite[Chapter I]{Faraut-Koranyi}. 

The self-dual case of symmetric convex cones, i.e., $\Gamma^\ast = \Gamma$, is notable. The associated tube domain is then a Riemannian symmetric space
which can be realized as a bounded symmetric domain via a generalized Cayley transform. These are true multivariate analogues of the upper-half plane and the unit disk.
The Laplace and Stieltjes-Fantappi\`e transforms of positive measures naturally enter into the ample and well developed representation theoretic picture. See \cite{Faraut-Koranyi} for historical comments and an extensive bibliography.

Regarding the two convex cones $\operatorname{CM}(\Gamma)$ and $\operatorname{CM}_H(\Gamma\times \R_+)$ as classes of multipliers—having a bounded modulus and a positive real part, respectively—on various Hilbert spaces of analytic functions associated to the corresponding tube domains, it is natural to ask whether they are invariant under the action of $G(\Gamma).$ The affirmative answer follows directly from the observation below.

Let $A$ be in $G(\Gamma)$ and let $F$ be in $\operatorname{CM}(\Gamma)$. Then $F$ is of the form \eqref{Eq:LT} for some measure $\mu$ on $\Gamma^*$ having moderate growth. Then, for $p\in \operatorname{int} \Gamma$,
\begin{align*}
    F\circ A (p) = \int_{\Gamma^*} \exp{(- p \cdot A^*x)} d\mu(x) = \int_{\Gamma^*} \exp{(- p \cdot x)} d\nu(x),
\end{align*}
where $\nu$ is the push-forward of $\mu$ under $A^* \in G(\Gamma^*)$. Hence, $F\circ A \in \operatorname{CM}(\Gamma)$.

Similarly, if $\Phi \in \operatorname{CM}_H(\Gamma\times \R_+)$, then $\Phi$ can be expressed in the form \eqref{Eq:FT} for some finite measure $\mu$ on $\Gamma^*$ and hence, for $p_0>0$ and $p\in \operatorname{int} \Gamma$,
\begin{align*}
    \Phi (Ap, p_0) = \int_{\Gamma^*} \frac{d\mu(x)}{p_0 + p \cdot A^*x} = \int_{\Gamma^*} \frac{d\nu(x)}{p_0 + p \cdot x},
\end{align*}
where the measure $\nu$ is as above. So, for the automorphism, $(A, \operatorname{Id})(p, p_0) = (Ap, p_0)$ of $\Gamma \times (0, \infty)$ we obtained
$\Phi \circ (A, \operatorname{Id}) \in \operatorname{CM}_H(\Gamma\times \R_+)$.

However, the holomorphic automorphism group of a tube domain over a symmetric convex cone is bigger than $G(\Gamma)$, see Chapter X in \cite{Faraut-Koranyi}. We show by means of a simple example that the above  elementary ``real" invariance fails.

\begin{remark}
    The algebra of bounded holomorphic functions $H^\infty(T_\Gamma)$ on $T_\Gamma$ is invariant under the holomorphic automorphisms of $T_\Gamma$. But the collection of complexified Laplace transforms, a proper subclass of $H^\infty(T_\Gamma)$, is not invariant under the holomorphic automorphisms of $T_\Gamma$ as illustrated in the following example with $\Gamma = \mathbb{R}^2_+$. The associated tube domain $T_\Gamma = \mathbb{R}^2_+ + i \mathbb{R}^2 \subseteq \mathbb{C}^2$ can be identified with $\mathbb{H}_R^+ \times \mathbb{H}_R^+$. Consider the automorphism $\Psi$ of $T_\Gamma$ given by
	$$
	\Psi(z_1, z_2) = \left(\frac{1-i + (1+i)z_1}{1+i + (1-i)z_1}, \ z_2 \right).
	$$
	The function
	$$
	F(z_1, z_2) = \exp{-(z_1 + z_2)}, \ \text{ for } (z_1, z_2) \in T_\Gamma,
	$$
	is the complexified Laplace transform of the Dirac mass supported at the point $(1, 1) \in \Gamma^*$.
	Now,
	$$
	F\circ \Psi (z_1, z_2) = \exp{ \left(-\frac{1-i + (1+i)z_1}{1+i + (1-i)z_1} - z_2 \right)} \ \text{ for }(z_1, z_2) \in T_\Gamma.
	$$
	In particular, $(1/2, 1/2) \in \operatorname{int}\Gamma$ and
	$$
	F \circ \Phi(1/2, 1/2) = \exp(-13/10 +  3i/5 )
	$$
	which is not a positive real number. Thus, $F \circ \Phi$ can never arise from the Laplace transform of a positive measure on $\Gamma^*$.
\end{remark}

	\section{Inversion of the Borel and the Stieltjes-Fantappi\`e transform}\label{sec:inversion}
   In \cite[Proposition~4.1]{HS}, Henkin and Shananin provided an inversion formula for the Laplace transform. In a similar spirit, this section derives explicit inversion formulae for both the Borel and Stieltjes-Fantappi\`e transforms. For the Borel transform, we establish an inversion formula over any acute closed convex cone by using the classical inverse Laplace transform. For the Stieltjes-Fantappi\`e transform, we restrict our focus to the positive orthant and obtain its inverse exploiting its interplay with the Radon transform. Notably, this latter inversion can also be interpreted as a jump formula in the context of the Stieltjes-Fantappi\`e transform. 
    \subsection{Borel inversion formula}
    
    We utilize the classical inversion techniques for the Laplace transform, which require the complexification of the elements in ${\operatorname{CM}_H(\Gamma\times \R_+)}$ as illustrated in \eqref{Eq:complex-FT}. 
	
		Let $\Phi \in \operatorname{CM}_H(\Gamma\times \R_+)$. For a fixed $p\in \operatorname{int} \Gamma$, consider the function, $\Phi_p(t):=\Phi(p,t),\,\,\text{for } t\in\R_+$. By the extension given in \eqref{Eq:complex-FT}, $\Phi_p$ has unique holomorphic extension to the right half plane $\mathbb H^+_R$, we denote this extension $\widetilde{\Phi_p}$.
        
 The following theorem provides the explicit inverse of the Borel transform using the extended function $\widetilde{\Phi_p}$. 
	\begin{theorem}\label{th:inv borel}
		The inverse of the Borel transform $\psi:\operatorname{CM}_H(\Gamma\times \R_+)\rightarrow \widetilde{\operatorname{CM}}(\Gamma)$, is given by
		\[
		\psi(\Phi)(p)= \lim\limits_{T\rightarrow\infty}\frac{1}{2\pi i}\int_{\gamma-iT}^{\gamma+iT} \widetilde \Phi_p(z) e^z dz,\,\, \forall p\in \operatorname{int}\Gamma, 
		\]
		where $\gamma$ is any positive real number, and $\Phi \in \operatorname{CM}_H(\Gamma\times \R_+)$ is the Stieltjes-Fantappi\`e transform of $\mu$ with
		\[
		\widetilde{\Phi}_p(z)= \int_{\Gamma^*} \frac{d\mu(x)}{z+p\cdot x},\,\,\text{for } z\in \mathbb H^+_R.
        \]
	\end{theorem}
	\begin{proof}
		Let $p\in \Gamma$, and $F\in\widetilde{\operatorname{CM}}(\Gamma)$. Then,
		\beas
		(\psi\circ\mathcal B)(F)(p)&=&\lim\limits_{T\rightarrow\infty}\frac{1}{2\pi i}\int_{\gamma-iT}^{\gamma+iT} \widetilde{\mathcal{B}F_p}(z) e^z dz\\ 
		&=& \lim\limits_{T\rightarrow\infty}\frac{1}{2\pi i} \int_{\gamma-iT}^{\gamma+iT} \ltt\int_0^\infty F(tp) e^{-tz} dt\rtt e^z dz\\
		&=& \lim\limits_{T\rightarrow\infty}\frac{1}{2\pi i} \int_{\gamma-iT}^{\gamma+iT} \mathcal L(f_p)(z) e^z dz,
		\eeas 
		where $f_p(t)=F(tp),\,\text{for } t\in \R_+$, and $\mathcal L$ is the one-sided Laplace transform of $f_p$ given by
		\[
		\mathcal L(f_p)(z)=\int_0^\infty f_p(t) e^{-tz} dt,\,\, \text{for }z\in \mathbb H^+_R.
		\]
		Following \cite[Chapter~VI, Theorem~5a]{Widder}, and applying the inversion formula for the Laplace transform, we get that 
		\[
		f_p(t)=\frac{1}{2\pi i} \lim\limits_{T\rightarrow\infty}\int_{\gamma-iT}^{\gamma+iT} \mathcal L(f_p)(z) e^{tz} dz, \,\, \text{for }t\in \R_+,
		\]
		where $\gamma>0$. 
		This implies,
		\[
		(\psi\circ\mathcal B)(F)(p)=f_p(1)=F(p).
		\]
		This completes the proof.
	\end{proof}

	\subsection{Stieltjes-Fantappi\`e inversion formula for $\mathcal{C}^\infty_c(\R^d_+)$ }
    Here, we derive an inversion formula for the Stieltjes-Fantappi{\`e} transform on the positive orthant $\R^d_+$ via the Radon transform inversion. Let $f\in \mathcal{C}^\infty_c(\R^d_+)$. The Stieltjes-Fanatppi{\`e} transform of $f$ is given by
	\[
	\Phi(p,p_0)=\int_{\R^d_+}\frac{f(t)}{p_0+p\cdot t} d\lambda_{leb}(t),
	\]
	where $\lambda_{leb}$ is the Lebesgue measure of $\R^d$, and $(p,p_0)\in \R^d_+\times \R_+.$
	Since $f$ is compactly supported, we can extend $\Phi(f)$ continuously on $\R^d\times(\C\setminus\{\operatorname{Im}(z)=0\})$ by taking
	\bea\label{eq:extn fan}
	\widetilde{\Phi}(p,z)=\int_{\R^d_+}\frac{f(t)}{z+p\cdot t} d\lambda_{leb}(t),\quad \text{for } (p,z)\in \R^d\times(\C\setminus\{\operatorname{Im}(z)=0\}).
	\eea
	Furthermore, by Morera's theorem, for any fixed $p\in\R^d$, $\widetilde{\Phi}(p,\cdot)\in\mathcal{O}(\C\setminus\{\operatorname{Im}(z)=0\}).$ 
	Next, we consider the jump function associated with $\widetilde{\Phi}$, given as,
	\bea\label{eq:jump fns}
	J_{\widetilde{\Phi}}(p, \xi)& \bydef &\frac{1}{\pi}\lim\limits_{\eta\rightarrow 0^+}\ltt{{\partial^{d-1}_{\xi}\widetilde{\Phi}}(p,\xi-i\eta)-{\partial^{d-1}_\xi\widetilde{\Phi}}(p,\xi+i\eta)}\rtt,
	\eea
	where $(p,\xi)\in \R^{d+1}$. Although currently it is unclear whether the above limit exists or not, later we will show that it does exist. Continuing with the above notations we have the following inversion result.
	\begin{theorem}\label{th:inversion fan}
		Let $f\in \mathcal C^\infty_c(\R^d_+)$. Then for $d$ even, 
		\bea\label{eq:inversion even}
		f(x)= \frac{-(-1)^{d/2}}{(2\pi)^d} {\int_{\mathbb S^{d-1}}{\int_\R}}\frac{ J_{\widetilde{\Phi}}(p,-\xi)}{p\cdot x-\xi} d\xi d\sigma_{d-1}(p),
		\eea
		and for $d$ odd,
		\bea\label{eq:inversion odd}
		f(x)=\frac{(-1)^{(d-1)/2}}{2(2\pi)^{d-1}} {\int_{\mathbb S^{d-1}}}  J_{\widetilde{\Phi}}(p,-p\cdot x) d\sigma_{d-1}(p),
		\eea
		where $\sigma_{d-1}$ is the surface area measure of $\mathbb S^{d-1}$.
	\end{theorem}
	Before moving on to the proof of this theorem, we introduce the {\em Radon transform}. 
    \begin{definition}\label{Def:Radon}
        The Radon transform of 
	an integrable function $g$ on $\R^d$ is defined as
	\beas
	\mathcal R(g)(p,s) \bydef \int_{H(p,s)} g(t) d\sigma_H(t),\, \text{ for a.e. } s\in\R,
	\eeas
	where $p\in \mathbb S^{d-1}$, $H(p,s)=\{t\in\R^d: p\cdot t=s\}$, and $\sigma_H$ is the surface area measure of $H(p,s)$.
    \end{definition}
    
     The following lemma shows that an integral on $\R^d$ can be viewed as an integral of the Radon transform on the real line.
	\begin{lemma}\label{le:COV}For any $g\in\mathcal C^\infty_c(\R^d)$,
		\beas\label{eq:cov cone}\nonumber
		\int_{\R^d} g(t) d\lambda_{leb}(t)= \int_\R \int_{H(p,s)} g(u) d\sigma_H(u) ds
		=\label{eq:cov cone} \int_\R \mathcal R (g)(p,s) ds.
		\eeas
	\end{lemma}
	\begin{proof}
		For a fixed $p\in\mathbb S^{d-1}$, let $\{p,q_2,\cdots,q_{d}\}$ is an orthonormal basis of $\R^d$. Then, consider the change of variables 
        \[
       (s,u)= (s,u_2,u_3,\cdots,u_{d})=( p.t,  q_2.t, q_3.t,\cdots, q_d.t),       
        \]
       where $s\in\R$ and $u=(u_2,u_3,\cdots,u_{d})\in\R^{d-1}$.
       Hence $$t= sp+u_2q_2+\cdots+u_dq_d.$$
       Also, for a fixed $s\in\R$, $ p.t=s$. Consequently, the map 
       \[
       u=(u_2,u_3,\cdots,u_{d})\rightarrow sp+u_2q_2+\cdots+u_dq_d, 
       \]
       parametrizes the hyperplane $H(p,s)$. Furthermore, by direct computations of the Jacobian, it follows that
       \beas
       dt_1 dt_2\cdots dt_n = ds \ du
       = ds \ d\sigma_H(u).
       \eeas
       Therefore, for any $g\in\mathcal C^{\infty}_c(\R^d)$,
       \beas
       \int_{\R^d} g(t) d\lambda_{leb}(t)&=&\int_\R \int_{\R^{d-1}} g(sp+u_2q_2+\cdots+u_dq_d) \ du \ ds
       \\&=&\int_\R \int_{H(p,s)} g(u)  \ d\sigma_H(u) \ ds,
       \eeas
       which completes the proof.
	\end{proof}
	\begin{proof}[\textbf{Proof of Theorem~\ref{th:inversion fan}}]
		By straightforward computations on  \eqref{eq:jump fns}, we get that
		\bea
		\label{eq:diff jump}J_{\widetilde{\Phi}}(p,-\xi)&=&\frac{1}{\pi}\lim\limits_{\eta\rightarrow 0^+} \int_{\R^d_+} f(t) W_{\xi,p,\eta}(t) \ d\lambda_{leb}(t),
		\eea
		where 
		\beas\label{eq:jump kernel}
		W_{\xi,p,\eta}(t)=\frac{(-1)^{d-1}}{((-\xi+p\cdot t)-i\eta)^{d}}-\frac{(-1)^{d-1}}{((-\xi+p\cdot t)+i\eta)^{d}}.
		\eeas
	Furthermore,
	\beas
	W_{\xi,p,\eta}(t)&=& {\partial^{d-1}_\xi} \ltt\frac{1}{(-\xi+p\cdot t-i\eta)}-\frac{1}{(-\xi+p\cdot t+i\eta)}\rtt\\
	&=& {\partial^{d-1}_\xi} P(-\xi+p\cdot t+i\eta),
	\eeas
	where $P: \mathbb{H}_U^+\rightarrow \R$ is the Poisson kernel of the upper half plane given as
	$$P(s+i\eta)=\frac{1}{\pi} \frac{\eta}{\eta^2+s^2},\,\,(s+i\eta)\in \mathbb H^+_U.$$
	
	As $\eta\rightarrow 0^+$, $P(s+i\eta)\rightarrow\delta_0(s)$, in the sense of distribution, where $\delta_0(s)$ is the Dirac delta distribution, i.e., for any test function $\phi$ on $\R$,
	\[
	\lim_{\eta\rightarrow0^+} \int_{\R} \phi(s) P(s+i\eta) ds=\phi(0).
	\]
    See \cite[Theorem~2.1]{Stein-Weiss} for a proof. 
	Let $p\in \mathbb S^{d-1}$.
	By \eqref{eq:diff jump}, and Lemma~\ref{le:COV}, we get that 
	\beas
	J_{\widetilde{\Phi}}(p,-\xi)
	&=&\lim_{\eta\rightarrow 0^+}\partial^{d-1}_\xi\int_{\R^d_+} f(t) P(-\xi+p\cdot t+i\eta) d\lambda_{leb}(t)\\
	&=&\lim_{\eta\rightarrow 0^+}\partial^{d-1}_\xi\int_\R P(-\xi+s+i\eta) \mathcal R(f)(p,s) ds\\
	&=& \partial^{d-1}_\xi \lim_{\eta\rightarrow 0^+} \int_\R P(-\xi+s+i\eta) \mathcal R(f)(p,s) ds\\
	&=&\partial^{d-1}_\xi \mathcal R(f)(p,\xi). 
	\eeas
	For even $d$, by 
	\cite[Section ~2.3]{Palamadov-book},
	\[
	(2\pi)^d f(x)=-(-1)^{d/2} \int_{\mathbb S^{d-1}}\int_\R \frac{{\partial^{d-1}_\xi} \mathcal R(f)(p,\xi)}{p\cdot x-\xi} d\xi d\sigma_{d-1}(p),
	\]
	where $\sigma_{d-1}$ is the surface area measure of $\mathbb S^{d-1}$.
	Hence, by \eqref{eq:diff jump}, we get that
	\[
	f(x)= \frac{-(-1)^{d/2}}{(2\pi)^d}\int_{\mathbb S^{d-1}}{\int_\R}\frac{ J_{\widetilde{\Phi}}(p,-\xi)}{p\cdot x-\xi} d\xi d\sigma_{n-1}(p).
	\]
	For odd $d$,  by
	\cite[Corollary~2.6]{Palamadov-book},
	\[
	f(x)=\frac{(-1)^{(d-1)/2}}{2(2\pi)^{d-1}} \int_{\mathbb S^{d-1}}\partial^{d-1}_\xi \mathcal R(f)(p,p\cdot x) d\sigma_{d-1}(p).
	\]
	Hence, by \eqref{eq:diff jump}, we get that
	\[
	f(x)=\frac{(-1)^{(d-1)/2}}{2(2\pi)^{d-1}} \int_{\mathbb S^{d-1}} J_{\widetilde{\Phi}}(p,-p\cdot x) d\sigma_{d-1}(p).
	\]
\end{proof}

\subsection{Recovering the missing angles for the Radon transform}
Theorem~\ref{th:inversion fan} establishes that the inverse Stieltjes-Fantappi{\`e} transform generally demands data for $\widetilde{\Phi}$ on the whole tube $\mathbb S^{d-1}\times \mathbb{R}$. However, we can overcome this global requirement by utilizing the Palamodov-Denisjuk technique \cite{PD88}, which successfully recovers the Radon transform from the incomplete initial data on $\mathbb S^{d-1}_+\times \R_+$, where $\mathbb S^{d-1}_+=\mathbb S^{d-1}\cap\R^d_+$.

\begin{theorem}
    Let $f\in\mathcal C^\infty_c(\R^d_+)$. The full Radon transform $\mathcal R(f)(\omega,s)$ on the entire domain can be recovered from the incomplete initial data restricted to $\mathbb S^{d-1}_+ \times (0, \infty)$.
\end{theorem}
\begin{proof}
    Note that, for any $(\omega,s)\in\mathbb S^{d-1}_+\times\R_+$, 
\[
\mathcal R(f)(-\omega,-s)=\mathcal R(f)(\omega,s),
\]
and as $f$ is compactly supported on $\R^d_+$,
\[
 \mathcal R(f)(\omega,-s)=  \mathcal R(f)(-\omega,s)=0.
\]
These symmetries demonstrate that knowing the data of $\mathcal R(f)$ on $\mathbb S^{d-1}_+\times \R_+$ immediately recovers the full lines of data on $\mathbb S^{d-1}_+\times \R$ and the opposite orthant $\mathbb S^{d-1}_-\times\R$, where  $\mathbb S^{d-1}_-=-\mathbb S^{d-1}_+ $.

When $d=1$, i,e., $\mathbb S^0=\{\pm 1\}$, using the above argument, from the initial data of $\{+1\}\times (0,\infty)$, we can immediately recover $\mathcal R(f)$ for the entire domain. 
Thus, for the rest of the proof, we consider $d\geq 2$.   
By the argument above, the problem has been reduced to recovering the data on the missing quadrants. Denoting $\mathcal F_d$ as the $d$-dimensional Fourier transform on $\R^d$ and $\mathcal F_1$ as the 1-dimensional Fourier transform with respect to the affine parameter, the Fourier Slice Theorem yields:
\[
\mathcal F_1\ltt \mathcal R(f)(\omega,\cdot)\rtt(t)=\mathcal F_d(f)(t \omega),\quad t\in \R.
\]
For $\Lambda>0$, the {\em filtered Fourier transform} of $f$ is defined as
\[
\mathcal F_{d,\Lambda}(f)(\xi)\bydef H(\Lambda-|\xi|)\mathcal F_d(f)(\xi), \quad \xi\in\R^d,
\]
where $H$ is the Heaviside step function. Consequently, the {\em filtered Radon transform} is defined via the inverse Fourier transform:
\[
\mathcal R_{d,\Lambda} (f)(\omega,s) \bydef \mathcal F_1^{-1}\ltt\mathcal F_{d,\Lambda}(f)(t\omega)\rtt(s),\quad (\omega,s)\in \mathbb S^{d-1}\times \R.
\]

For $\lambda>0$, let us consider the following region
\beas
K_\lambda
&=& \{\xi=(\xi_1,\xi')\in\mathbb S^{d-1}: |\xi_1|\geq \tan(\lambda) |\xi'|\}.
\eeas 

Then, by \cite[Theorem~2]{PD88}, for a fixed $s\in\R$, we can recover the filtered Radon transform $\mathcal R_\lambda(f)(\omega,s)$ for any missing angle $\omega=(\omega_1, \omega')\in \mathbb S^{d-1}\setminus K_\lambda$ via the following integral equation.
\[
\mathcal R_{d,\Lambda}(f)(\omega,s)=\int_{\R} d\alpha \int_{|\tau|\geq \rho} \frac{d\tau}{|\tau-\omega_1|} E_N(\tau,\beta,\gamma) \mathcal R(f)(v,s-R\alpha),
\]
where $\rho = \tan(\lambda)|\omega'|$, $v=(\tau,\omega')$, and the kernel $E_N$ is given by:
\bea\label{eq:P-D transform}
\nonumber E_N(\tau,\beta,\gamma) = \frac{\exp(N\gamma)}{2\pi^2} \left[ \frac{\sin(N(\tau-\beta)-\Theta_-)}{\tau-\beta} - \frac{\sin(N(\tau+\beta)-\Theta_+)}{\tau+\beta} \right]\\  
+\frac{1}{2\pi^2} \left[ \frac{\sin(N(\tau-\beta))}{\tau-\beta} - \frac{\sin(N(\tau+\beta))}{\tau+\beta} \right],
\eea
with parameters defined as $\beta = \sqrt{1-|\omega'|^2}$, $\gamma = \sqrt{\omega_1^2 - \tau^2}$, $\Theta_\pm = \arg(\pm \beta + i\gamma)$, and $N=\Lambda R$, where $R$ is a positive number such that $\operatorname{supp} f $ is contained in the $R$ radius ball centered at the origin.
However, the region $K_\lambda$ may not be a subset of $\mathbb S^{d-1}_+ \cup \mathbb S_{-}^{d-1}$, for any choice of $\lambda$. Although by taking $\lambda=\arctan\frac{1}{\sqrt{d-1}}$, and applying a rotation that takes the point $(1,0,\cdots,0)$ to $\ltt\frac{1}{\sqrt d},\frac{1}{\sqrt d},\cdots,\frac{1}{\sqrt d}\rtt$, we can shift $K_\lambda$ into $\mathbb S^{d-1}_+ \cup \mathbb S^{d-1}_-$, and can apply the following algorithm to recover the missing angles for $\mathcal R(f)$.
    \begin{itemize}
        \item [(i)] Begin with a compactly supported smooth function $f\in \mathcal C^\infty_c(\R^d_+)$.
        \item[(ii)] Consider the rotated function $f_{new}=f\circ A$, where $A$ is a unitary matrix that maps the vector $\ltt\frac{1}{\sqrt d},\frac{1}{\sqrt d},\cdots,\frac{1}{\sqrt d}\rtt$ to $(1,0,\cdots,0)$. Note that $A(K_{\lambda})\subseteq \mathbb S_+^{d-1}\cup \mathbb S^{d-1}_{-}$.
        \item[(iii)] By the rotational symmetry of the Radon transform, the Radon transform of $f_{\text{new}}$ is related to the Radon transform of $f$ via:
    \[
    \mathcal R(f_{\text{new}})(\omega,s) = \mathcal R(f)(A\omega,s),\quad (\omega,s)\in \mathbb S^{d-1}\times \R.
    \]
    Consequently, our known initial data on the orthant provides the exact values of $\mathcal R(f_{\text{new}})(\omega,s)$ for all $(\omega,s) \in K_\lambda \times \R$.
        \item[(iv)]Consider the filtered Radon transform $\mathcal R_{d,\Lambda}(f_{new})$, and by applying \eqref{eq:P-D transform}, we can recover $\mathcal R_{d,\Lambda}(f_{new})$ globally from the values in $K_{\lambda}\times \R$.
        \item[(v)] Then we can  recover $\mathcal F_{d,\Lambda}(s\omega)$, using the relation,
        \[
        \mathcal F_1\ltt\mathcal R_{d,\Lambda}(f)(\omega,\cdot)\rtt(s)= \mathcal F_{d,\Lambda}(f_{new})(s\omega),\quad s\in\R.
        \]
        \item[(vi)] Now, the filtered Fourier transform allows us to recover $\mathcal F_d(f_{new})(s\omega)$, for all $(s,\omega)\in \R\times \mathbb S^{d-1}$.   
        \item[(vii)] By the Fourier slice theorem, we recover $\mathcal R(f_{new})(\omega,s)$, for all $(\omega,s)\in\mathbb S^{d-1}\times \R$. Finally, by applying the inverse of $A$, we can recover $\mathcal R(f)(\omega,s)$, for all $(\omega,s)\in\mathbb S^{d-1}\times \R$.
    \end{itemize}

\end{proof}

\section{Finite determinateness} \label{Section:Finite-Deter}
In this section, we explore an array of finite determinateness phenomena of completely monotone functions on a cone $\Gamma$ via their finite jets (i.e., the Taylor coefficients up to certain given orders at finitely many points in $\operatorname{int} \Gamma$).
For instance, the Laplace (respectively, Stieltjes-Fantappi\'e) transform of a finitely supported positive measure on $\Gamma^*$ yields a finite positive linear combination of the exponentials $\exp(-p \cdot x)$ (respectively, a finite positive linear combination of simple fractions $1/(p_0 + p \cdot x)$), where $x \in \operatorname{supp}(\mu)$. And, of course these finite sums are uniquely determined finitely many jets. An early occurrence of Hankel kernels built on jets of completely monotone functions aimed at detecting their finiteness appeared in Karlin's monograph \cite[Section 7 of Chapter 2]{Karlin}.

As ensured by the following abstract result of Richter and Tchakalov \cite[Theorem 1.24]{Schmudgen-book}, the classes of completely monotone functions under consideration contain an abundance of these finitely determined functions matching prescribed jet data.
\begin{theorem}[Richter-Tchakalov] \label{Thm:RT}
		Suppose that ($\mathcal{Y}, \mu$) is a measure space and $V$ is a finite-dimensional linear subspace of $C_\mathbb{R}(\mathcal{Y}) \cap L^1_\mathbb{R}(\mathcal{Y}, \mu)$. Let $L^\mu$ be the linear functional on $V$ defined by $L^\mu(f) =\int f d\mu$, for $f\in V$. Then there is a positive $k$-atomic measure $\nu = \sum_{j=1}^k m_j \delta_{x^{(j)}}$, where $k\leq \operatorname{dim} V$, such that
		\begin{align*}
			\int fd\mu = \int f d\nu = \sum_{j=1}^k m_j f(x^{(j)}), \ \text{ for } f\in V.
		\end{align*}
	\end{theorem}

As a consequence, relying upon a non-constructive proof, we derive the semi-algebraic nature of the set of Taylor coefficients of the Laplace transforms up to a fixed order. Then we will exploit the Hankel kernels associated with jets of a complete monotone function to analyze finite determinateness.

	\subsection{Semi-algebraic nature of the coefficient set} 
	Let us fix a point $\lambda$ in $\operatorname{int} \Gamma$ and $N \in \mathbb{N}_0$. Let $\mu$ be a measure on $\Gamma^*$ with moderate growth. Then complexified Laplace transform $F_\mu(z)$ has a power series representation
    $$
	F_\mu(z) = \sum_{\alpha \in \mathbb{N}_0^d} c_{\alpha } (F_\mu) \ (z-\lambda)^\alpha
	$$
    in a neighbourhood of the point $\lambda$ in the tube domain $T_{\Gamma}$, where 
    $$ c_\alpha(F_\mu) = \frac{1}{\alpha !} \partial^\alpha_z F_\mu|_{z=\lambda} = (-1)^{|\alpha|} \int_{\Gamma^*} \lambda^\alpha \exp(-\lambda \cdot x) \ d\mu(x).$$
	Note that, $c_\alpha(F_\mu) \in \mathbb{R}$ as $\lambda \in \operatorname{int} \Gamma$.
	We consider the coefficient set of order $N$ as follows:
	$$
	\mathcal{C}_N = \left \lbrace  (c_{\alpha} (F_\mu))_{|\alpha | \leq N} \in \mathbb{R}^{d_N} : \mu \text{ is a measure of moderate growth on $\Gamma^*$} \right \rbrace,
	$$
	where $d_N = \binom{N+d}{d}$. Clearly, $\mathcal{C}_N$ is a closed convex subset of $\mathbb{R}^{d_N}$. 

    A subset of $\mathbb{R}^d$ is said to be (real) {\em semi-algebraic} if it is a finite union of sets defined by polynomial (with real coefficients) equalities and inequalities. See \cite[Chapter 2]{BCR} for more details.
	\begin{proposition} \label{coefficient}
	Suppose the closed convex cone $\Gamma^*$ is defined by a finite system of polynomial inequalities. Then $\mathcal{C}_N$ is a semi-algebraic set.
   
	\end{proposition}
	\begin{proof}
		Let $(c_{\alpha})_{|\alpha | \leq N} = (c_{\alpha}(F_\mu))_{|\alpha| \leq N}  \in \mathcal{C}_N$ for some measure $\mu$ on $\Gamma^*$ having moderate growth.
         In view of Theorem \ref{Thm:RT}, there exists a positive quadrature formula with weights $\{w_k: 1\leq m\}$ and nodes $\{ x^{(k)} \in \Gamma^*: 1\leq m\}$ for some $m \leq d_N$ such that 
		\begin{align}
			\int_{\Gamma^*} x^\alpha \exp(- \lambda \cdot x)\ d\mu(x) = \sum_{k=1}^{m} w_k \ (x^{(k)})^\alpha \exp(-\lambda \cdot x^{(k)}),
		\end{align}
		for $0 \leq |\alpha| \leq N$. Set, $\nu = \sum_{k=1}^{m} w_k\ \delta_{x^{(k)}}$, where $\delta_{x^{(k)}}$ denotes the Dirac mass at the point $x^{(k)}$. Therefore, $c_{\alpha} = c_{\alpha}(F_\nu)$.
		Thus, we can describe $\mathcal{C}_N$ as 
		$$
		\left \lbrace (c_{\alpha}(F_\eta) )_{ |\alpha| \leq N} \in \mathbb{R}^{d_N} : \eta = \sum_{k=1}^{d_N} w_k \delta_{x^{(k)}} \in \ \mathcal{M}^+(\Gamma^*), \ \ w_k \geq 0 \text{ and } x^{(k)} \in \Gamma^* \right \rbrace.
		$$

By absorbing the exponentials  $\exp(-\lambda \cdot x^{(k)})$ into the non-negative coefficients, an element $(c_{\alpha} )_{|\alpha| \leq N}$ belongs to  the coefficient set $\mathcal{C}_N$ if and only if there exists a finite, positive measure $\sigma$ supported by at most $d_N$ points of $\Gamma^\ast$ such that
		$$ 
        c_\alpha = \int_\Gamma x^\alpha d\sigma(x), \ \ |\alpha| \leq N.
        $$

		Since the convex cone $\Gamma^*$ is defined by a finite system of polynomial inequalities, it is semi-algebraic. So, in virtue of  Tarski's  elimination of quantifiers principle \cite{BCR} applied to the last description of $\mathcal{C}_N$, we conclude that $\mathcal{C}_N$ is a semi-algebraic set in $\mathbb{R}^{d_N}$.
	\end{proof}

	As a byproduct of the preceding proofs, the coefficient set (or equivalently, the jet at a prescribed point) of completely monotone functions defined on $\Gamma$ coincides with the power moments of finitely supported positive measures on $\Gamma^*$. An immediate observation is the positivity condition involving the associated Hankel kernel:
	$$ \sum_{|\alpha +\beta| \leq N} c_{\alpha + \beta} s_\alpha s_\beta =  \int_\Gamma |\sum_\gamma s_\gamma x^\gamma|^2 d\sigma(x) \geq 0, \ \ s_\alpha, s_\beta \in {\mathbb R}.$$
	
	\begin{corollary} In the condition of Proposition~\ref{coefficient}, assume the dimension $d$ of the underlying space is bigger than $1$ and choose $N \geq 6.$ Then, there exists a
	coefficient system $(c_\gamma)_{|\gamma| \leq 6}$ with the property that the Hankel form $(\alpha, \beta) \mapsto c_{\alpha+\beta}, \ \ |\alpha|, |\beta| \leq 3$ is PSD but there is no point $\lambda \in {\rm int \Gamma}$ such that $(c_\gamma) \in {\mathcal{C}_N }.$
	\end{corollary} 
	
	By duality, the statement is equivalent to the existence of a polynomial $p(x)$ of degree $6$ which is non-negative on $\Gamma$ and it cannot be decomposed as a sum of squares of polynomials. The polynomial optimization community is well aware of this pitfall \cite{NP}. Since the interior of $\Gamma$ is non-empty, Motzkin's polynomial $x_1^2 x_2^2 (x^2 + x_2 - 2) + 1$ serves as an example. For details we refer to \cite[Chapter 17]{Schmudgen-book} and \cite{Evelyne,Mourrain,MS16} for a wider perspective on this notorious difficulty arising in multivariate (truncated) moment problems.

	A refined analysis of real algebraic geometry aspects of semi-algebraic convex cones in Euclidean space is contained in Sinn's article \cite{Sinn}.

\begin{remark}
    As above we can also consider the coefficient set of the complexified Stieltjes-Fantappi\`e transforms $\Phi_\mu(z, z_0)$ of finite measures $\mu$ on $\Gamma^*$ by fixing a point $(\lambda, 1) \in \operatorname{int} \Gamma \times \mathbb{R}_+$ with Taylor coefficients,
$$
c_{\alpha, j} (\Phi_\mu) = \frac{1}{\alpha! j !} \ \partial^j_{z_0} \partial_z^\alpha \Phi_\mu |_{(z, z_0)=(\lambda, 1)}, \ \ \ 0 \leq |\alpha|+ j \leq N. 
$$
Set $d_N' = \binom{N+ d+1}{d+1}$. Then as in Proposition~\ref{coefficient}, we can prove that the coefficient set 
	$$
	 \widetilde{\mathcal{C}_N } = \left \lbrace  (c_{\alpha j} (\Phi_\mu))_{|\alpha | \leq N} \in \mathbb{R}^{d_N'} : \mu \in \mathcal{M}^+(\Gamma^*) \right \rbrace
	$$
	is a semi-algebraic set in $\mathbb{R}^{d_N'}$, whenever $\Gamma^*$ can be defined by finitely many polynomial inequalities.
\end{remark}

	\subsection{Hankel kernels associated to Laplace transforms} \label{Subsec:flat-Laplace}

     We translate the complete monotonicity of Laplace transforms into the positivity of Hankel kernels formed by their jets at prescribed interpolation points. This dictionary characterizes the Laplace transforms of finite atomic positive measures, also known as {\em completely monotone exponential polynomials}.

	Let $\mu$ be a non-negative measure on $\Gamma^*$ of moderate growth. Then its Laplace transform
	$$
	F(p)= \int_{\Gamma^*} e^{-p\cdot x} d\mu(x) \ \ \text{ for } p \in \Gamma
	$$
	belongs to $C^\infty(\operatorname{int} \Gamma)$ and satisfies \eqref{Eq:CM-cond}.
	
	Let $S=\{ p^{(j)} : 1\leq j \leq N\} \subset \operatorname{int} \Gamma$ be a finite collection nodes in $\operatorname{int} \Gamma$. 
	For each $n \in \mathbb{N}_0$, we consider the following chain of finite dimensional subspaces of $C_\mathbb{R}(\Gamma^*)$:
	$$
	\mathcal{M}_n = \operatorname{span}_{\mathbb{R}} \{x^\alpha e^{-p^{(j)} \cdot x}: |\alpha| \leq n, 1\leq j \leq N \}.
	$$

    Let us consider the following ordering $\leq$ on this basis of $\mathcal{M}_n$: 
    \begin{itemize}
        \item[(i)] if $j < s$, then $x^\alpha e^{-p^{(j)}\cdot x} \leq x^\beta e^{-p^{(s)}\cdot x}$ for every $|\alpha|, |\beta| \leq n$;
        \item[(ii)] if $j=s$, then $x^\alpha e^{-p^{(j)}\cdot x} \leq x^\beta e^{-p^{(s)}\cdot x}$ provided $\alpha \leq \beta$ in the lexicographic ordering on $\mathbb{N}_0^d$.
    \end{itemize}
	We now define the positive semi-definite kernel $K_n$ on $\mathcal{M}_n \times \mathcal{M}_n$ as follows:
	$$
	K_n( f, g) \bydef \int_{\Gamma^*} f(x) g(x) d\mu(x), \ \ f, g \in \mathcal{M}_n.
	$$
	It is easy to check that $K_n$ is a positive semi-definite kernel. 
	Let $d_n$ denote the dimension of the space $\mathcal M_n$. The Hankel matrix $\Xi_{n}$ (with respect to the above ordered basis of $\mathcal{M}_n$) induced by $K_n$ on $\mathcal M_n$ is the real positive semi-definite matrix of order $d_n$ such that
	\[
	\left[\Xi_{n}\right]_{\alpha j, \beta k}=\int_{\Gamma^*} x^{\alpha+\beta} e^{-(p^{(j)} + p^{(k)}) \cdot x} d\mu(x).
	\]
	{\em The rank of the matrix $\Xi_n$ is called the rank of the positive semi-definite kernel $K_n$}.
	\bigskip

    \noindent \textbf{GNS construction for $K_n$:}
The kernel $K_n$ induces a positive semi-definite inner product $\langle \cdot , \cdot\rangle_{K_n}$ on $\mathcal{M}_n \times \mathcal{M}_n$ by 
    $$\langle f, g \rangle_{K_n} = K_n(f, g), \ \ f, g\in \mathcal{M}_n.$$
    The null space corresponding to the kernel $K_n$ is the space
	\begin{align*}
		\mathcal{N}_n = \{f \in \mathcal{M}_n : \langle f, f\rangle_{K_n} = 0 \}.
	\end{align*}
    By Cauchy-Schwarz inequality for positive semi-definite inner product, $\mathcal{N}_n$ can also be described as 
	\begin{align*}
		\mathcal{N}_n = \{f \in \mathcal{M}_n : \langle f, g\rangle_{K_n} = 0, \ \forall g\in \mathcal{M}_n \}.
	\end{align*}
    Therefore, the quotient space $\mathcal{H}_n = \mathcal{M}_n / \mathcal{N}_n$ becomes a real Hilbert space under the induced inner product i.e., 
	$$
	\langle f+ \mathcal{N}_n, g + \mathcal{N}_n \rangle = \langle f, g \rangle_{K_n}. 
	$$
	Thus, $\operatorname{dim} \mathcal{H}_n = \operatorname{rank}(K_n)$.
	Moreover, we can view $\mathcal{H}_n$ as $[\mathcal{M}_n ]$, where $[ \cdot ]$ represents the class of an element, or a vector space of functions in $L^2(\Gamma^*, \mu)$.
	In other terms
	$$ \mathcal{H}_n  =  \operatorname{span}_{\mathbb{R}} \{[x^\alpha e^{-p^{(j)} \cdot x}]\in L^2(\Gamma^*, \mu): |\alpha| \leq n, 1\leq j \leq N \}.$$

    We are now set to state one of our main results regarding the finite determinateness of the complete monotone functions on $\Gamma$.
	\begin{theorem} \label{Thm-multinode-flat}
		Let $\mu$ be a positive measure of moderate growth defined on $\Gamma^\ast$ and let $S=\{ p^{(j)} : 1\leq j \leq N\}$ be a finite collection of points in $ \operatorname{int} \Gamma$. The measure $\mu$ has finite support if and only if there exists a positive integer $n$ with the property $\operatorname{rank}(K_n) = \operatorname{rank}(K_{n+1})$.
	\end{theorem}

	In the context of multivariate power moment problems, the rank stability phenomenon described in the statement above is known as a {\it flat extension} \cite{MS16}.
    To prove the theorem, we establish a sequence of algebraic lemmas and utilize the Gelfand–Naimark–Segal (GNS) construction associated with the positive semi-definite Hankel kernels introduced above. Throughout this subsection, we fix $\mu$ and $n$ according to the statement of Theorem~\ref{Thm-multinode-flat}. Both implications will be demonstrated simultaneously.

	\vspace{2mm}
    
	In the lemmas below we assume $\operatorname{rank}(K_n) = \operatorname{rank}(K_{n+1})$.
	\begin{lemma} \label{Lem-1}
		 $\mathcal{M}_{n+1} = \mathcal{M}_n + \mathcal{N}_{n+1}.$
	\end{lemma}
	\begin{proof}
		Clearly, $\mathcal{N}_{n+1} \cap \mathcal{M}_n \subseteq \mathcal{N}_n$. Consider the following canonical maps 
		\begin{align*}
			\sigma_1: \mathcal{M}_n / (\mathcal{N}_{n+1} \cap \mathcal{M}_n) \rightarrow \mathcal{M}_n / \mathcal{N}_n  \text{ and } \sigma_2: \mathcal{M}_n / (\mathcal{N}_{n+1} \cap \mathcal{M}_n) \rightarrow \mathcal{M}_{n+1} / \mathcal{N}_{n+1}
		\end{align*}
		defined as 
		\begin{align*}
			\sigma_1(f + \mathcal{N}_{n+1} \cap \mathcal{M}_n) \bydef f + \mathcal{N}_{n} \text{ and } \sigma_2(f + \mathcal{N}_{n+1} \cap \mathcal{M}_n) \bydef f + \mathcal{N}_{n+1}.
		\end{align*}
		Note that, $\sigma_1$ is surjective and $\sigma_2$ is injective. Thus,
		\begin{align*}
			\operatorname{rank} K_n = \operatorname{dim} \mathcal{M}_n / \mathcal{N}_n \leq \operatorname{dim} \mathcal{M}_n / (\mathcal{N}_{n+1} \cap \mathcal{M}_n) \leq \operatorname{dim} \mathcal{M}_{n+1} / \mathcal{N}_{n+1} = \operatorname{rank} K_{n+1}.
		\end{align*} 
		Since $\operatorname{rank} K_n = \operatorname{rank} K_{n+1}$, the above inequalities become equalities and hence, both $\sigma_1$ and $\sigma_2$ are bijections.
		
		So, if $f\in \mathcal{M}_{n+1}$, then there exists $g\in \mathcal{M}_n$ such that $f+\mathcal{N}_{n+1} = \sigma_2(g+ \mathcal{N}_{n+1} \cap \mathcal{M}_n)$. This implies $f-g \in \mathcal{N}_{n+1}$ and so, $f = g+ \mathcal{N}_{n+1} \in \mathcal{M}_n + \mathcal{N}_{n+1}$. Therefore, $\mathcal{M}_{n+1} \subseteq \mathcal{M}_n + \mathcal{N}_{n+1}$. The reverse inclusion is obvious and this completes the proof.
	\end{proof}
	
	\begin{lemma} \label{Lem-2}
		$$\mathcal{N}_n = \mathcal{N}_{n+1} \cap \mathcal{M}_n.$$
	\end{lemma}
	\begin{proof}
		The inclusion $\mathcal{N}_{n+1} \cap \mathcal{M}_n \subseteq \mathcal{N}_n$ is obvious from the definitions. For the other containment, we use the following description of the null space associated to $K_m$ on $\mathcal{M}_m$, for any $m \in \mathbb{N}$.
		$$
		\mathcal{N}_m = \{g \in \mathcal{M}_m: \langle g, h \rangle_{K_m} = 0, \ \forall h\in \mathcal{M}_m\}.
		$$
		Now, let $f\in \mathcal{N}_n$. Let $g \in \mathcal{M}_{n+1}$. Then by Lemma \ref{Lem-1}, there exist $g_1 \in \mathcal{M}_n$ and $g_2 \in \mathcal{N}_{n+1}$ such that $g=g_1 + g_2$. The above description of the null space associated with the kernel yields,
		\begin{align*}
			\langle f, g\rangle_{K_{n+1}} = \langle f, g_1\rangle_{K_{n+1}} + \langle f, g_2\rangle_{K_{n+1}} =0.
		\end{align*}
		Thus, $f\in  \mathcal{N}_{n+1} \cap \mathcal{M}_n$ and so, $\mathcal{N}_n \subseteq  \mathcal{N}_{n+1} \cap \mathcal{M}_n.$
	\end{proof}

	\begin{lemma} \label{Lem-3}
		Suppose $1\leq j \leq d$ and $f\in \mathcal{N}_{n+1}$. If $x_j f \in \mathcal{M}_{n+1}$, then $x_j f \in \mathcal{N}_{n+1}$.
	\end{lemma}
	\begin{proof}
		Let $f$ and $j$ be as in the statement of the lemma. By Lemma \ref{Lem-1}, $x_jf = h_1 + h_2$ for some $h_1 \in \mathcal{M}_n$ and $h_2 \in \mathcal{N}_{n+1}$. For any $g\in \mathcal{M}_n$, $x_j g \in \mathcal{M}_{n+1}$ and so, $\langle f, x_j g\rangle_{K_{n+1}} = 0$ as $f\in \mathcal{N}_{n+1}$. Again, it is easy to verify that
		\begin{align*}
			\langle x_j f, g\rangle_{K_{n+1}} = \langle f, x_j g\rangle_{K_{n+1}} = 0.
		\end{align*}
		So,
		\begin{align*}
			0= \langle h_1 + h_2, g\rangle_{K_{n+1}} = \langle h_1, g\rangle_{K_{n}} + \langle h_2, g\rangle_{K_{n+1}} = \langle h_1, g\rangle_{K_{n}}.
		\end{align*}
		Therefore, $h_1 \in \mathcal{N}_n$ and finally, using Lemma \ref{Lem-2}, we have $x_j f \in \mathcal{N}_{n+1}$.
	\end{proof}
	
	\begin{lemma} \label{Lem-4}
		$\mathcal{M}_{n+2} = \mathcal{M}_n + \mathcal{N}_{n+2}$.
	\end{lemma}
	\begin{proof}
		Suppose, $f = \sum_{j=1}^N \sum_{|\alpha|\leq n+2} c_{j \alpha} \ x^\alpha e^{-p^{(j)} \cdot x} \in \mathcal{M}_{n+2}$. We can write $f= g_1 + g_2$,
		where 
		$$
		g_1 = \sum_{j=1}^N \sum_{|\alpha|\leq n+1} c_{j \alpha} \ x^\alpha e^{-p^{(j)} \cdot x} \in \mathcal{M}_{n+1} \text{ and }
		g_2= \sum_{j=1}^N \sum_{|\alpha|= n+2} c_{j \alpha} \ x^\alpha e^{-p^{(j)} \cdot x}.
		$$
		Also, $g_2$ can be expressed as $g_2 = \sum_{r=1}^d x_r h_r$, for some $h_r \in \mathcal{M}_{n+1}$. For each $1\leq r \leq d$, by Lemma \ref{Lem-1}, $h_r = u_r + v_r$ for some $u_r \in \mathcal{M}_n$ and $v_r \in \mathcal{N}_{n+1}$. So, $g_2 = \sum_{r=1}^d x_r u_r + x_r v_r$. 
		For each $1\leq r\leq d$,
		\begin{align*}
			|\langle x_r v_r, x_r v_r \rangle_{K_{n+2}}|^2 = |\langle x_r^2 v_r, v_r \rangle_{K_{n+3}}|^2 \leq \langle x_r^2 v_r, x_r^2 v_r \rangle_{K_{n+3}} \langle v_r, v_r \rangle_{K_{n+1}} =0,
		\end{align*}
		since $v_r \in \mathcal{N}_{n+1}$. Thus, $x_r v_r \in \mathcal{N}_{n+2}$ and so, $g_2 \in \mathcal{M}_{n+1} + \mathcal{N}_{n+2}$. Using Lemma \ref{Lem-1} and the fact that $\mathcal{N}_{n+1} \subseteq \mathcal{N}_{n+2}$, we finally have $f=g_1 + g_2 \in \mathcal{M}_n + \mathcal{N}_{n+2}$. Therefore, $\mathcal{M}_{n+2} \subseteq \mathcal{M}_n + \mathcal{N}_{n+2}$. The reverse containment is obvious.
	\end{proof}
	\begin{remark}
		The above lemma also illustrates that $\operatorname{rank} K_{n}= \operatorname{rank} K_{n+1}$ implies that $\operatorname{rank} K_{n+r} =\operatorname{rank} K_n$ for every $r\in \mathbb{N}$.
	\end{remark}
	
	\vspace{3mm}
	So far We have $\mathcal{M}_{n+1} = \mathcal{M}_n + \mathcal{N}_{n+1}$. But the null space $\mathcal{N}_{n+1}$ describes those functions in $\mathcal{M}_{n+1}$ which are equivalent to the zero function in $L^2_\mathbb{R}(\Gamma^*, \mu)$. Therefore, $[\mathcal{M}_{n+1}] = [\mathcal{M}_n]$ viewed as a subspace of the real inner product space $L^2_\mathbb{R}(\Gamma^*, \mu)$. To be more precise, $\mathcal{H}_n = \mathcal{H}_{n+1}$ in $L^2_\mathbb{R}(\Gamma^*, \mu)$. Similarly, Lemma \ref{Lem-4} implies that $\mathcal{H}_{n} = \mathcal{H}_{n+2}$ in $L^2_\mathbb{R}(\Gamma^*, \mu)$. Thus, $\mathcal{H}_{n} = \mathcal{H}_{n+k}$ for every $k\geq 1$. Therefore,
	\begin{align*}
		\mathcal{H} \bydef \cup_{j=1}^\infty \mathcal{H}_j = \mathcal{H}_n.
	\end{align*}
	Now, for $1\leq j \leq d$, we consider the multiplication operators $X_j$ by the coordinate functions $x_j$ on $L^2_\mathbb{R}(\Gamma^*, \mu)$.
	Also, the space $\mathcal{H}$ is invariant under the action of $X_j$ for $1\leq j \leq d$. Hence, $\mathcal{H}_n$ is invariant under the action of $X_j$ for every $1\leq j \leq d$. 
	
	Let $\mathbf{X}=(X_1, \dots, X_d)$ be the $d$-tuple of commuting real symmetric linear operators on the finite dimensional space $\mathcal{H}_n$ to prove Theorem \ref{Thm-multinode-flat}.
	
	\vspace{2mm}
	
	\begin{proof}[\textbf{Proof of Theorem~\ref{Thm-multinode-flat}}]
		
		If the measure $\mu$ has finite support, then the space $L^2(\Gamma^*, \mu)$ has finite dimension. Moreover, the class $[x^\alpha e^{-p.x}]$ is non-zero for any multi-index $\alpha$
		and point $p \in \operatorname{int} \Gamma$. Further, these exponential polynomials with $p \in S$ separate all real points. In other terms, exponential polynomials are dense in $L^2(\Gamma^\ast, \mu)$.
	Therefore there exists a positive integer $n$ with the property $\mathcal{H} = \mathcal{H}_n.$
          
		For the converse implication, we assume as in the statement of Theorem~\ref{Thm-multinode-flat}, via the above lemmas, that $\mathcal{H}_n = \mathcal{H}_{n+1}.$
		
		Since $\mathbf{X}=(X_1, \dots, X_d)$ is a $d$-tuple of commuting real symmetric linear operators on $\mathcal{H}_n$, the joint spectrum $\sigma(\mathbf{X}) \subseteq \mathbb{R}^d$ and there is a unique spectral measure $E$ on $\sigma(\mathbf{X})$ such that
		$$
		X_j = \int_{\sigma(\mathbf{X})} x_j dE(x) \ \ \text{ for } 1\leq j \leq d.
		$$
		In this case $\sigma(\mathbf{X})$ is a finite set consisting of the joint eigenvalues of $\mathbf{X}$. Let $\sigma(\mathbf{X})= \{\lambda^{(j)} \in \mathbb{R}^d: 1\leq j \leq r \}$ be the collection of the distinct joint eigenvalues of $\mathbf{X}$.

		Consider the polynomial $R(x) = \prod_{j=1}^r \|x-\lambda^{(j)}\|^2$. Since $\mathbf{X}$ is a commuting $d$-tuple of operators, we can define without ambiguity $R(\mathbf{X})$ and observe that $R(\mathbf{X}) = 0$. Therefore,
		\begin{align*}
			\int_{\Gamma^*} R(x) e^{-p^{(1)} \cdot x} e^{-p^{(2)} \cdot x} d\mu(x) = \langle R(\mathbf{X}) e^{-p^{(1)} \cdot x}, e^{-p^{(2)} \cdot x} \rangle_{\mathcal{H}_n} =0.
		\end{align*}
		
		Note that, $e^{-p^{(1)} \cdot x} e^{-p^{(2)} \cdot x}$ is a positive continuous function on $\Gamma^*$. Thus, the support of the measure $\mu$ is contained in the zero set of the polynomial $R(x)$. But from the expression of $R$ it is clear that the zero set is $\sigma(\mathbf{X})$.
		Hence, $\mu$ is finitely supported measure on $\Gamma^*$ and the support is contained in $\sigma(\mathbf{X})$. 
		
	\end{proof}
	
	\begin{remark}
		We connect the joint spectral measure $E$ associated with the commuting tuple $\mathbf{X}$ in the following way. For any polynomial $h(x)$, there exists $m \in \mathbb{N}$ depending on the degree of $h(x)$ such that
		\begin{align*}
			\int_{\Gamma^*} h(x) e^{-2p^{(1)}\cdot x} d\mu &=\langle h(x) e^{-p^{(1)}\cdot x}, e^{-p^{(1)}\cdot x} \rangle_{K_m}\\
			&= \langle \left(\int_{\Gamma^*} h(x) dE(x) \right) e^{-p^{(1)}\cdot x}, e^{-p^{(1)}\cdot x}  \rangle \\
			&= \int_{\sigma(\mathbf{X})} h(x) dE_{e^{-p^{(1)}\cdot x}, e^{-p^{(1)}\cdot x}}(x).
		\end{align*}
		Also, we know that $\mu$ is finitely supported and so, $d\mu = e^{2p^{(1)} \cdot x} dE_{e^{-p^{(1)}\cdot x}, e^{-p^{(1)}\cdot x}}$. Relying on a different point, say $p^{(2)} \in \Gamma$, one finds
		$d\mu = e^{2p^{(2)} \cdot x} dE_{e^{-p^{(2)}\cdot x}, e^{-p^{(2)}\cdot x}}$.
		This argument is well known in the Paley-Wiener theory. See for instance formula (2.13) on page 101 of \cite{Stein-Weiss}. 
	\end{remark}

	\subsection{Hankel kernels associated to Stieltjes-Fantappi\`e transforms} \label{Subsec:Flat-Fantappie}
	In this subsection, we establish a version of the flat extension theorem for Stieltjes-Fantappi{\`e} transforms, analogous to the result obtained for Laplace transforms in the previous subsection. {\em Instead of considering multiple jets at finitely many nodes, we focus here on the Taylor coefficients at a single node}.  
	
	Let $\mu$ be a finite positive Borel measure supported on $\Gamma^*$, and $\Phi(p,p_0)$ be the Stieltjes-Fantappi{\`e} transform of $\mu$ as given in \eqref{Eq:FT}, where $(p,p_0)\in \operatorname{int} \Gamma \times \mathbb{R}_+$. Fix a point $\lambda \in \operatorname{int}\Gamma$. The derivatives of $\Phi$ at $(\lambda, 1)$ are
    $$
    \partial_{z_0}^j \partial_{z}^\alpha \Phi |_{(z, z_0)=(\lambda, 1)} = (-1)^{|\alpha|+j}(|\alpha|+j)! \int_{\Gamma^*}\frac{x^\alpha \ d\mu(x)}{\ltt1+\lambda\cdot x\rtt^{|\alpha|+j+1}}.
    $$
   Recall that $\Phi$ has a holomorphic extension in the tube domain $T_{\Gamma \times \mathbb{R}_+}$. So, it has a power series expansion in a neighbourhood of the point $(\lambda, 1)$.
   
	Let us consider the following $\mathbb{R}$-linear subspaces of $C_\mathbb{R}(\Gamma^*)$, the algebra of real-valued continuous function on $\Gamma^*$. For $n\in \N_0$, we define
	\begin{align*}
		\mathcal M_n=
		\operatorname{span}_\R\left\{\frac{x^\alpha}{\ltt 1+\lambda\cdot x\rtt^{|\alpha|+j+1/2}}: (\alpha, j)\in \mathbb{N}_0^d \times \mathbb{N}_0 \text{ with } |\alpha|+j\leq n\right\}.
	\end{align*}

	Define a positive semi-definite kernel $K_n$ on $\mathcal{M}_n \times \mathcal{M}_n$ as follows:
	\begin{align*}
		K_n(f, g) = \int_{\Gamma^*} f(x) g(x) d\mu(x),\quad \text{for } f, g\in \mathcal M_{n}.
	\end{align*}
	Let $d_n'$ denote the dimension of the space $\mathcal M_n$. The Hankel matrix $\Xi_{n}$ induced by $K_n$ on $\mathcal M_n$ is a symmetric matrix of order $d_n'$
	given by 
	\[
	\left[\Xi_{n}\right]_{\alpha j, \beta k}=\int_{\Gamma^*} \frac{x^{\alpha+\beta}}{\ltt 1+\lambda\cdot x\rtt^{|\alpha|+|\beta|+j+k+1}}d\mu(x).
	\]
	{\em The rank of the matrix $\Xi_n$ (according to a fixed ordered basis of $\mathcal{M}_n$) is called  the rank of the positive semi-definite kernel $K_n$}. Again, as in the previous subsection we have a GNS construction related to the kernel $K_n$ having the null space 
	\begin{align*}
		\mathcal{N}_n = \{f \in \mathcal{M}_n : \langle f, f\rangle_{K_n} = 0 \}.
	\end{align*}

    The following theorem is a version of the flat extension theorem in this setup.
	\begin{theorem}\label{th:flat}
		Let $\mu$ be a finite positive measure on $\Gamma^*$. Then, the measure $\mu$ is finitely supported in $\Gamma^*$ if and only if there exists $n\in\N$ such that $\operatorname{rank} K_n=\rnk K_{n+1}$.
	\end{theorem}

	The proof is similar to what we have done in Subsection \ref{Subsec:flat-Laplace}. With appropriate adjustments and modifications of the lemmas there, we record the following without proof.
	\begin{proposition}\label{Prop:Flat}
    Suppose $\operatorname{rank} K_n=\rnk K_{n+1}$. Then,
		\begin{itemize}
			\item[(i)] $\mathcal{M}_{n+1} = \mathcal{M}_n + \mathcal{N}_{n+1};$
			
			\vspace{2mm}
			\item[(ii)] $\mathcal{N}_n = \mathcal{N}_{n+1} \cap \mathcal{M}_n;$
			
			\vspace{2mm}
			\item[(iii)] If for any $1\leq j \leq d$, $f\in \mathcal{N}_{n+1}$ and $\frac{x_j}{1+ \lambda \cdot x} f \in \mathcal{M}_{n+1}$, then $\frac{x_j}{1+ \lambda \cdot x} f \in \mathcal{N}_{n+1}$;
			
			\vspace{2mm}
			\item[(iv)] $\mathcal{M}_{n+2} = \mathcal{M}_n + \mathcal{N}_{n+2}$.
			
		\end{itemize}
	\end{proposition}
	
	\vspace{3mm}
	Note that, the null space $\mathcal{N}_{n+1}$ describes those functions in $\mathcal{M}_{n+1}$ which are equivalent to the zero function in $L^2_\mathbb{R}(\Gamma^*, \mu)$. Therefore, from part (i) of Proposition \ref{Prop:Flat}, $[\mathcal{M}_{n+1}] = [\mathcal{M}_n]$ viewed as a subspace of the real inner product space $L^2_\mathbb{R}(\Gamma^*, \mu)$. To be more precise, if for each $k$, $\mathcal{H}_k$ denotes the quotient space $\mathcal{M}_k / \mathcal{N}_k$, then we have $\mathcal{H}_n = \mathcal{H}_{n+1}$ in $L^2_\mathbb{R}(\Gamma^*, \mu)$. Similarly, the part (iv) of Proposition \ref{Prop:Flat}, means that $\mathcal{H}_{n} = \mathcal{H}_{n+2}$ in $L^2_\mathbb{R}(\Gamma^*, \mu)$. Thus, $\mathcal{H}_{n} = \mathcal{H}_{n+k}$ for every $k\geq 1$. Therefore,
	\begin{align*}
		\mathcal{H} \bydef \cup_{j=1}^\infty \mathcal{H}_j = \mathcal{H}_n.
	\end{align*}
	Now, for $1\leq j \leq d$, we consider the multiplication operator $Y_j$ corresponding to the function $\frac{x_j}{1+\lambda \cdot x}$ on $L^2_\mathbb{R}(\Gamma^*, \mu)$.
	Also, the space $\mathcal{H}$ is invariant under the action of $Y_j$ for $1\leq j \leq d$. Hence, $\mathcal{H}_n$ is invariant under the action of $Y_j$ for every $1\leq j \leq d$. 
	
	We now utilize the $d$-tuple, $\mathbf{Y}=(Y_1, \dots, Y_d)$ of commuting real symmetric linear operators on the finite dimensional space $\mathcal{H}_n$ to prove Theorem \ref{th:flat} as follows.
	
	\begin{proof}[\textbf{Proof of Theorem~\ref{th:flat}}]
		We begin with the joint spectrum $\sigma(\mathbf{Y}) \subseteq \mathbb{R}^d$ and the unique spectral measure $E$ on $\sigma(\mathbf{Y})$ associated with the $d$-tuple $\mathbf{Y}$ such that
		$$
		Y_j = \int_{\sigma(\mathbf{Y})} x_j dE(x) \ \ \text{ for } 1\leq j \leq d.
		$$
		In this case, $\sigma(\mathbf{Y})$ is a finite set consisting the joint eigenvalues of $\mathbf{Y}$. Let $\sigma(\mathbf{Y})= \{\gamma^{(j)}: 1\leq j \leq r \}$ be the collection of the distinct joint eigenvalues of $\mathbf{Y}$.
		
		Consider the polynomial $R(x) = \prod_{j=1}^r \|x-\gamma^{(j)}\|^2$. Since $\mathbf{Y}$ is a commuting $d$-tuple of operators, we can define $R(\mathbf{Y})$ and $Y_j$'s being operators on a finite dimensional space, we observe that $R(\mathbf{Y}) = 0$ on $\mathcal{H}_n$. 
		
		Observe that, the function $e_0(x)= \frac{1}{\sqrt{1+ \lambda \cdot x}}$ is in $\mathcal{M}_n$. If $e_0 \in \mathcal{N}_n$, then $\int_{\Gamma^*} \frac{d\mu(x)}{1+\lambda \cdot x} =0$ implying that $\mu$ is the zero measure on $\Gamma^*$. Thus, the theorem becomes trivial in this case. So, we assume $e_0$ is a non-trivial function in $\mathcal{H}_n$. Therefore,
		\begin{align*}
			\int_{\Gamma^*} R(x) e_0(x)^2 d\mu= \langle R(\mathbf{Y}) e_0, e_0\rangle_{K_{n+1}} =0.
		\end{align*}
		
		Note that, $R(x)$ is a non-negative polynomial and $e_0$ is strictly positive on $\Gamma^*$. Thus, the support of the measure $\mu$ is contained in the zero set of the polynomial $R(x)$. But from the expression of $R$ it is clear that the zero set is $\sigma(\mathbf{Y})$.
		Hence, $\mu$ is finitely supported measure on $\Gamma^*$ and the support is contained in $\sigma(\mathbf{Y})$. 
	\end{proof}

	\section{Interpolation and approximation via non-commutative operator calculus}\label{Sec:NC}
	
	This section is devoted to the interpolation and approximation of completely monotone functions defined in an acute, closed convex solid cone $\Gamma \subset \mathbb{R}^d$, holomorphically extended to the corresponding tube domain in the complex space. The approximation is constructive and the approximants are given by Laplace transforms of Wigner distributions or Fantappi\`e transforms of certain analytic functionals arising from the Hilbert space factorization of the previously analyzed Hankel kernel encoding the jet data at a prescribed point. 
    Certain canonical tuples of \textit{non-commuting} symmetric operators appear in the realization-type formulae (adopting control theory terminology) for these finitely determined entire and rational approximants in the Laplace and Stieltjes–Fantappi\`e settings, respectively. A notable feature of this approach is its computational accessibility, as the entire process ultimately relies on straightforward operations in matrix analysis.

\subsection{Directional exponential sum interpolation of Laplace transforms} \label{Subsec:Laplace-interpolation}
Let $\lambda_0$ be the vertex of an affine closed convex solid acute cone $\Gamma$. A function $f: \Gamma \to \mathbb{R}$ is said to be a {\em directionally completely monotone} if each fixed $\omega \in \Gamma$, the one-variable function 
$$
f_\omega(t) \bydef f(\lambda_0 + t \omega), \ \ t \in [0, +\infty)
$$
is completely monotone. In particular, we say that $f$ is a {\em directional exponential sum} if for each $\omega \in \Gamma$, the function $f_\omega$ is a finite sum of exponentials the form 
$$
f_\omega(t) = \sum_{j=1}^N c_j \exp(-\gamma_j t),
$$
for some non-negative real numbers $c_j$ and $\gamma_j$.

\begin{example} 
    Let $A=(1, 0)^T (1, 0)$ and $B=(1, 1)^T (1, 1)$ be two rank one PSD $2 \times 2$ matrices. Note that, $AB \neq BA$ and 
    $$AB+BA = \begin{bmatrix}
        2 & 1 \\ 1& 0
    \end{bmatrix}.$$ Now $v=(1, -1-\sqrt{2} \ )^T$ is an eigenvector corresponding to the eigenvalue $1-\sqrt{2}$ of $AB+BA$. Let us consider the function on the cone $\Gamma = \overline{\mathbb{R}^2_+}$ with vertex $\lambda_0 = (0, 0)$, given by
    $$
    F(p_1, p_2) \bydef \langle \exp(-p_1 A - p_2 B)v, v\rangle, \ \ p_1, p_2 \geq 0.
    $$
    For a fixed direction $\omega =(x_1, x_2)^T \in \Gamma$, we have the slice function:
$$
F_\omega(t) = F(t x_1, t x_2) = \langle \exp(-t(x_1 A + x_2 B)v, v\rangle.
$$
Now, $x_1A + x_2B$ is a PSD matrix with eigenvalues $\frac{1}{2}(x_1+ 2x_2 \pm \sqrt{x_1^2 + 4x_2^2})$, denoted by $\gamma_1$ and $\gamma_2$, respectively. Let $U$ be an orthogonal matrix such that 
$$U^T(x_1A + x_2B) U = \begin{bmatrix}
    \gamma_1 & 0\\ 0 & \gamma_2
\end{bmatrix}.$$
Thus, it is easy to verify that 
$$
F_\omega(t) = c_1 \exp(-t \gamma_1) + c_2 \exp(-t \gamma_2),
$$
where $c_1, c_2\geq 0$, proving that $F_\omega$ is completely monotone on $\overline{\mathbb{R}^1_+}$ and hence, $F$ is directionally complete monotone.

On the other hand, $v$ being an eigenvector of $AB+BA$ corresponding to the negative eigenvalue $1-\sqrt{2}$, we have
$$
\frac{\partial^2F}{\partial x_1 \partial x_2}(0, 0) =\langle (AB + BA)v, v \rangle < 0.
$$
Moreover, $F$ being a smooth function we can find a point $(p_1, p_2) \in \operatorname{int} \Gamma$ such that $\frac{\partial^2F}{\partial x_1 \partial x_2}(p_1, p_2) < 0.$ Thus $F$ is not a completely monotone function.
\end{example}

Now, we state one of our main interpolation results. 
	\begin{theorem}\label{th:Lap conv}
		Let $\Gamma \subset {\mathbb R}^d$ be an acute closed convex solid cone and let $\mu$ be a positive measure of moderate growth defined on $\Gamma^\ast$.
		Fix a point $\lambda \in \operatorname{int} \Gamma$. For every $n \in \mathbb{N}$, there exists a Wigner distribution associated 
		to a $d$-tuple of real symmetric matrices, such that its Laplace transform $F_n(z)$ is matching the Taylor coefficients of
		the Laplace transform
		$$ F(z) = \int_{\Gamma^\ast} \exp{(-z\cdot x)}\ d\mu(x), \ \ \Re z \in \operatorname{int} \Gamma,$$ 
        up to order $2n+1$ at $z = \lambda$. 
        
		The functions $F_n$ are entire of exponential type, and they are directional exponential sums on the affine cone $\lambda+\Gamma$. 
        
        Moreover, the sequence $(F_n)_n$ converges to $F$ on the open tube domain $T_{\lambda + \Gamma}$ associated with the cone $\lambda + \Gamma$, in the Fr\'echet topology of uniform convergence on compact subsets.
	\end{theorem}
    
\begin{proof}
Let us start with the Laplace transform $F$ of a positive measure $\mu$ having moderate growth and supported in $\Gamma^*$ and a point $\lambda \in \operatorname{int} \Gamma$ as in the statement of the theorem. Fix $n\in \mathbb{N}$.

	Let $H_n$ denote the finite dimensional subspace of $L_{\mathbb{R}}^2(\Gamma^*, \mu)$ defined as
    \begin{align*}
        H_n= \operatorname{span}_\mathbb{R} \{x^\alpha e^{-\frac{\lambda}{2} \cdot x} : \ \alpha \in \mathbb{N}_0^d, \ |\alpha| \leq n\}.
        \end{align*}
    
    As in the preceding section, for every $j \in \{1, \dots, d\}$, we consider the compression: 
    $$
    X^{(n)}_j = P_n X_j P_n
    $$ 
    of $X_j$, the multiplication by the variable $x_j$ on $L^2(\Gamma^*, \mu)$, on $H_n$. Here $P_n$ denotes the orthogonal projection onto $H_n$. Note that $X_j$ might be an unbounded, closed graph operator.
	
	The tuple $\boldsymbol{X}^{(n)} = (X^{(n)}_1, \dots, X^{(n)}_d)$ of real symmetric, finite dimensional space operators has joint numerical range $W(\boldsymbol{X}^{(n)})$ contained in the closed convex cone $\Gamma^\ast$. Indeed, the convexity of the closed  cone $\Gamma^\ast$ yields
	$$ \langle \boldsymbol{X}^{(n)} \psi, \psi \rangle =\left(\int_{\Gamma^\ast} x_1 \psi(x)^2 d\mu(x), \ldots, \int_{\Gamma^\ast} x_d \psi(x)^2 d\mu(x) \right) \in \Gamma^\ast,$$
	whenever $\psi(x) = \sum_{|\alpha| \leq n} c_\alpha x^\alpha e^{-\frac{\lambda}{2} \cdot x}$ in $H_n$ satisfies 
	$ \|\psi\| = 1.$
	Accordingly,
	$$ \Re (p\cdot \boldsymbol{X}^{(n)}) \geq 0$$ as an operator, for every $ p \in \Gamma$. 
We define the associated entire function,
	$$ F_n(z) \bydef \langle \exp (- (z- \lambda) \cdot \boldsymbol{X}^{(n)}) e^{-\frac{\lambda}{2} \cdot x}, e^{-\frac{\lambda}{2} \cdot x} \rangle, \quad z\in \mathbb{C}^d.$$
	Denote $u(x) =  e^{-\frac{\lambda}{2} .\cdot x}$ and let ${\mathcal W}_u(\boldsymbol{X}^{(n)})$ be the Wigner distribution associated to $\boldsymbol{X}^{(n)}$ and vector $u$ such that
	$$( {\mathcal W}_u(\boldsymbol{X}^{(n)})(\xi), \phi(x \cdot \xi)) = \langle \phi( x\cdot \boldsymbol{X}^{(n)}) u, u\rangle, \ \ \phi \in {\mathcal E}({\mathbb R}).$$
	Consequently, 
	$$ F_n(z + \lambda) = ({\mathcal W}_u(\boldsymbol{X}^{(n)})(\xi),\,  e^{-z \cdot \xi})$$
	is the Laplace transform of the distribution ${\mathcal W}_u(\boldsymbol{X}^{(n)})$. Clearly, $F_n(z)$ is entire function of exponential type. 

\vspace{2mm}

  \noindent \textbf{Matching  the Taylor coefficients:} We now show that the partial derivatives in the variables $z$ of $F_n$ and $F$ are identical at $z = \lambda$, up to order $2n+1$.
    
For $\alpha=(\alpha_1, \dots, \alpha_d) \in \mathbb{N}_0^d$, the coefficient of $\frac{1}{|\alpha|!}(z-\lambda)^\alpha$ in the non-commutative power series expansion of $\exp (- (z- \lambda) \cdot \boldsymbol{X}^{(n)})$ at $\lambda$, turns out to be the sum of the words in $-X^{(n)}_1, \dots,-X^{(n)}_d$ containing $\alpha_j$ many $X^{(n)}_j$ for each $j$. We denote the collection of such words by $\mathcal{M}(\alpha)$.
Thus, for $|\alpha| \leq 2n+1$, the coefficient of $(z-\lambda)^\alpha$ in the power series of $F_n(z)$ at $\lambda$
would become:
\begin{align}\label{Eq:NC-sum}
\frac{1}{|\alpha|!} \sum_{M \in \mathcal{M}(\alpha)} \langle M e^{-\frac{\lambda}{2} \cdot x}, e^{-\frac{\lambda}{2} \cdot x} \rangle.
\end{align}

Note that, if $|\alpha| \leq n$, then for any word $M \in \mathcal{M}(\alpha)$,
\begin{align}\label{Eq:NC-Laplace}
M e^{-\frac{\lambda}{2} \cdot x} = (-1)^{|\alpha|} [\boldsymbol{X}^{(n)}]^\alpha e^{-\frac{\lambda}{2} \cdot x}.
\end{align}
Moreover, if $|\alpha|\leq 2n+1$ and $M\in \mathcal{M}(\alpha)$, then we write $M=M_1M_2$ such that $M_1$ has length at most $n+1$ and $M_2$ has length at most $n$. So, each summand in \eqref{Eq:NC-sum} can be expressed in the form $\langle M_1 e^{-\frac{\lambda}{2} \cdot x}, M_2 e^{-\frac{\lambda}{2} \cdot x}\rangle$. Suppose, $M_j$ contains $s_{j, k}$ many $X^{(n)}_k$ in its expression as a word with $1\leq k \leq d$ and $\beta^{(j)} = (s_{j, 1}, \cdots, s_{j, d} )$ for $j=1,2$. Thus, using \eqref{Eq:NC-Laplace} for $M_1$ and $M_2$ we obtain 
\begin{align*}
\langle M_1 e^{-\frac{\lambda}{2} \cdot x}, M_2 e^{-\frac{\lambda}{2} \cdot x}\rangle &= \langle (-1)^{|\beta^{(1)}|} [\boldsymbol{X}^{(n)}]^{\beta^{(1)}} e^{-\frac{\lambda}{2} \cdot x}, (-1)^{|\beta^{(2)}|} [\boldsymbol{X}^{(n)}]^{\beta^{(2)}} e^{-\frac{\lambda}{2} \cdot x}\rangle \\
& =(-1)^{|\alpha|} \int_{\Gamma^*} x^\alpha e^{-\lambda \cdot x} d\mu(x).
\end{align*}
Therefore, \eqref{Eq:NC-sum} yields that the coefficient of $(z-\lambda)^\alpha$ in the power series of $F_n(z)$ at $z=\lambda$, for $|\alpha| \leq 2n+1$, is
$$
(-1)^{|\alpha|} \frac{|\mathcal{M}(\alpha)|}{|\alpha|!}\int_{\Gamma^*} x^\alpha e^{-\lambda \cdot x} d\mu(x) =  \frac{(-1)^{|\alpha|}}{\alpha !} \int_{\Gamma^*} x^\alpha e^{-\lambda \cdot x} d\mu(x).
$$
Hence, for $|\alpha| \leq 2n+1$
$$
\partial^\alpha_ z F_n |_{z=\lambda} = \partial^\alpha_ z F |_{z=\lambda}, 
$$
i.e., $F_n$ and $F$ have the matching Taylor coefficients up to order $2n+1$ at $\lambda$. 

\vspace{2mm}

 \noindent \textbf{Directional exponential sum:} Denote the function $e^{-\frac{\lambda}{2} \cdot x}$ by $u$. For a fixed $\omega \in \Gamma$, by orthogonal diagonalization of the real symmetric operator $ \omega \cdot X^{(n)}$ we obtain that
	\begin{align}\label{Eq:dir-exp-sum}
	   F_n(\lambda + t \omega) = \sum_{k=1}^N |c_k|^2 \exp{(-t \gamma_k )}, \ \ t\in [0, +\infty) 
	\end{align}
    is an sum of $N ( \leq \dim H_n)$ exponentials where the nodes $\gamma_k \in \sigma (\omega \cdot \boldsymbol{X}^{(n)}) \subset [0, +\infty)$ and the coefficients $|c_k|^2 \geq 0$ depend on the direction $\omega$. Therefore, $F_n$ is a directional exponential sum on the affine cone $\lambda +\Gamma$.
    In fact, if $U$ is an orthogonal matrix with an ordered orthonormal eigen-basis on $H_n$ as its column, then the real numbers $c_1, \dots, c_d$ are the inner products of the ordered basis vectors with $u$, respectively. So, $\sum_{j=1}^N |c_j|^2 = \|u\|^2 = F(\lambda)$. 
    
     Note that $F_{n, \omega}$ has a unique holomorphic extension to $\mathbb{H}^+_R = [0,\infty) + i \mathbb{R}$ given by
    $$
    F_{n, \omega}(z) =  \langle \exp(- z \omega \cdot \boldsymbol{X}^{(n)})u, u \rangle, \ \ z\in \mathbb{H}^+_R.
    $$
    Further, $\omega \cdot \boldsymbol{X}^{(n)}$ being positive semi-definite,
    $$ \| \exp(-z  \omega \cdot \boldsymbol{X}^{(n)}) \| \leq 1 \quad \text{for } z\in \mathbb{H}^+_R.$$
    
    Also, from \eqref{Eq:dir-exp-sum}, we see that, for $t+is \in \mathbb{H}^+_R$,
    $$
    |F_{n, \omega}(t +is)| =\left| \sum_{k=1}^N |c_k|^2 \exp{(-(t+is) \gamma_k )} \right| \leq \sum_{k=1}^N |c_k|^2 = F(\lambda).
    $$
    Further more, for $1\leq k \leq n$, we have
    \begin{align*}
        \frac{d^k F_{n, \omega}}{dz^k}(z) = (-1)^k \langle (\omega \cdot \boldsymbol{X}^{(n)})^k \exp(-z \ \omega \cdot \boldsymbol{X}^{(n)})) u, u \rangle.
    \end{align*}
For $1\leq k \le n$, the vector $(\omega \cdot \boldsymbol{X}^{(n)})^k u$ in $H_n$ can be expressed as 
$$
(\omega \cdot \boldsymbol{X}^{(n)})^k u = \sum_{\alpha \in \mathbb{N}_0^d; |\alpha|=k} \frac{k!}{\alpha!}\omega^\alpha [\boldsymbol{X}^{(n)}]^\alpha u.
$$
This is because, for any word $M$ of length $k$ $(\leq n)$ with $\alpha_j$ many $X^{(n)}_j$ as its letters for $1\leq j \leq d$, $M u = [\boldsymbol{X}^{(n)}]^\alpha u$ as illustrated in \eqref{Eq:NC-Laplace}. Therefore,
\begin{align*}
 \| (\omega \cdot \boldsymbol{X}^{(n)})^k u  \|^2 &= \sum_{\alpha, \beta \in \mathbb{N}_0^d; |\alpha|=|\beta|=k} \frac{(k!)^2}{\alpha! \beta!} \omega^{\alpha + \beta}\langle [\boldsymbol{X}^{(n)}]^\alpha u,  [\boldsymbol{X}^{(n)}]^\beta u\rangle \\
 &= \sum_{\alpha, \beta \in \mathbb{N}_0^d; |\alpha|=|\beta|=k} \frac{(k!)^2}{\alpha! \beta!} \omega^{\alpha + \beta} \int_{\Gamma^*} x^{\alpha + \beta} e^{-\lambda \cdot x} d\mu(x)\\
 &= \sum_{\alpha, \beta \in \mathbb{N}_0^d; |\alpha|=|\beta|=k} \frac{(k!)^2}{\alpha! \beta!} \omega^{\alpha + \beta} \partial^{\alpha + \beta}_z F|_{z=\lambda}\\
 &=\frac{d^k F_\omega}{dt^k}(0),
\end{align*}
    where $F_\omega$ is the slice function given by 
    $$
    F_\omega(t) \bydef F(\lambda + t \omega), \quad t\in [0, \infty).
    $$
    Whenever $0\leq k \leq n$, we assemble these information to get
    \begin{align*}
        \left| \frac{d^k F_{n, \omega}}{dz^k}(z) \right| &= \left| \langle  \exp(-z \ \omega \cdot \boldsymbol{X}^{(n)})) u, (\omega \cdot \boldsymbol{X}^{(n)})^k u \rangle \right| \\
        & \leq \|\exp(-z \ \omega \cdot \boldsymbol{X}^{(n)})) u\| \| (\omega \cdot \boldsymbol{X}^{(n)})^k u\| \\
        & \leq (F(\lambda))^{1/2} \left(\frac{d^k F_\omega}{d t^k}(0)\right)^{1/2}.
    \end{align*}

Thus, for each $k \in \mathbb{N}_0$, the sequence of smooth functions $\{F_{n, \omega}\}_n$ on $\overline{\mathbb{H}^+_U}$ satisfies
\begin{align} \label{Eq:Arzela-cond}
    \left| \frac{d^k F_{n, \omega}}{dz^k}(z) \right|  \leq (F(\lambda))^{1/2} \left(\frac{d^k F_\omega}{d t^k}(0)\right)^{1/2}, \ \text{for all } n \geq k. 
\end{align}

\vspace{2mm}
    
\noindent \textbf{Convergence of the sequence of approximants:} We complete the proof of the theorem by showing that the sequence of entire function $\{F_n\}_n$ converges to $F$ in $\mathcal{O}(T_{\lambda + \Gamma})$, where $T_{\lambda + \Gamma} = (\lambda + \operatorname{int} \Gamma) + i \ \mathbb{R}^d$.

\vspace{2mm}

\noindent \textbf{Claim 1:} The sequence $\{F_n\}_n$ is uniformly bounded on $T_{\lambda + \Gamma}$.

\noindent \textbf{Proof of Claim 1:} Take any $z= \lambda + p + i q \in T_{\lambda + \Gamma}$ where $p \in \operatorname{int} \Gamma$ and $q\in \mathbb{R}^d$. Then,
$$
F_n(z) = \langle \exp(-P - i Q) u, u \rangle,
$$
where $P= p \cdot \boldsymbol{X}^{(n)}$ and $Q= q\cdot \boldsymbol{X}^{(n)}$ which are both real symmetric operators on $H_n$. Define the $H_n$-valued function
   $$
   v(t) \bydef \exp(-t(P+ iQ)) u, \quad t\in [0, \infty).
   $$
Then 
$$
v'(t) \bydef \frac{d v}{dt}(t) = -(P+i Q) v(t).
$$ and 
\begin{align} \label{Eq:dervative}
    \frac{d}{dt}(\|v(t)\|^2) = \langle v'(t),  v(t) \rangle + \langle v(t), v'(t) \rangle = -2\langle  P \ v(t),  v(t)\rangle.
\end{align}
Denoting $\hat{v}(t) = \frac{v(t)}{\|v(t)\|}$, we have 
\begin{align*}
    \langle  P \ v(t),  v(t)\rangle = \|v(t)\|^2 \langle  p \cdot \boldsymbol{X}^{(n)} \ \hat{v}(t),  \hat{v}(t)\rangle = \|v(t)\|^2  (p \cdot \langle \boldsymbol{X}^{(n)} \ \hat{v}(t),  \hat{v}(t)\rangle) \geq 0, 
\end{align*}
since $\langle \boldsymbol{X}^{(n)} \ \hat{v}(t),  \hat{v}(t)\rangle \in W(\boldsymbol{X}^{(n)}) \subseteq \Gamma^*$. Thus, from \eqref{Eq:dervative} we have
$$
 \frac{d}{dt}(\|v(t)\|^2) \leq 0 \quad \text{for } t\in [0, \infty), 
$$
and so, for $t\in [0, \infty)$,
$$
\|v(t)\|^2 \leq \|v(0)\|^2 = \|u\|^2 = F(\lambda).
$$
Putting $t=1$, we have 
$$
|F_n(z)| \leq \|\exp(-(P+iQ))u \| \|u\| \leq \|u\|^2 = F(\lambda),
$$
for any $z \in T_{\lambda + \Gamma}$. This proves our Claim 1.

\vspace{2mm}

Now, by Montel's theorem we get a subsequence $\{F_{n_j}\}_j$ of $\{F_n\}_n$ such that $F_{n_j}$ converges to $G$ in the Fr\'echet topology of $\mathcal{O}(T_{\lambda+\Gamma})$. We shall prove that $G=F$ in $T_{\lambda + \Gamma}$. Then the original sequence $\{F_n\}_n$ itself will converge to $F$ in $\mathcal{O}(T_{\lambda+\Gamma})$. Therefore, without any loss of generality, we assume that $F_n$ converges to $G$. 

At this stage, it might be tempting to conclude that $G=F$ follows from the standard normal family argument as we have matching derivatives of $F_n$ and $F$ up to order $2n+1$ at the point $\lambda$,. But this is not the case because $\lambda$ is not an interior point of $T_{\lambda +\Gamma}$. So, we analyze the slice functions along each direction in $\Gamma$.

Fix a direction $\omega \in \Gamma$. We know that the slice function $F_{n, \omega}$ satisfies the bounds in \eqref{Eq:Arzela-cond}. Applying Arzela-Ascoli theorem, we get a subsequence $\{F_{n_j, \omega}\}$ such that for each $k\in \mathbb{N}_0$,
\begin{align}\label{Eq:Arzela-Ascoli}
\frac{d^k F_{n_j, \omega}}{dz^k} \  \rightarrow \ \frac{d^k g_\omega} {dz^k}, \quad \text{as } j\to \infty,
\end{align}
uniformly on compact subsets of $\overline{\mathbb{H}^+_R}$, for some smooth function $g_\omega$ on $\overline{\mathbb{H}^+_R}$. In particular, $g_\omega \in \mathcal{O}(\mathbb{H}^+_R)$.

\vspace{2mm}

\noindent \textbf{Claim 2:} For each $k\in \mathbb{N}_0$,
\begin{align}\label{Eq:equal-derivatives}
    \frac{d^k F_\omega}{dt^k}(0) = \frac{d^k g_\omega}{dz^k}(0).
\end{align}

\noindent \textbf{Proof of Claim 2:} From \eqref{Eq:Arzela-Ascoli}, we have
\begin{align}\label{Eq:conv-zero}
    \frac{d^k F_{n_j, \omega}}{dz^k}(0) \  \rightarrow \ \frac{d^k g_\omega} {dz^k}(0), \quad \text{as } j\to \infty.
\end{align}
Again for $k \leq n_j$,
\begin{align*}
    \frac{d^k F_{n_j, \omega}}{dz^k}(0) &= (-1)^k \langle (\omega \cdot \boldsymbol{X}^{(n_j)} )^k u, u \rangle  \\
    &= (-1)^k \sum_{\alpha \in \mathbb{N}_0^d; |\alpha|=k} \frac{k!}{\alpha!}\omega^\alpha \langle  [\boldsymbol{X}^{(n_j)} ]^\alpha u, u \rangle \\
    &=(-1)^k \sum_{\alpha \in \mathbb{N}_0^d; |\alpha|=k} \frac{k!}{\alpha!}\omega^\alpha \int_{\Gamma^*} x^\alpha e^{-\lambda \cdot x} d\mu(x) \\
    &= \sum_{\alpha \in \mathbb{N}_0^d; |\alpha|=k} \frac{k!}{\alpha!}\omega^\alpha \partial^\alpha_z F_\omega|_{z=\lambda}
    = \frac{d^k F_\omega}{dt^k}(0).
\end{align*}
The second equality above is obtained by using \eqref{Eq:NC-Laplace}.
Thus, \eqref{Eq:conv-zero} yields the equality as in \eqref{Eq:equal-derivatives} for every $k\in \mathbb{N}_0$. This proves our Claim 2.

\vspace{2mm}
\noindent \textbf{Concluding part:} Since $\lambda$ is an interior point of $\Gamma$, there exists $\epsilon > 0$ such that $\lambda + t\omega \in \operatorname{int} \Gamma$ for $t\in (-\epsilon, \epsilon)$. Also, the function $F$ is defined in the tube domain $T_\Gamma$. Thus, the function
$$
F_\omega(t) = F(\lambda + t\omega) = \int_{\Gamma^*} \exp(-(\lambda+ t\omega)\cdot x) d\mu(x)
$$
is also defined for $t\in (-\epsilon, \infty)$. However, $F$ being holomorphic in $T_\Gamma$, $F_\omega$ is real analytic in $(-\epsilon, \infty)$. So, $F_\omega$ has a power series representation in neighbourhood $(-\epsilon', \epsilon')$ of $0$. But the equality of derivatives of $F_\omega$ and $g_\omega$ at $0$ as in \eqref{Eq:equal-derivatives} implies that the same power series provides a real analytic extension of $g_\omega$ in $(-\epsilon', \epsilon')$. Remember that $g_\omega$ is holomorphic in $\mathbb{H}^+_R$. Thus, $g_\omega$ is real analytic in $(-\epsilon', \infty)$ and so, by the identity theorem for real analytic functions, $F_\omega = g_\omega$ on $(-\epsilon', \infty)$.

Therefore, for any arbitrary point $\lambda + p \in \lambda + \operatorname{int} \Gamma$. Then 
$$
\lim_{j\to \infty} F_{n_j}(\lambda + p) = \lim_{j\to \infty} F_{n_j, p} (1) = g_p(1) = F_p(1)=F(\lambda +p).
$$
Again, recall that $F_n$ converges to $G$ in $\mathcal{O}(T_{\lambda + \Gamma})$. So, 
$$
\lim_{j\to \infty} F_{n_j}(\lambda + p) = G(\lambda + p).
$$
Thus, $G(\lambda + p)=F(\lambda + p)$ and hence, $G=F$ in $\lambda + \operatorname{int} \Gamma$.  

Finally, we note that $\Gamma$ being a solid cone, $\lambda + \operatorname{int} \Gamma$ is a non-empty open set in $\mathbb{R}^d$ and so, by the Cauchy-Riemann equations for holomorphic functions in the open tube $T_{\lambda + \Gamma}$ we see that $\lambda + \operatorname{int} \Gamma$ is a uniqueness set for $T_{\lambda +\Gamma}$. Hence, $F=G$ in $T_{\lambda + \Gamma}$. This completes the proof.

\end{proof}
\begin{remark}
    The support of Wigner distribution  ${\mathcal W}_u(\boldsymbol{X}^{(n)})$ is contained in the convex hull of the joint numerical range $W(\boldsymbol{X}^{(n)})$, while its singular support is contained in {the affine part of the dual of the projective hypersurface} 
	$$ \det [ x_0 + x \cdot \boldsymbol{X}^{(n)}] = 0.$$
	For a proof we refer to \cite[Proposition 4]{SW}.
	In particular, for a fixed $\omega \in \Gamma$, the nodes $\gamma_k$ are the roots of the characteristic polynomial $\det (t - \omega \cdot X^{(n)})$.
\end{remark}

	\begin{remark}
       It is important to mention that the positive measure $\mu$ appearing in the preceding theorem may not be moment determinate. As a matter of fact, the analysis of indeterminate measures opens a wide array of challenging questions \cite{Ahiezer-Krein, Schmudgen-book}. However, the tempered measure $e^{-\lambda\cdot x} d\mu(x)$ is moment determinate because the germ of its Laplace transform at $z=\lambda$ uniquely determines the full Laplace transform, and consequently, the measure itself.
	\end{remark}

    \subsection{The determinantal variety} 
    Let us fix a positive integer $n$ and continue the notations and assumptions of Theorem~\ref{th:Lap conv} and its proof to analyze the $d$-tuple of real symmetric operators $\boldsymbol{X}^{(n)}$, acting on the finite dimensional real Hilbert space $H_n$, from an algebraic perspective.
    For notational convenience we denote the $d$-tuple $\boldsymbol{X}^{(n)}$ by $\boldsymbol{Y}= (Y_1, \ldots, Y_d)$ and the function $e^{- \frac{\lambda}{2} \cdot x}$ in $H_n$ by $u$. Then,
$$ \langle Y_j \boldsymbol{Y}^\alpha u, \, \boldsymbol{Y}^\beta u \rangle = - (-1)^{|\alpha + \beta|} \frac{\partial^{|\alpha + \beta| + 1}}{\partial p^{\alpha + \beta + 1_j}}F|_{p=\lambda}, \ \ |\alpha|, |\beta| \leq n,$$
where $1_j \in \mathbb{N}_0^d$ denotes the multi-index having $1$ at $j$-th position and $0$ elsewhere.
Since the vector $u$ is $\boldsymbol{Y}$-cyclic,
$$ H_n = \operatorname{span}_\mathbb{R} \{ \boldsymbol{Y}^\alpha u :  \ |\alpha| \leq n\}, $$
and the dimension of the Hilbert space $H_n$ is equal to the rank of the associated Hankel kernel:
$$ \dim H_n = \operatorname{rank}  \Xi,$$
where 
$$ \Xi_{\alpha, \beta} = (-1)^{|\alpha + \beta|} \frac{\partial^{|\alpha + \beta|}}{\partial p^{\alpha + \beta}}F|_{p=\lambda}, \ \ |\alpha|, |\beta| \leq n.$$
The odd order $2n+1$ data (i.e., Taylor coefficients of $F$ at $\lambda$ up to order $2n+1$) provides an extension of the above positive semi-definite kernel to values 
\begin{equation}\label{split-kernel}
 \Xi_{\alpha + 1_j, \beta} =  \Xi_{\alpha, \beta+ 1_j} = \langle Y_j \boldsymbol{Y}^\alpha u, \boldsymbol{Y}^\beta u \rangle, \ \ |\alpha|, |\beta| \leq n.
 \end{equation}
 We are interested in the invertibility of the linear pencil 
$$z_0 + z\cdot \boldsymbol{Y} : H_n \longrightarrow H_n.$$ To this aim, we complexify the Hilbert space $H_n$ and allow $z_0$ and the entries of 
$z = (z_1, \ldots, z_d)$ to be complex numbers. Clearly, $\Xi$ is essentially the Gram matrix $((\langle \boldsymbol{Y}^\alpha u, \ \boldsymbol{Y}^\beta u \rangle))$. Employing standard linear algebra arguments we can connect the rank of this Gram matrix and the matrix $((\langle  (z_0 + z\cdot \boldsymbol{Y}) \boldsymbol{Y}^\alpha u, \ \boldsymbol{Y}^\beta u \rangle))$ with the determinantal variety $\det [ z_0 + z\cdot \boldsymbol{Y} ] = 0$. This is described in the following result.

\begin{proposition} Let $\Xi$ be the Hankel kernel associated to the jet of order $2n+1$ at $\lambda \in \operatorname{int} \Gamma$  of a completely monotone function $F$ in $\Gamma$ and let $\boldsymbol{Y} = (Y_1, \ldots, Y_d)$ be the $d$-tuple of real symmetric finite dimensional operators which factors  $\Xi$ as in (\ref{split-kernel}). A $(d+1)$-tuple of complex numbers $(z, z_0)$ satisfies
$$ \det [ z_0 + z\cdot \boldsymbol{Y}] = 0$$ 
if and only if 
$$ \operatorname{rank}  ((z_0 \Xi_{\alpha, \beta} + \sum_{j=1}^d z_j \Xi_{\alpha+1_j, \beta})) < \operatorname{rank}  (( \Xi_{\alpha, \beta})),$$
both sides regarded as symmetric kernels in the variables $(\alpha, \beta), \ |\alpha|, |\beta| \leq n.$
\end{proposition}

In other terms, the projective spectrum of the $d$-tuple $\boldsymbol{Y}$ is built in the Hankel matrix of the jet data associated to $F$. This is relevant to the location of the nodes in the Laplace transform of Wigner distribution ${\mathcal W}_u(\boldsymbol{Y})$ as
$$ ({\mathcal W}_u(\boldsymbol{Y})(\xi), e^{- t \omega \cdot \xi}) = \sum_{k=1}^{\dim H_n} c_k \exp (- t \gamma_k), \ \ t > 0, \ \ \omega \in {\rm int} \  \Gamma,$$
where $-\gamma_k$ are the roots of 
$ \det [z_0 + \omega \cdot \boldsymbol{Y}] = 0$ regarded as a polynomial in $z_0$. The above finite sum of exponential is obtained via a similar argument preceding \eqref{Eq:dir-exp-sum}.

	\subsection{Rational Herglotz-Nevanlinna interpolation of Stieltjes-Fantappi\`e transforms} \label{Subsec:Fantappie-interpolation}
	
	We now return back to interpolation scheme for the Stieltjes-Fantappi\`e transforms on an acute closed convex solid cone $\Gamma \subset \mathbb{R}^d$ is. Let 
	$$\Phi(p, p_0) = \int_{\Gamma^\ast} \frac{d\mu(x)}{p_0 + p\cdot x}, \quad p_0 >0, p \in \operatorname{int} \Gamma,$$ be the Stieltjes-Fantappi\`e transform of a measure $\mu$ in $\mathcal{M}^+(\Gamma^*)$. 
    
	Fix a point $\lambda \in \operatorname{int} \Gamma$. For each $n\in \mathbb{N}$, consider the real linear subspace $H_n$ of $L_\mathbb{R}^2(\Gamma^*, \mu)$ given by
    \begin{align*}
         H_n = \operatorname{span}_\mathbb{R} \left \lbrace \frac{x^\alpha}{(1+ \lambda \cdot x)^{|\alpha|+ j}} : \ (\alpha, j) \in \mathbb{N}_0^d \times \mathbb{N}_0, \ |\alpha| + j \leq n\right \rbrace.
    \end{align*}
    Let $L_n$ denote the subspace of  $ H_n$ generated by rational functions which vanish at infinity, specifically,
	$$ L_n = {\rm span}_\mathbb{R} \left \lbrace \frac{x^\alpha}{(1+ \lambda \cdot x)^{|\alpha|+ j}} \in L_\mathbb{R}^2(\Gamma^*, \mu) : |\alpha| + j \leq n, \ j \geq 1 \right \rbrace.$$
	Note that, $1\in H_n$ and  the function $x_j f$ belongs to $H_{n+1}$ whenever $f \in L_n$, for every $ 1 \leq j \leq d$.
	The orthogonal projection of $H_{n+1}$ onto $L_n$ is denoted by $\pi_n$. Let $X_j$ stand for the closed operator of multiplication by $x_j$ on $L^2(\Gamma^*, \mu)$, where $1\leq j \leq d$.
	Define the linear operators $X^{(n)}_j$ acting on $H_n$ as follows:
	$$ X^{(n)}_j (f) \bydef \pi_n X_j \pi_n (f), \quad  f \in H_n, \ \ 1 \leq j \leq d.$$
	It is essential to note that $L_n$ is a subspace of the domain of definition of $X_j$ and
	$$ \langle X^{(n)}_j f, g \rangle = \int x_j f g d\mu, \ \ \text{ for } f, g \in H_n.$$
	Moreover, $ X^{(n)}_j f = 0$ whenever $f \in H_n \ominus L_n$. 
	
	Each $X^{(n)}_j$ is a real symmetric linear operator acting on the finite dimensional space $H_n$, but in general, the tuple 
	$\mathbf{X}^{(n)} = ( X^{(n)}_1, \ldots, X^{(n)}_d)$ is non-commutative. However, the  joint numerical range $W(\mathbf{X}^{(n)})$ is contained in $\Gamma^*$ as
    $$
    p \cdot  \langle \mathbf{X}^{(n)} \psi, \psi  \rangle = \int_{\Gamma^\ast} (p \cdot x)\  (\psi(x))^2 d\mu(x) \geq 0,
    $$
    for any $p \in \Gamma$ and $\psi \in L_n$.

	The inverse of the homogeneous pencil associated to $\mathbf{X}^{(n)}$ evaluated at a fixed vector:
	\begin{equation}\label{Phi}
	\Phi_n(z, z_0) \bydef \langle (z_0 + z \cdot \mathbf{X}^{(n)})^{-1} \pi_n 1, \pi_n 1  \rangle,
	\end{equation}
	belongs to the Herglotz-Nevanlinna class associated to the tube domain $T_{\Gamma \times \mathbb{R}_+}$. Indeed, for $(z, z_0) = (p+ i q, p_0 + i q_0) \in T_{\Gamma \times \mathbb{R}_+}$, we have
	$$
	2 \Re ( z_0 + z \cdot {\mathbf X}^{(n)} )^{-1} 
	= 2 (z_0 + z \cdot {\mathbf X}^{(n)} )^{-1} ( p_0 + p\cdot {\mathbf X}^{(n)}) 
	(z_0 + z \cdot {\mathbf X}^{(n)} )^{*-1},
	$$
	and 
	\begin{align*}
		\Re \Phi_n(z, z_0) &=  \langle ( p_0 + p\cdot \mathbf{X}^{(n)} ) g , \ g  \rangle \geq 0,
	\end{align*}
	where $g = (z_0 + z \cdot {\mathbf X}_n )^{*-1} \pi_n1.$ Here and henceforth, we complexify the real Hilbert space $H_n$.
	Above, we exploit the duality between $\Gamma$ and $\Gamma^\ast$ and the fact that the restricted joint numerical range of the $d$-tuple $\mathbf{X}^{(n)}$ belongs to $\Gamma^\ast$.
	As a matter of fact, the operator  $(z_0 + z \cdot \mathbf{X}^{(n)})^{-1}$ has non-negative real part.
	Moreover, $\Phi_n(z, z_0)$ is a rational function of degree not exceeding the dimension of the space $H_n$. 

\vspace{2mm}

\noindent \textbf{Directional Stieltjes sum:}
For a fixed $\omega \in \Gamma$, applying the spectral resolution to the real symmetric operator $\omega \cdot \mathbf{X}^{(n)}$, we infer
$$
\Phi_n(t \omega, z_0) = \langle (z_0 + t\omega \cdot \mathbf{X}^{(n)})^{-1} \pi_n 1, \pi_n 1 \rangle, \quad t \in [0, +\infty), \ \operatorname{Re} z_0 > 0, 
$$
is a sum of $N \leq \dim H_n$ simple fractions:
\begin{align*}
    \Phi_n(t\omega, z_0) = \sum_{j=1}^N \frac{c_j}{z_0 + t \gamma_j}.
\end{align*}
The nodes $\gamma_j$ belong to the spectrum $\sigma(\omega \cdot \mathbf{X}^{(n)}) \subset [0, +\infty)$ and the coefficients $c_j \geq 0$ depend on $\omega$. Setting $z_0 = 1$ yields a standard Stieltjes function of one variable arising from finite point masses. Thus, we refer to this property of being a finite sum of simple fractions along each ray in $\Gamma$ as a \emph{directional Stieltjes sum}.

	We return now to the fixed point $\lambda \in \operatorname{int} \Gamma$ and evaluate the partial derivatives of $\Phi_n(p, p_0)$ with respect to the $p$ and $p_0$ variables, at the point $(\lambda, 1)$.

 As a first step we note that, for every $n\geq 1$,
	\begin{align}\label{Eq:inv}
		(1 + \lambda \cdot \mathbf{X}^{(n)}) \left(\frac{1}{1+ \lambda \cdot x} \right) = \frac{1}{1+ \lambda \cdot x} + \pi_n  \left(\frac{\lambda \cdot x}{1+ \lambda \cdot x} \right)  =  \pi_n 1, 
	\end{align}
since $\frac{1}{1+ \lambda \cdot x} \in L_1 \subseteq L_n$. So, 
\begin{align*}
    \Phi_n(\lambda, 1) = \langle \frac{1}{1+ \lambda \cdot x}, \pi_n 1  \rangle 
                       = \int_{\Gamma^*} \frac{d\mu(x)}{1+ \lambda \cdot x} = \Phi(\lambda, 1).
\end{align*}
Again, for $0 \leq j \leq d$,
	\begin{align*}
		\partial_{p_j} \Phi_n (p, p_0)  = -  \langle (p_0 + p \cdot \mathbf{X}^{(n)})^{-1}  X^{(n)}_j   (p_0 + p \cdot \mathbf{X}^{(n)})^{-1} \ \pi_n 1, \ \pi_n 1  \rangle.  
	\end{align*}
Thus, writing $x_0 =1$, for $0 \leq j \leq d$ we get 
	\begin{align*}
		\partial_{p_j} \Phi_n|_{(p, p_0) =(\lambda, 1)} &= -  \langle (1 + \lambda \cdot \mathbf{X}^{(n)})^{-1}  X^{(n)}_j   (1 + \lambda \cdot \mathbf{X}^{(n)})^{-1} \pi_n 1, \pi_n 1 \rangle \\
		& = -  \langle  X^{(n)}_j   \frac{1}{1 + \lambda \cdot x}, \ \frac{1}{1 + \lambda \cdot x}  \rangle \ \ \ [\text{ by } \eqref{Eq:inv}] \\
		&= -  \langle  \pi_n  \left(\frac{x_j}{1 + \lambda \cdot x}\right), \ \frac{1}{1 + \lambda \cdot x} \rangle \ \ \ [\text{ as } \frac{1}{1+\lambda \cdot x} \in L_1 \subseteq L_n]  \\
		& = - \int_{\Gamma^*} \frac{x_j}{(1 + \lambda \cdot x)^2} d\mu(x) = \partial_{p_j} \Phi|_{(p, p_0) =(\lambda, 1)}.
	\end{align*}
	
	Further, for $n \geq 2$ and $ 0 \leq k \leq d$,
	\begin{align}\label{Eq:inv-2}
	    (1+ \lambda \cdot \mathbf{X}^{(n)}) \left(\frac{x_k}{(1+ \lambda \cdot x)^2}\right) 
	= \pi_n \left( \frac{x_k (1+ \lambda \cdot x)}{(1+ \lambda \cdot x)^2} \right) = \pi_n \left( \frac{x_k}{1+ \lambda \cdot x}\right).
	\end{align}
	Thus,
	\begin{align*}
		& \partial_{p_k} \partial_{p_j} \Phi_n|_{(p, p_0)=(\lambda, 1)} \\
        & =  \langle (1+ \lambda \cdot \mathbf{X}^{(n)})^{-1} X^{(n)}_k (1+ \lambda \cdot \mathbf{X}^{(n)})^{-1} X^{(n)}_j (1+ \lambda \cdot \mathbf{X}^{(n)} )^{-1} \ \pi_n 1, \ \pi_n 1  \rangle  \\
		 & \ +  \langle (1+ \lambda \cdot \mathbf{X}^{(n)})^{-1} X^{(n)}_j (1+ \lambda \cdot \mathbf{X}^{(n)})^{-1} X^{(n)}_k (1+ \lambda \cdot \mathbf{X}^{(n)})^{-1} \ \pi_n 1, \ \pi_n 1  \rangle \\ 
		&=  \langle X^{(n)}_k (1+ \lambda \cdot \mathbf{X}^{(n)} )^{-1} X^{(n)}_j \left(\frac{1}{1+ \lambda \cdot x} \right), \ \frac{1}{1+ \lambda \cdot x}  \rangle  \\
		& \ +  \langle  X^{(n)}_j (1+ \lambda \cdot \mathbf{X}^{(n)})^{-1} X^{(n)}_k \left(\frac{1}{1+ \lambda \cdot x}\right), \ \frac{1}{1+ \lambda \cdot x}  \rangle \ [\text{ by } \eqref{Eq:inv}]\\ 
		&=  \langle  (1+ \lambda \cdot \mathbf{X}^{(n)} )^{-1} \pi_n \left(\frac{x_j}{1+ \lambda \cdot x} \right), \ \pi_n \left( \frac{x_k}{1+ \lambda \cdot x} \right) \rangle  \\
		&  \ +  \langle   (1+ \lambda \cdot \mathbf{X}^{(n)})^{-1} \pi_n \left(\frac{x_k}{1+ \lambda \cdot x}\right), \ \pi_n \left( \frac{x_j}{1+ \lambda \cdot x} \right)  \rangle  \\
		&= \langle  \frac{x_j}{(1+ \lambda \cdot x)^2} , \ \frac{x_k}{1+ \lambda \cdot x} \rangle  +  \langle   \frac{x_k}{(1+ \lambda \cdot x)^2}, \ \frac{x_j}{1+ \lambda \cdot x}  \rangle \ [\text{ by } \eqref{Eq:inv-2}] \\ 
		&= 2 \int_{\Gamma^*} \frac{x_j x_k}{(1+ \lambda \cdot x)^3} d\mu(x) = \partial_{p_k} \partial_{p_j} \Phi|_{(p, p_0)=(\lambda, 1)}.
	\end{align*}

	Furthermore, following the arguments of \eqref{Eq:inv} and \eqref{Eq:inv-2}, we see that
    for $\alpha \in \mathbb{N}_0^d$ and $j \geq 1$ with $|\alpha| + j\leq n$,
	$$
	(1+ \lambda \cdot \mathbf{X}^{(n)} )^{-1} \pi_n \left( \frac{x^\alpha}{ (1+ \lambda \cdot x)^{|\alpha|+j-1 }} \right) = \left( \frac{x^\alpha}{ (1+ \lambda \cdot x)^{|\alpha|+j}}\right).
	$$
    Together with this fact, repeated application of the above differentiation procedure yields that, for every $n \geq 1$,
	$$ \partial^j_{p_0} \partial^\alpha_p \Phi_n|_{(p, p_0)= (\lambda, 1)} = \partial^j_{p_0} \partial^\alpha_p \Phi|_{(p, p_0)= (\lambda, 1)}, \ \ |\alpha|+j \leq n.$$
	
	In conclusion, we combine the preceding jet interpolation construction into a unified statement.
	\begin{theorem} \label{Thm:S-F-approximation}
    Let $\Gamma \subset {\mathbb R}^d$ be an acute closed convex solid cone and let $\Phi$ be the Stieltjes-Fantappi\`e transform of a finite positive measure on $\Gamma^*$, given by
    $$
    \Phi(z, z_0) = \int_{\Gamma^*} \frac{d\mu(x)}{z_0 + z\cdot x}, \ \ (z, z_0) \in T_{\Gamma \times \mathbb{R}_+}.
    $$
    Fix a point $(\lambda, 1)$ in $\operatorname{int} \Gamma \times \mathbb{R}_+$. Then there exists a sequence of rational Herglotz-Nevanlinna functions $\{ \Phi_n \}_n$ defined on the tube domain $T_{\Gamma \times \mathbb{R}_+}$ such that $\Phi_n$ converges to $\Phi$ uniformly on compact subsets of $T_{\Gamma \times \mathbb{R}_+}$. 
        The Taylor coefficients of $\Phi_n$ and $\Phi$ are the same up to order $n$ at the point $(\lambda, 1)$. Moreover, each $\Phi_n$ is a directional Stieltjes sum.
	\end{theorem}

		We have already established in the discussion preceding the statement of the theorem that the Taylor coefficients of $\Phi_n$ and $\Phi$ are the same up to order $n$ at $(\lambda, 1)$. The approximation part of the theorem follows from a normal family argument, since $\Phi_n(\lambda, 1)=\Phi(\lambda, 1)$ and $\Phi_n$ is in the Herglotz-Nevanlinna class of $T_{\Gamma \times \mathbb{R}_+}$, for every $n$. 

      We conclude this section with a note concerning the rate of convergence of our approximation algorithm, relying on the {\em pluricomplex Green function} of a tube domain whose existence is ensured by its biholomorphic equivalence with a bounded domain. More precisely, the tube domain over a {\it symmetric cone} carries a generalized Cayley transform to a bounded domain, see \cite[Chapter X]{Faraut-Koranyi}.
\begin{remark}
By post composing a suitably chosen Cayley transform from $\mathbb{H}^+_R$ to $\mathbb{D}$ mapping $\Phi(\lambda)=\Phi_n(\lambda)$ to $0$, with both $\Phi_n$ and $\Phi$, we can produce uniformly bounded functions $\tilde{\Phi}_n$ and $\tilde{\Phi}$, respectively in $T_{\Gamma \times \mathbb{R}_+}$. It is easy to verify that the local Taylor coefficient-matching of $\tilde{\Phi}_n$ and $\tilde{\Phi}$ up to order $n$ at the interior point $(\lambda, 1)$, of $T_{\Gamma \times \mathbb{R}_+}$ continues to hold as the same is true for $\Phi_n$ and $\Phi$.
This forces a global geometric error bound across the entire tube domain via the pluricomplex Green's function $\mathcal{G}$ associated to the domain $T_{\Gamma \times \mathbb{R}_+}$ having logarithmic pole at $(\lambda, 1)$ and defined as:
$$\mathcal{G}((z, z_0), (\lambda, 1)) \bydef \sup  v(z, z_0).$$
The supremum is taken over all non-positive plurisubharmonic (psh) functions $v$ on $T_{\Gamma \times \mathbb{R}_+}$ such that $v$ has a logarithmic pole at $(\lambda, 1)$.
For more details on pluricomplex Green's function, see \cite[Chapter 6]{Klimek}.

The error function:
$$E_n(z, z_0) \bydef \tilde{\Phi}(z, z_0) - \tilde{\Phi}_n(z, z_0), \ \ (z, z_0)\in T_{\Gamma \times \mathbb{R}_+},$$
is uniformly bounded by $2$ in $T_{\Gamma \times \mathbb{R}_+}$.
Also, $E_n(z, z_0)$ has a local zero of order $n$ at $(\lambda, 1)$ i.e.,
$$E_n(z, z_0) = O(\|(z, z_0) - (\lambda, 1)\|^n) \quad \text{as } (z, z_0) \to (\lambda, 1).$$
Therefore the function, given by
$$v(z, z_0) = \frac{1}{n} \log \left| \frac{E_n(z, z_0)}{2} \right|, \quad (z, z_0)\in T_{\Gamma \times \mathbb{R}_+},$$
is a non-positive psh function in $T_{\Gamma \times \mathbb{R}_+}$ having a logarithmic pole at $(\lambda, 1)$. Thus,
$$\frac{1}{n} \log \left| \frac{E_n(z, z_0)}{2} \right| \le \mathcal{G}((z, z_0), (\lambda, 1)),$$
and so,
$$|\tilde{\Phi}(z, z_0) - \tilde{\Phi}_n(z, z_0)| \le 2 \left( e^{\mathcal{G}((z, z_0), (\lambda, 1))} \right)^n$$

Reverting back to the original error function $\Phi(z, z_0) - \Phi_n(z, z_0)$ via the inverse Cayley transform from $\mathbb{D}$ to $\mathbb{H}^+_R$, we observe that on any compact subset $S \subset T_{\Gamma \times \mathbb{R}_+}$, 
$$|\Phi(z, z_0) - \Phi_n(z, z_0)| \le C_S |\tilde{\Phi}(z, z_0) - \tilde{\Phi}_n(z, z_0)| \le 2 C_S \left( e^{\mathcal{G}((z, z_0), (1, \lambda))} \right)^n,$$
for some constant $C_S \geq 0$.
If we define $q_S = \sup_{(z, z_0) \in S}  e^{\mathcal{G}((z, z_0), (1, \lambda))}$, then because $\mathcal{G}$ is strictly negative in the interior of the domain, $q_S < 1$.
Taking the supremum over $S$ yields a strict geometric convergence rate:
$$\|\Phi - \Phi_n\|_{S, \infty} \le 2 C_S \cdot (q_S)^n.$$
\end{remark}

        \vspace{2mm}
        
\subsection{Analytic duality and directional Stieltjes sums}\label{sub:analytic dual}
In complete analogy to the Laplace transform framework and Wigner distributions, the approximants $\Phi_n$ (as in Theorem~\ref{Thm:S-F-approximation}) for the Stieltjes-Fantappi\`e transfroms of positive measures 
emerge as Fantappi{\`e} transforms of analytic functionals carried by the convex hull of the joint numerical range of the $d$-tuple $\mathbf{X}^{(n)}$.
In the end the associated Wigner distributions serve as representatives of these analytic functionals.

 To be more specific, let us first consider the convex hull of $W(\mathbf{X}^{(n)})$, denoted by $\operatorname{conv} W(\mathbf{X}^{(n)})$, which is a convex compact set in $\C^d$. Due to compactness, we can identify $\operatorname{conv} W(\mathbf{X}^{(n)})$ as a compact subset of the $d$-dimensional complex projective space $\C\mathbb P^d$ via the following identification:
  \beas
  \operatorname{conv} W(\mathbf{X}^{(n)})&\rightarrow& \C\mathbb P^d\\
  (x_1,x_2,\cdots,x_d)&\rightarrow& [1:x_1:x_2:\cdots:x_d],
  \eeas
  where $[z_0:z_1:\cdots:z_d]$ denotes the homogeneous coordinates in $\C\mathbb P^d$. 
  Furthermore, let us consider the domain
  \[
\Omega_{W(\mathbf{X}^{(n)})}=\left\{[z_0:z]\in\C\mathbb P^{d}: z_0+\langle z, x\rangle\neq 0,\, \forall x\in \operatorname{conv} W(\mathbf{X}^{(n)}) \right\}.
  \]
  Then $\Omega_{W(\mathbf{X}^{(n)})}^*$ is the dual complement of the compact set $\operatorname{conv}W(\mathbf{X}^{(n)})$, as in \eqref{eq:projective dual}, and due to convexity of the set $\operatorname{conv}W(\mathbf{X}^{(n)})$, $\Omega_{W(\mathbf{X}^{(n)})}^*=\operatorname{conv}W(\mathbf{X}^{(n)})$. Here we are viewing $\operatorname{conv}W(\mathbf{X}^{(n)})$ as a subset of $\C\mathbb P^d$ as mentioned above. Also, by construction, the approximant $\Phi_n$ is $-1$-homogeneous holomorphic function in $\Omega_{W(\mathbf{X}^{(n)})}$. Hence, $\Phi_n$ can be viewed as a section of the tautological holomorphic line bundle $\mathcal O_{-1}\ltt\Omega_{W(\mathbf{X}^{(n)})}\rtt$; see \cite[Section~3.2]{Andersson Passare book} for a definition. Hence, by the Aizenberg--Martineau duality theorem (\cite[Theorem~3.5.3]{Andersson Passare book}), there exists an analytic functional $\Psi_{n}\in\mathcal O'(\operatorname{conv} W(\mathbf{X}^{(n)}))$ such that for $[z_0:z]\in \Omega_{W(\mathbf{X}^{(n)})}$,
\beas
\Phi_n([z_0:z])&=&\Psi_n\ltt[1:x]\rightarrow\frac{1}{\langle [1:x],[z_0:z]\rangle },\,x\in W(\mathbf{X}^{(n)})\rtt,
\eeas
Passing to the standard affine chart, the projective pairing simplifies to $ \langle [1:x],[z_0:z]\rangle = z_0 + \langle z,x\rangle$, and hence,
$$
\Phi_n(z_0,z) = \Psi_n \left( x \mapsto \frac{1}{z_0 + \langle z,x\rangle},\,x\in W(\mathbf{X}^{(n)}) \right),
$$
which shows that the Fantappi{\`e} transform of $\Psi_n$ is precisely the approximating rational function $\Phi_n$.

Given the original definition (\ref{Phi}) of the rational function $\Phi_n$, we infer that the analytic functional $\Psi_n$ is the restriction (in the sense of linear functionals) of Wigner's distribution associated to the $d$-tuple $\mathbf{X}^{(n)}$ and the vector $\pi_n 1$.  Note that, in general, a tempered distribution of compact support is not determined by its restriction to the space of holomorphic functions.

\section{Matrix-valued completely monotone functions} \label{Sec:Matrix-CM}
For a fix $m \in \mathbb{N}$, {\em we say that a function $F: \Gamma \rightarrow M_m(\mathbb{R})$ is $M_m(\mathbb{R})$-valued completely monotone function if it is smooth in $\operatorname{int} \Gamma$ and the complete monotonicity sign property \eqref{Eq:CM-cond} holds in the sense of usual matrix-ordering}.

Let $A = (A_1, \ldots,A_d)$ be a tuple of real symmetric $m \times m$ matrices. The two transforms which
produce completely monotone functions in the scalar case have a direct analogue in the matrix-valued framework. Specifically, the Laplace transform:
$$ {\mathcal L}(A)(p) \bydef \exp (- p \cdot A),\ \ p \in \Gamma,$$
and the Stieltjes-Fantappi\`e transform:
$$ \Phi(A)(p, p_0) \bydef (p_0 + p\cdot A)^{-1},\ \ p_0 > 0, \ p\in \operatorname{int} \Gamma,$$
whenever they are well-defined. They are the key protagonists in our constructive approximation scheme illustrated in Subsection \ref{Subsec:Laplace-interpolation} and \ref{Subsec:Fantappie-interpolation}. Thus it is natural to look for a criterion to have the approximants in the respective class of complete monotone functions.
Aiming at such a criterion, we provide necessary and sufficient conditions involving the partial derivatives of these matrix-valued transforms so that the $d$-tuple becomes commutative. Before stating that criterion, we need a purely algebraic lemma which is of independent interest.
\begin{lemma}\label{Lem:comm}
    If $A_1$ and $A_2$ are two real $m\times m$ PSD matrices and the Weyl-symmetrized words $C_{k, 1}(A_1, A_2)$ are PSD for all $k\in \mathbb{N}$, then $A_1$ commutes with $A_2$. By Weyl-symmetrized word $C_{k, 1}(A_1, A_2)$, we mean the coefficient of $s^k t$ in the expansion of $(sA_1 + tA_2)^{k+1}$ i.e., 
    $$
    C_{k, 1}(A_1, A_2) = \sum_{j=0}^k A_1^{k-j} A_2 A_1^j.
    $$
\end{lemma}
\begin{proof}
By the hypothesis, $A_1$ is PSD and hence, $A_1$ is diagonalizable. Let $\mathcal{U}=\{u_1, \dots, u_m\}$ be an orthonormal eigen-basis for $A_1$, with respective non-negative eigenvalues $\lambda_1, \dots, \lambda_m \ge 0$. We write $C_k$ as a short hand notation for $C_{k, 1}(A_1, A_2)$. We represent $A_2$ and $C_k$ in the basis $\mathcal{U}$ and denote their $(u, v)$-entry by $b_{uv}$ and $(C_k)_{uv}$, respectively, for each $k\geq 1$. Then,
$$
(C_k)_{uv} = \sum_{j=0}^{k} \lambda_u^{k-j} b_{uv} \lambda_v^j = b_{uv} \sum_{j=0}^{k} \lambda_u^{k-j} \lambda_v^j.
$$
Evaluating this geometric series gives two cases: 
\begin{itemize}
    \item[(i)] if $\lambda_u = \lambda_v = \lambda$:
$$(C_k)_{uv} = (k+1) \lambda^k b_{uv}.$$

\item[(ii)] if $\lambda_u \neq \lambda_v$:
$$(C_k)_{uv} = b_{uv} \frac{\lambda_u^{k+1} - \lambda_v^{k+1}}{\lambda_u - \lambda_v}.$$
\end{itemize}

Because $C_k$ is PSD, every $2 \times 2$ principal minor of the matrix $((C_k)_{u,v})$ must have non-negative determinant. For any two indices $u$ and $v$, this requires:
$$(C_k)_{uu} (C_k)_{vv} \ge (C_k)_{uv}^2.$$

Assuming $\lambda_u \neq \lambda_v$, we now prove that $b_{uv} = 0$. 
For that purpose, substituting our formulae for the entries into this inequality yields:
\begin{align}\label{Eq:matrix-comm}
    (k+1)^2 (\lambda_u \lambda_v)^k b_{uu} b_{vv} \ge b_{uv}^2 \left( \frac{\lambda_u^{k+1} - \lambda_v^{k+1}}{\lambda_u - \lambda_v} \right)^2.
\end{align}
 We handle this in two sub-cases based on whether the eigenvalues are zero.
\begin{itemize}
    \item[\text{Case 1.}] One eigenvalue is zero.

Suppose $\lambda_v = 0$ and $\lambda_u > 0$. Then for $k \ge 1$, the diagonal entry $(C_k)_{vv} = 0$.
A fundamental property of PSD matrices is that if a diagonal entry is zero, its entire corresponding row and column must also be zero. Therefore, $(C_k)_{uv} = 0$.
Looking at our formula for $(C_k)_{uv}$ when $\lambda_v = 0$, we have,
$$
(C_k)_{uv} = b_{uv} \frac{\lambda_u^{k+1} - 0}{\lambda_u - 0} = b_{uv} \lambda_u^k.
$$
Since $\lambda_u > 0$ and $(C_k)_{uv} = 0$, it immediately follows that $b_{uv} = 0$.

\item [\text{Case 2.}] Both eigenvalues are strictly positive.

Suppose $\lambda_u > \lambda_v > 0$. We can divide both sides of the inequality \eqref{Eq:matrix-comm} by $(\lambda_u \lambda_v)^k$ to get
$$
(k+1)^2 b_{uu} b_{vv} \ge b_{uv}^2 \left( \frac{(\lambda_u/\lambda_v)^{(k+1)/2} - (\lambda_v/\lambda_u)^{(k+1)/2}}{(\lambda_u/\lambda_v)^{1/2} - (\lambda_v/\lambda_u)^{1/2}} \right)^2.
$$
Now take limit as $k \to \infty$.
The left-hand side, $(k+1)^2 b_{uu} b_{vv}$, has polynomial growth.
On the right-hand side, because $\lambda_u > \lambda_v$, the ratio $\lambda_u/\lambda_v > 1$. Therefore, the term $(\lambda_u/\lambda_v)^{(k+1)/2}$ dominates, and the entire right-hand side grows exponentially at the rate of $(\frac{\lambda_u}{\lambda_v})^k$.

If $b_{uv} \neq 0$, the exponential growth of the right-hand side will eventually surpass the polynomial growth of the left-hand side, violating the PSD inequality for sufficiently large $k$. This will lead to a contradiction as the matrix $C_k$ is PSD for all $k$. Hence, we must have $b_{uv} = 0$.
\end{itemize}
Therefore, $b_{uv} = 0$ whenever $\lambda_u \neq \lambda_v$. This means that $A_2$ is block-diagonal with respect to the distinct eigenspaces of $A_1$. Consequently, $A_1$ and $A_2$ are simultaneously diagonalizable, meaning they commute.
\end{proof}

  Now, we consider the $d$-tuple $A=(A_1, \dots, A_d)$ of $n \times n$ real symmetric matrices and the associated linear pencil
$$L(x) = x\cdot A, \ \ \text{ for } x\in \mathbb{R}^d.$$

\begin{theorem}\label{Thm:operator-CM}
    Let the matrix $L(p)$ be PSD for every vector $p$ lying on an extreme ray of $\Gamma$. Suppose there exists a set of $d$ linearly independent vectors $v^{(1)}, \dots, v^{(d)}$ chosen from the extreme rays of $\Gamma$ such that for every pair of indices $i, j \in \{1, \dots, d\}$, the Weyl-symmetrized words $C_{k, 1}(L(v^{(i)}), L(v^{(j)}))$ are PSD for all $k \ge 1$. Then $A_1, \dots, A_d$ mutually commute, and all their joint eigenvalues lie in $\Gamma^*$.
\end{theorem}
\begin{proof}
First we prove pairwise commutativity on the basis $\{v^{(1)},\dots, v^{(d)}\}.$
Towards that let $X = L(v^{(i)})$ and $Y = L(v^{(j)})$ be the matrices associated with any two vectors from our chosen basis.

Since $v^{(i)}$ and $v^{(j)}$ are both extreme rays, by our assumptions, $X$ and $Y$ are both PSD as well as $C_{k, 1}(X, Y)$ is PSD for every $k\geq 1$. So, by our Lemma \ref{Lem:comm}, $X$ commutes with $Y$. Thus,
$$
L(v^{(i)}) L(v^{(j)}) = L(v^{(j)}) L(v^{(i)}),
$$
for all pairs $i, j \in \{1, \dots, d\}$.
Because the vectors $v^{(1)}, \dots, v^{(d)}$ form a basis for $\mathbb{R}^d$. Let $e_m$ be the $m$-th standard basis vector of $\mathbb{R}^d$ and let
$e_m = \sum_{i=1}^d c_i v^{(i)}$, for some real numbers $c_i$.
So,
$$
A_m = L(e_m) = \sum_{i=1}^d c_i L(v^{(i)}).
$$
Because linear combinations of mutually commuting matrices also commute with one another, any two matrices $A_r$ and $A_m$ commute.

Finally, we show that the joint eigenvalues of $A$ lie in $\Gamma^*$.

Because $A_1, \dots, A_d$ are mutually commuting and symmetric, they share an orthonormal basis of joint eigenvectors. Let $u$ be any joint eigenvector of $A$, and let $\lambda = (\lambda_1, \dots, \lambda_d) \in \mathbb{R}^d$ be its corresponding joint eigenvalue, such that $A_j u = \lambda_j u$ for all $j$.

Now, let $p$ be an arbitrary vector in the cone $\Gamma$.
By the definition of a closed convex cone, any vector $p \in \Gamma$ can be expressed as a finite sum of non-negative multiples of the vectors on its extreme rays:
$p = \sum_{m} s_m p^{(m)},$ where $s_m \ge 0$ and $p^{(m)}$ are extreme rays of $\Gamma$. So,
$$
L(p) = \sum_{m} s_m L(p^{(m)}).
$$
By our hypothesis, $L(p^{(m)})$ is PSD for every $m$ and hence, $L(p)$ is PSD matrix for all $p \in \Gamma$. Applying $L(p)$ to our joint eigenvector $u$:
$$
L(p)u = \sum_{i=1}^d p_i (A_i u) = \left( \sum_{i=1}^d p_i \lambda_i \right) u = ( p \cdot \lambda) u.
$$
This implies that $u$ is an eigenvector of $L(p)$ with the eigenvalue $ p\cdot \lambda$.
But $L(p)$ being PSD, $p \cdot \lambda \geq 0$, for any $p \in \Gamma$. Hence, $\lambda \in \Gamma^*$.
\end{proof}

From Subsection \ref{Subsec:Laplace-interpolation}, we recall the compressions $X^{(n)}_j$ of the coordinate multipliers $X_j$ on the real finite dimensional Hilbert space $H_n$, for $n\in \mathbb{N}$ and $u=e^{-\frac{\lambda}{2} \cdot x}$. We consider these compressions which are real symmetric, in the following corollary. 
\begin{corollary} \label{Cor:comm-LT}
Under the assupmtions of Theorem \ref{Thm:operator-CM} with $A_j = X^{(n)}_j$ for $1\leq j \leq d$,  the entire function $\langle \exp (-(z-\lambda) \cdot A) u, u \rangle$
 is a finite positive linear combination of exponentials on the tube domain over the forward cone $\lambda + \Gamma$ i.e., 
 $$
 \langle \exp (-(z-\lambda) \cdot A) u, u \rangle = \sum_{j=1}^{\operatorname{dim} H_n} \gamma_j \exp(-(z-\lambda) \cdot \lambda^{(j)}),
 $$
 where $\gamma_j \geq 0$ and $\lambda^{(j)} \in \Gamma^*$.
\end{corollary}
\begin{proof}
By Theorem \ref{Thm:operator-CM}, $A$ becomes a $d$-tuple of commuting real symmetric matrices. So, we can simultaneously orthogonally diagonalize each of the operators in $A$ such that
$$
A_j = U^* \Lambda_j U \ \text{ and } \Lambda_j = \operatorname{diag} (\lambda^{(1)}_j, \ldots, \lambda^{(N)}_j),
$$
where $U$ is the orthogonal matrix made of orthonormal joint eigenbasis of $H_n$, $N= \operatorname{dim} H_n$, for $1\leq j \leq d$. Note that the joint eigenvalues $\lambda^{(j)}=(\lambda^{(j)}_1, \ldots, \lambda^{(j)}_d)$ of $A$  are in $\Gamma^*$. Thus, for $z$ in the tube domain $T_{\lambda+ \Gamma}$,
$$
\langle \exp (-(z-\lambda) \cdot A) u, u \rangle = \sum_{j=1}^N \gamma_j \exp(-(z-\lambda) \cdot \lambda^{(j)}).
$$ 
\end{proof}

Finally, we move on to the Stieltjes-Fantappi\`e side. From subsection \ref{Subsec:Fantappie-interpolation}, we recall the real symmetric operators $X^{(n)}_j$ on the finite dimensional real Hilbert space $H_n$ arising from the coordinate multipliers and $v= \pi_n 1$. In the following corollary, we consider these operators.
\begin{corollary} \label{Cor:comm-FT}
 Under the assumptions of Theorem \ref{Thm:operator-CM} with $A_j = X^{(n)}_j$ for $1\leq j \leq d$, the function
 $$
 \Phi_n(z, z_0)=  \langle (z_0 + z \cdot A)^{-1} v, v \rangle, \ (z, z_0) \in T_{\Gamma \times \mathbb{R}_+},
 $$
 is of the form:
 $$
 \Phi_n(z, z_0) = \sum_{j=1}^{\operatorname{dim} H_n} \frac{\gamma_j}{z_0 + z\cdot \lambda^{(j)}}, \, \, \exists \ \gamma_j \geq 0, \ \lambda^{(j)} \in \Gamma^*.
 $$
\end{corollary}

\begin{proof}
    The proof is similar to the previous corollary.
\end{proof}

\section{Commutative dilations and restrictions of matrix tuples} \label{Sec:Comm-dilation}

\subsection{Construction of Gaussian cubature}	On the effective approximation side, we have so far produced a jet interpolation scheme of completely monotone functions, defined on a cone $\Gamma$, by elements which are finitely determined but only directionally complete monotonic. However, Corollary \ref{Cor:comm-LT} and Corollary \ref{Cor:comm-FT} shed some light in obtaining approximants that are nonnegative linear combination of exponential sums, or nonnegative linear combination of simple fractions by imposing certain conditions on the compressed coordinate multipliers.
    
    We outline below how to use the key matrix tuple appearing in this scheme for reaching, at least theoretically, a true interpolation by
	simple, finitely determined completely monotone functions. We closely follow the notes \cite{Dilcub1,Dilcub2}.
	
    Let $\mu$ be a positive measure supported on $\Gamma^\ast$ having moderate growth. Let $M = (M_1, \ldots, X_d)$ be the (unbounded, closed graph) multipliers by the coordinate functions in $L_\mathbb{R}^2(\Gamma^*, \mu)$. Fix a positive integer $n$ and consider a point $\lambda \in {\rm int} \ \Gamma$. We worked  with the Krylov subspace $H_n$ of order $n$ generated by $M$ and the vector $u = e^{- \frac{\lambda}{2} \cdot x}$, i.e., 
	$H_n \subset L_\mathbb{R}^2(\Gamma^*, \mu)$ is generated by the vectors $M^\alpha u, \ \ |\alpha| \leq n.$ Let $P$ denote the orthogonal projection of $L_\mathbb{R}^2(\Gamma^*, \mu)$ onto $H_n$ and define
	the compressed tuple $X$ of symmetric operators $X_j = PM_jP, \ 1 \leq j \leq d$. In particular, whenever $|\alpha|, |\beta| \leq n$ one finds
    \begin{align*}
       \langle X_j X^\alpha u , X^\beta u \rangle &= \langle PM_j P (PMP)^\alpha u , (PMP)^\beta u \rangle   
	  = \langle PM_j M^\alpha u, M^\beta \rangle \\
    & = \langle M_j M^\alpha u, M^\beta u \rangle 
     =  \langle M^{\alpha + 1_j} u, M^\beta u \rangle \\
    & = \int x^{\alpha + \beta + 1_j} e^{- \lambda \cdot x} d\mu(x).
	 \end{align*}
	
	In virtue of Theorem \ref{Thm:RT} of Richter and Tchakalov, there exists a finite point mass positive measure $\nu$, supported on $\Gamma^\ast$, with the property
	$$ \int x^\alpha d\nu(x) = \int x^\alpha e^{- \lambda \cdot x} d\mu(x), \ \ |\alpha| \leq 2n+1.$$ Denote by $\tilde{H}_n$ the Krylov subspace of order $n$ generated by the 
	tuple $N=(N_1, \dots, N_d)$ of coordinate multipliers and the vector $1$ in $L_\mathbb{R}^2(\Gamma^*, \nu)$. Let $Q$ be the orthogonal projection of $L_\mathbb{R}^2(\Gamma^*, \nu)$ onto $\tilde{H}_n$ and consider the tuple of symmetric operators $Y=(Y_1, \dots, Y_d)$ with
	$Y_j = QN_jQ$. For $|\alpha|, |\beta| \leq n$ and $1 \leq j \leq d$, a repetition of the above computation yields:
    \begin{align*}
    \langle Y_j Y^\alpha 1, Y^\beta 1 \rangle =  \int x^{\alpha + \beta + 1_j} e^{- \lambda \cdot x} d\nu(x) 
    &= \int x^{\alpha + \beta + 1_j} e^{- \lambda \cdot x} d\mu(x) \\
    &=  \langle X_j X^\alpha u , X^\beta u \rangle. 
    \end{align*}
	
	Consequently, the $d$-tuples of (non-commutative) symmetric operators $X$ and $Y$ are unitarily equivalent via a unitary from $H_n$ to $\tilde{H}_n$ carrying the vector $u$ into $1$.
	
	On the other hand, the system of multipliers $N$ is commutative, acting on a space of dimension at most $ D = \binom{2n+1 + d}{d}$.
	
	By reversing these observations, we obtain the following algorithm of interpolating the Laplace transform
	$$ F(z) = \int_{\Gamma^*} e^{- z\cdot x} d\mu(x), \ \ z\in T_\Gamma,$$
	by a genuine finite sum of exponential with positive coefficients
	$$ F_n(z) = \int_{\Gamma^*} e^{- z\cdot x} d\nu(x), \ \ z\in T_\Gamma,$$
	satisfying
	$$ \partial_z^\alpha (F-F_n)|_{z=\lambda} = 0, \ \ |\alpha| \leq 2n+1.$$
	
	Explicitly, we seek symmetric completions of the matrices $X_j$ which enter into the factorization of the Hankel kernel associated to $F$ at the point $z = \lambda$:
	$$ N_j = \begin{pmatrix} X_j & V_j\\ V_j^\ast & W_j\end{pmatrix}, \ 1 \leq j \leq d,$$
	acting on a Hilbert space $\mathcal{H}_n = H_n \oplus (\mathcal{H}_n \ominus H_n)$  of dimension $D$, constrained by the conditions: 
        $$ \sigma(N) \subset \Gamma^\ast,$$
	$$ [N_j , N_k ] = 0, \ 1 \leq j,k \leq d,$$
	and
	$$ V_j^\ast X^{\alpha} u = 0, \ \ |\alpha| \leq n-1, \ 1 \leq j \leq d.$$
	Indeed, the second relation implies 
	$$ N^\alpha u = X^\alpha u, \ \ |\alpha| \leq n.$$
	
	The complexity of this simultaneous matrix completion problem is considerably higher than the directional complete monotonic relaxation offered by our Wigner distribution method.
	
In a similar manner with appropriate modifications, we can obtain an algorithm of interpolating the Stieltjes-Fantappi\`e transform 
	$$ \Phi(z, z_0) = \int_{\Gamma^*} \frac{d\mu(x)}{z_0 + z\cdot x}, \ \ (z, z_0) \in T_{\Gamma \times \mathbb{R}_+},$$
	by a genuine finite sum of simple fractions with positive coefficients
	$$ \Phi_n(z) = \int_{\Gamma^*} \frac{d\nu(x)}{z_0 + z\cdot x}, \ \ (z, z_0) \in T_{\Gamma \times \mathbb{R}_+},$$
	satisfying
	$$ \partial_{z_0}^j \partial_z^\alpha (\Phi - \Phi_n)|_{(z, z_0)=(\lambda, 1)} , \ \ |\alpha|+j \leq n,$$ where $\nu$ is a positive, finitely supported measure on $\Gamma^*$.
		
\subsection{Uniqueness of the interpolatory tuples of matrices} Returning to the construction of the directional completely monotone functions matching the prescribed data, we isolate a purely Hilbert space statement which reveals the key properties of the tuples of symmetric matrices $\boldsymbol{X}^{(n)}$ appearing in 	our approximation scheme (Theorem~\ref{th:Lap conv} and Theorem~\ref{Thm:S-F-approximation}).

\begin{proposition}\label{unique-tuple} Let $A = (A_1, \ldots, A_d)$ be a system of symmetric matrices acting on a finite dimensional real Hilbert space $H$ and let $n$ be a positive integer. Suppose a vector $u \in H$ has the spanning property
$$ {\rm span} \{ A^\alpha u \in H : \ |\alpha| \leq n \} = H,$$
where the monomial $A^\alpha$ is taken in any order of the factors, and
$$ A_j A_k v = A_k A_j v, \ \text{ for } \ v \in  {\rm span} \{ A^\alpha u \in H: \ |\alpha| \leq n-1 \}, \ 1 \leq j, k \leq d.$$
Then the Taylor polynomial of degree $2n+1$ in the expansion of \\
$ \langle \exp ( -x \cdot A)u, u \rangle $ at $x=0$ determines $A$ up to unitary equivalence.
\end{proposition} 

\begin{proof} Let $\alpha, \beta \in \mathbb{N}_0^d$ with $|\alpha|, |\beta| \leq n.$ Due to the commutativity assumption, the coefficient of $x^{\alpha+\beta + 1_j}$ in the Taylor series of
$ \langle \exp ( -x \cdot A)u, u \rangle $ is 
$\frac{(-1)^{|\alpha + \beta + 1_j|}}{(\alpha + \beta + 1_j)!} \langle A_j A^\alpha u, A^\beta u \rangle.$ Invoking now the spanning property hypothesis, the action of the operator $A_j$ is determined for any vector of $H$.
\end{proof} 

It would be interesting to characterize in function theory terms the entire function $\langle \exp ( -z \cdot A)u, u \rangle $ arising from the above proposition. Also, the compressed resolvents
$\langle (z_0 - z \cdot A)^{-1} u, u \rangle$ associated to such tuples of matrices remain to be identified by some intrinsic properties among all rational functions.

\section{Examples} \label{Sec:Exp}
 
\subsection{A pair of non-commuting projections} \label{Sub:pair-nc-projections}

As described in Theorem~\ref{th:Lap conv}, the approximants of the Laplace transform are related to the Wigner distribution. Below, we elaborate on an example of a Wigner distribution and represent it in terms of the Radon transform, the Fourier transform, and the standard disc distribution.

Consider two PSD non-commutative $2\times 2$ matrices given by
\[
A_1=\begin{pmatrix}2&0\\0&0\end{pmatrix} \,\,\text{and }\, A_2=\begin{pmatrix}1&1\\1&1\end{pmatrix}.
\]
These are homotheties of two orthogonal projections. Consider the linear pencil:
$$
A(x,y)=xA_1+yA_2.
$$
Note that $A$ is PSD for $x,y >0$. Thus, by taking the unit vector $e=(1, 0) \in\mathbb{R}^2$, there exists a Wigner distribution $\mathcal{W}_e(A_1, A_2)$ which, by definition, satisfies the following condition.
\beas 
( \mathcal{W}_e(A_1, A_2)(u, v), \ \phi(xu+yv) ) = \langle \phi(A(x,y)) e, \ e \rangle = [\phi(A(x,y))]_{11},
\eeas
for any Schwartz class function $\phi \in \mathcal E({\mathbb R}^2)$. 

The eigenvalues of $A(x,y)$ are $ x+y \pm r$, where $r=\sqrt{x^2+y^2}.$ Hence, the corresponding spectral projections are $P_\pm=\frac{1}{2}(I\pm N)$, respectively, where $$
N=\frac1r\begin{pmatrix}x&y\\y&-x\end{pmatrix}
\quad\text{satisfies}\quad
N^2=I,
$$
and $A(x,y)=(x+y)I+r N$. Thus, for any $\phi \in \mathcal E(\mathbb{R}^2)$, we obtain
$$
\phi(A(x,y))
=
\frac{\phi(x+y+r)+\phi(x+y-r)}{2}I
+
\frac{\phi(x+y+r)-\phi(x+y-r)}{2}N.
$$
For notational simplicity, denoting the Wigner distribution $\mathcal{W}_e(A_1, A_2)$ by $\mathcal{W}$, we have the following closed form action.
$$ ( \mathcal{W}(u,v) , \phi(xu+yv) )
=
\frac12\left(1+\frac{x}{r}\right)\phi(x+y+r)
+
\frac12\left(1-\frac{x}{r}\right)\phi(x+y-r).
$$

As expected, this formula involves only the evaluation of the test function $\phi$ at points of the spectrum of $A(x,y)$. Below, we identify $\mathcal{W}$ as a distribution via the Fourier and Radon transforms. The first key step involves representing the linear pencil in terms of the Pauli matrices in quantum mechanics.

\vspace{2mm}

\subsubsection{Reduction to Pauli matrices}
Let us recall that the standard form of the $2\times 2$ Pauli matrices are given by
$$
\sigma_x = \begin{pmatrix} 0 & 1 \\ 1 & 0 \end{pmatrix}, \quad
\sigma_y = \begin{pmatrix} 0 & -i \\ i & 0 \end{pmatrix}, \quad
\sigma_z = \begin{pmatrix} 1 & 0 \\ 0 & -1 \end{pmatrix}.
$$
Then, writing $A$ in terms of these matrices, we get that 
\bea\label{eq:pau rep}
A(x,y)=(x+y)I+x\sigma_z+y\sigma_x.
\eea
{The canonical Pauli Wigner distribution $W_{\mathrm P}(s,t)$ is given by
$$
( W_{\mathrm P}, \psi(xs+yt) )
\bydef
\bigl[\psi(x\sigma_z+y\sigma_x)\bigr]_{11}, \ \  \text{for } \psi\in\mathcal E(\R^2),
$$} see for instance Section 4 in \cite{Anderson}.
To connect the two distributions $\mathcal{W}$ and $W_\mathrm P$, for a fixed $\phi\in \mathcal E(\R^2)$, and a fixed point $(x,y)\in\R^2$, we consider the following shifted function on $\R$.
\[
\psi(\lambda)=\phi(x+y+\lambda),\quad\text{for }\lambda\in\R.
\]
Consequently, by \eqref{eq:pau rep}, it follows that
\[
[\phi(A(x,y))]_{11}=\bigl[\psi(x\sigma_z+y\sigma_x)\bigr]_{11}.
\]
As a result, we obtain that
\beas
( \mathcal{W},\phi(xu+yv)) &=&( W_{\mathrm P},\psi(xs+yt) )\\
&=& ( W_{\mathrm{P}}, \phi(x+y+xs+yt) )\\
&=& ( W_{\mathrm{P}}, \phi(x(s+1)+y(t+1)) ),
\eeas
for any $\phi\in \mathcal E(\R^2)$.
Now, by applying the affine translation 
$$
T(s,t)=(s+1,t+1),\quad \text{for }(s,t)\in\R^2,
$$
we get that, in the sense of distribution
\bea\label{eq:final form}
\nonumber \mathcal{W}(u,v) &=& TW_{\mathrm P}(u,v)\\
&=& W_{\mathrm P}(u-1,v-1).
\eea
Using this relation, we can represent the Fourier transform of $W$ explicitly, which yields $\mathcal W$ by taking the inverse Fourier transform.
\vspace{2mm}

\subsubsection{Representation via the Fourier transform}
Recall that, for a compactly supported distribution $T$ on $\R^2$, the Fourier transform of $T$ is given by
\[
\widehat{T}(\xi,\eta)=( T(u,v), e^{\,i(\xi u+\eta v)} ), \quad \text{for } (\xi,\eta)\in\R^2.
\]
By \cite{SW}, it follows that $W_\mathrm P$, and hence by \eqref{eq:final form}, $\mathcal{W}$ has compact support. 
We first consider the Fourier transform of $W_\mathrm P$.
\[
\widehat{W_{\mathrm{P}}}(\xi,\eta) = ( W_{\mathrm{P}}(s,t), e^{i(\xi s + \eta t)} ) = \bigl[e^{i(\xi\sigma_z+\eta\sigma_x)}\bigr]_{11}, \quad \text{for } (\xi,\eta)\in\R^2.
\]
To compute the matrix exponential, note that $[\xi\sigma_z+\eta\sigma_x]^2=r^2I$, where $r^2=\xi^2+\eta^2$. Thus,
\beas
e^{i(\xi\sigma_z+\eta\sigma_x)}& = & \cos(r)I+i\frac{\sin(r)}{r} (\xi\sigma_z+\eta\sigma_x).
\eeas
Hence, 
\[
\widehat{W_{\mathrm{P}}}(\xi,\eta)=\bigl[e^{i(\xi\sigma_z+\eta\sigma_x)}\bigr]_{11}= \cos(r) + i\xi\frac{\sin(r)}{r},  \quad \text{for } (\xi,\eta)\in\R^2.
\]
Furthermore, from the relation in \eqref{eq:final form}, it follows that
\[
\widehat{\mathcal{W}}(\xi,\eta) = e^{i(\xi+\eta)}\,\widehat{W_{\mathrm{P}}}(\xi,\eta),  \quad \text{for } (\xi,\eta)\in\R^2.
\]
Hence,
\[
\widehat{\mathcal{W}}(\xi,\eta) =e^{\,i(\xi+\eta)}
\left(
\cos\sqrt{\xi^2+\eta^2}
+
i\,\xi\,\frac{\sin\sqrt{\xi^2+\eta^2}}{\sqrt{\xi^2+\eta^2}}
\right), \quad \text{for } (\xi,\eta)\in\R^2.
\]
Now, by taking the inverse Fourier transform of $W_\mathrm P$, we get that
\[
W_{\mathrm P}(s,t)
=
-\frac{1+s}{2\pi}\,
\operatorname{Pf}\!\left(1-s^2-t^2\right)_+^{-3/2},\quad\text{for }(s,t)\in\R^2,
\]
where $\operatorname{Pf}$ denotes the Hadamard finite part integral. This also shows that the support of $W_\mathrm P$ is the closed unit disc in $\R^2$. Finally, by \eqref{eq:final form}, we get the following explicit form of $\mathcal W$. 
\[
\mathcal W(u,v)=-\frac{u}{2\pi}\,
\operatorname{Pf}\!\left(1-(u-1)^2-(v-1)^2\right)_+^{-3/2},\quad \text{for }(u,v)\in\R^2.
\]
Furthermore, the support of $\mathcal W$ is the closed unit disk centered at $(1,1)$. 

\vspace{2mm}

\subsubsection{Representation via the Radon Transform}
Another way to characterize the Wigner distribution is via its Radon transform. We adopt the standard convention for the Radon transform of a compactly supported distribution $T$ on $\mathbb{R}^2$, defined by its action on a function $\psi\in\mathcal E(\R^2)$:
\bea\label{eq:radon defn}
\langle \mathcal R T(\omega,\cdot),\psi\rangle
=
( T(z),\psi(\omega\cdot z) ),
\quad
\text{for }\omega=(\omega_1,\omega_2)\in \mathbb S^1,
\eea
where $z=(s,t)\in\R^2$.
First, we compute the Radon transform of the Pauli Wigner distribution $W_P$.
Note that, writing $(x,y)=r(\omega_1,\omega_2)$, where $r=(x^2+y^2)^{\frac{1}{2}}$ and $(\omega_1,\omega_2)\in\mathbb S^1$, we can write the above definition as follows:
\bea\label{eq:Radon polar}
( T , \psi(xu+yv) )
=
( \mathcal R T(\omega,\cdot), q\mapsto \psi(r q) ),\quad \text{for } q\in \R.
\eea
Also, $x\sigma_z+y\sigma_x=r(\omega_1\sigma_z+\omega_2\sigma_x)$, and the matrix $\omega_1\sigma_z+\omega_2\sigma_x$ has eigenvalues $\pm1$. Its spectral projections are given by
$$
P_\pm(\omega)=\frac12\bigl(I\pm(\omega_1\sigma_z+\omega_2\sigma_x)\bigr).
$$
Hence, taking the $(1,1)$-entry of $P_\pm$, we get that
\[
\bigl[\psi(x\sigma_z+y\sigma_x)\bigr]_{11}
=
\frac{1+\omega_1}{2}\,\psi(r)
+
\frac{1-\omega_1}{2}\,\psi(-r).
\]
Hence by \eqref{eq:Radon polar}, we obtain that for $(\omega,q)\in\mathbb S^{1}\times \R$,
\bea\label{eq:Radon pauli}
\mathcal RW_{\mathrm P}(\omega,q)
=
\frac{1+\omega_1}{2}\,\delta(q-1)
+
\frac{1-\omega_1}{2}\,\delta(q+1),
\eea
where $\delta$ denotes the Dirac delta distribution. 

Finally, as $\mathcal W(u,v)=W_{\mathrm P}(u-1,v-1)$, it follows that 
\[
\mathcal R \mathcal{W}(\omega,q)=\mathcal RW_{\mathrm P}(\omega,q-\omega\cdot (1,1)),\quad \text{for }(\omega,q)\in\mathbb S^{1}\times \R.
\]

Consequently,
$$
\mathcal R \mathcal{W}(\omega,q)
=
\frac{1+\omega_1}{2}\,\delta(q-\omega_1-\omega_2-1)
+
\frac{1-\omega_1}{2}\,\delta(q-\omega_1-\omega_2+1), \quad \text{for }(\omega,q)\in\mathbb S^{1}\times \R.
$$

\vspace{2mm}

\subsubsection{Representation via the standard disc distribution} 
The standard disc distribution for a disc of radius $R$ and centered at the origin, denoted by $\mathbb D(0,R)$, is given by
\beas
G_R(s,t) \bydef \frac{H(R^2 - s^2 - t^2)}{2\pi\sqrt{R^2 - s^2 - t^2}},\quad \text{for } (s,t)\in \mathbb D(0,R),
\eeas
where $H$ is the Heaviside step function.
The Radon transform of $G_R$ is given by
\beas
\mathcal{R}G_R(\omega,q) = \frac{1}{2}H(R^2 - q^2), \quad \text{for }(\omega,q)\in\mathbb S^{1}\times \R.
\eeas
Furthermore, since the distributional derivative $H'(x)=\delta(x)$, it follows that
\bea\label{eq:rad deriv}
\mathcal R\ltt \partial_R G_R\rtt(\omega,q)=\partial_R\ltt \mathcal{R}G_R\rtt (\omega,q)= R\delta(R^2-q^2).
\eea
Also, by the \eqref{eq:radon defn}, it follows that, for any $\psi\in\mathcal E(\R^2)$,
\beas
\langle \mathcal{R}(\partial_s G_R)(\omega, \cdot), \psi \rangle &=& \langle \partial_s G_R(z), \psi(\omega \cdot z) \rangle\\
&=&- \langle G_R(z), \partial_s [\psi(\omega \cdot z)] \rangle\\
&=&- \omega_1 \langle G_R(z), \psi'(\omega \cdot z) \rangle\\
&=& - \omega_1 \langle \mathcal{R}G_R(\omega, \cdot), \psi' \rangle\\
&=& \langle \omega_1 \partial_q \mathcal{R}G_R(\omega, \cdot), \psi \rangle.
\eeas
This implies that 
\beas
\mathcal{R}(\partial_s G_R)(\omega,q) &=& \omega_1 \partial_q \mathcal{R}G_R\\
&=&  \omega_1 q \delta(R^2-q^2), \quad \text{for }(\omega,q)\in\mathbb S^{1}\times \R. 
\eeas
Combining this with \eqref{eq:rad deriv} and \eqref{eq:Radon pauli} we get that
\beas
\mathcal{R}\ltt(\partial_R-\partial_s) G_R|_{R=1}\rtt(\omega,q)&=&{\frac{1+\omega_1}{2}\delta(q-1) + \frac{1-\omega_1}{2}\delta(q+1)}\\
&=& \mathcal R(W_\mathrm P)(\omega,q),\quad \text{for }(\omega,q)\in\mathbb S^{1}\times \R.
\eeas
Hence, by the injectivity of the Radon transform, it follows that
\beas
W_\mathrm P(s,t)=(\partial_R-\partial_s) G_R|_{R=1}(s,t), \quad \text{for }(s,t)\in\mathbb D(0,1).
\eeas
Finally, using \eqref{eq:final form}, we get that
$$
\mathcal W(u,v) = \left(\partial_R - \partial_u) G_R(u-1, v-1) \right|_{R=1},\quad \text{for }(u,v)\in\mathbb D(0,1).
$$

\vspace{2mm}

\subsubsection{Commutative dilation of $A_1$ and $A_2$} 
Note that, the $4 \times 4$ PSD matrices
\begin{align*}
  \tilde{A}_1 = \begin{bmatrix}
      2 & 0 & 2 & 0 \\ 0 & 0 & 0 & 0\\ 2 & 0 & 2 & 0 \\ 0 & 0 & 0 & 0
  \end{bmatrix} \text{ and } \tilde{A}_2 =  \begin{bmatrix}
      1 & 1 & -1 & -1 \\ 1 & 1 & -1 & -1\\ -1 & -1 & 1 & 1 \\ -1 & -1 & 1 & 1
  \end{bmatrix}
\end{align*}
are commuting dilations of $A_1$ and $A_2$, respectively. For $\tilde{e} = (1, 0, 0, 0)$ in $\mathbb{R}^4$, there exists a Wigner distribution $\tilde{W}_{\tilde{e}}(u,v)$ satisfying:
\beas 
( \tilde{W}_{\tilde{e}}(u, v), \ \phi(xu+yv) ) &=& [\phi(\tilde{A}_1 x + \tilde{A}_2 y)]_{11},
\eeas
for any $\phi \in \mathcal E({\mathbb R}^2)$. But $\tilde{W}_{\tilde{e}}$ being a compactly supported distribution, we can choose $\phi$ to be a smooth unbounded function. In particular, 
$$
g(x, y) \bydef ( \tilde{W}_{\tilde{e}}(u, v),\  \exp(-xu -yv) )= [\exp(-x\tilde{A}_1 - y \tilde{A}_2)]_{11}.
$$
Again, $\tilde{A}_1 \tilde{A}_2 = \tilde{A}_2 \tilde{A}_1 =0$. So, the power series expansion of $\exp(-x\tilde{A}_1 - y \tilde{A}_2)$ around the origin is given by
$$
\exp(-x\tilde{A}_1 - y \tilde{A}_2) = I_4 -x\tilde{A}_1 - y \tilde{A}_2 + \frac{1}{2} x^2 \tilde{A}_1^2 + \frac{1}{2} y^2 \tilde{A}_2^2 + \dots.
$$
Therefore, the degree $1$ Taylor polynomial of $g$ around the origin is $$1-2x - y$$ which is the same as the degree $1$ Taylor polynomial of the function 

\noindent $( \mathcal{W}_e(u, v),\  \exp(-xu -yv) )$ around the origin. 

{\em This illustrates how the compressions of $\tilde{A}_1$ and $\tilde{A}_2$ provide a linear approximation around the origin for the completely monotone function $g$ on $\mathbb{R}^2_+$}.

\bigskip

The \emph{characteristic function} of a cone $\Gamma$ is defined by
$$\varphi(z) = \int_{\Gamma^\ast} e^{-z\cdot x} \, dx_1 dx_2, \quad z \in T_\Gamma.$$
The characteristic function plays a central role in the theory of Hardy and Bergman spaces on symmetric tube domains \cite[Chapter I]{Faraut-Koranyi}. In the following examples, we analyze $\varphi$ from our interpolation and approximation perspective when $\Gamma$ is the positive orthant and the Lorentz cone.

\subsection{The positive orthant} 
The characteristic function for the positive quadrant $\mathbb{R}^2_+$ is $\varphi(x, y) = \frac{1}{xy}$. Since $\mathbb{R}^2_+$ is a product space endowed with a product Lebesgue measure, applying the standard one-dimensional Gaussian quadrature rule to each coordinate axis yields the appropriate finitely determined interpolants in both the Laplace and Stieltjes–Fantappi\`e settings. So, we omit the explicit computations involving tensor products of Laguerre orthogonal polynomials and their associated Jacobi matrices.

\subsection{The Lorentz cone}
Let us look at our approximation scheme on the so-called Lorentz cone, defined as
$$
\Gamma = \{x=(x_1, x_2, x_3)\in \mathbb{R}^3 : \ x_1^2 -x_2^2 -x_3^2 \geq 0 \ \text{ and } x_1 \geq 0 \}.
$$
It is a symmetric cone and so, $\Gamma^* = \Gamma$. The characteristic function for this cone is:
$$\varphi(p) = \int_\Gamma e^{- p \cdot x} \, dx$$

Fix a point $\lambda = (1, 0, 0) \in \operatorname{int} \Gamma$.
It can be shown that $\varphi(\lambda) =2\pi$ and 
$$\varphi(x, y, z) = 2\pi\ (x_1^2 - x_2^2 - x_3^2)^{-3/2}.$$
The Taylor polynomial of degree $2$ of $\varphi$ is 
\begin{align*}
   1 - 3(x-1) + 6(x_1-1)^2 + \frac{3}{2}x_2^2 + \frac{3}{2}x_3^2.
\end{align*}

Now we compute the bounded holomorphic interpolant (on the tube domain $T_\Gamma$) which matches the Taylor coefficients of $\varphi$ at $\lambda$ up to order $2$ in two ways. The first one is the scheme described in Theorem \ref{th:Lap conv}, and the second one is a heuristic approach where we find a Wigner distribution yielding an interpolant. 

\vspace{2mm}

\subsubsection{Via compressed coordinate multipliers}
Consider the space 
$$H_2 = \operatorname{span}_\mathbb{R}\{x^\alpha \exp(-\lambda \cdot x/2) : |\alpha| \leq 2 \}$$ and the PSD Hankel kernel 
$$
\Xi = (\Xi_{\alpha, \beta})_{|\alpha|, |\beta| \leq 2}, \ \ \Xi_{\alpha, \beta} = (-1)^{|\alpha + \beta|} \partial^{\alpha + \beta}_p\varphi|_{p=\lambda}.
$$
Then $\operatorname{rank} \Xi =10$ and so, the inner product on $H_2$ given by
$$
\langle x^\alpha \exp(-\lambda \cdot x/2) , x^\beta \exp(-\lambda \cdot x/2)  \rangle = \Xi_{\alpha, \beta} , \ \ |\alpha|, |\beta| \leq 2, 
$$
makes $H_2$ a real Hilbert space of dimension $10$. 
By Gram-Schmidt orthogonalization we get an orthonormal basis $\{e_j : 1\leq j \leq 10\}$ of $H_2$, where
\begin{align*}
    & e_1 = \frac{1}{\sqrt{2\pi}} e^{-x_1/2},\, e_2 = \frac{1}{\sqrt{6\pi}} (x_1 - 3) e^{-x_1/2},\, e_3 = \frac{1}{\sqrt{48\pi}} (x_1^2 - 8x_1 + 12) e^{-x_1/2}, \\
    & e_4 = \frac{1}{\sqrt{120\pi}} (x_2^2 - x_3^2) e^{-x_1/2}, \, e_5 = \frac{1}{\sqrt{60\pi}} \left(x_2^2 + x_3^2 - \frac{1}{2}x_1^2\right) e^{-x_1/2}, \\
    & e_6 = \frac{1}{\sqrt{6\pi}} x_2 e^{-x_1/2}, \, e_7 = \frac{1}{\sqrt{30\pi}} x_2(x_1 - 5) e^{-x_1/2}, \,  e_8 = \frac{1}{\sqrt{6\pi}} x_3 e^{-x_1/2}, \\
    & e_9 = \frac{1}{\sqrt{30\pi}} x_3(x_1 - 5) e^{-x_1/2}, \, e_{10} = \frac{1}{\sqrt{30\pi}} x_2x_3 e^{-x_1/2}.   
\end{align*}

Then the matrix representation of the compressions $X^{(2)}_j = P_{H_2} X_j |_{H_2}$ for $1\leq j \leq 3$ is as below:
$$ X^{(2)}_1 = \begin{bmatrix}
3 & \sqrt{3} & 0 & 0 & 0 & 0 & 0 & 0 & 0 & 0 \\
\sqrt{3} & 5 & 2\sqrt{2} & 0 & 0 & 0 & 0 & 0 & 0 & 0 \\
0 & 2\sqrt{2} & 7 & 0 & 0 & 0 & 0 & 0 & 0 & 0 \\
0 & 0 & 0 & 7 & 0 & 0 & 0 & 0 & 0 & 0 \\
0 & 0 & 0 & 0 & 7 & 0 & 0 & 0 & 0 & 0 \\
0 & 0 & 0 & 0 & 0 & 5 & \sqrt{5} & 0 & 0 & 0 \\
0 & 0 & 0 & 0 & 0 & \sqrt{5} & 7 & 0 & 0 & 0 \\
0 & 0 & 0 & 0 & 0 & 0 & 0 & 5 & \sqrt{5} & 0 \\
0 & 0 & 0 & 0 & 0 & 0 & 0 & \sqrt{5} & 7 & 0 \\
0 & 0 & 0 & 0 & 0 & 0 & 0 & 0 & 0 & 7
\end{bmatrix}, 
$$
$$ 
X^{(2)}_2= \begin{bmatrix}
0 & 0 & 0 & 0 & 0 & \sqrt{3} & 0 & 0 & 0 & 0 \\
0 & 0 & 0 & 0 & 0 & 2 & \sqrt{5} & 0 & 0 & 0 \\
0 & 0 & 0 & 0 & 0 & \frac{1}{\sqrt{2}} & \sqrt{10} & 0 & 0 & 0 \\
0 & 0 & 0 & 0 & 0 & \sqrt{5} & 2 & 0 & 0 & 0 \\
0 & 0 & 0 & 0 & 0 & \sqrt{\frac{5}{2}} & \sqrt{2} & 0 & 0 & 0 \\
\sqrt{3} & 2 & \frac{1}{\sqrt{2}} & \sqrt{5} & \sqrt{\frac{5}{2}} & 0 & 0 & 0 & 0 & 0 \\
0 & \sqrt{5} & \sqrt{10} & 2 & \sqrt{2} & 0 & 0 & 0 & 0 & 0 \\
0 & 0 & 0 & 0 & 0 & 0 & 0 & 0 & 0 & \sqrt{5} \\
0 & 0 & 0 & 0 & 0 & 0 & 0 & 0 & 0 & 2 \\
0 & 0 & 0 & 0 & 0 & 0 & 0 & \sqrt{5} & 2 & 0
\end{bmatrix}, 
$$
and
$$
X^{(2)}_3= \begin{bmatrix}
0 & 0 & 0 & 0 & 0 & 0 & 0 & \sqrt{3} & 0 & 0 \\
0 & 0 & 0 & 0 & 0 & 0 & 0 & 2 & \sqrt{5} & 0 \\
0 & 0 & 0 & 0 & 0 & 0 & 0 & \frac{1}{\sqrt{2}} & \sqrt{10} & 0 \\
0 & 0 & 0 & 0 & 0 & 0 & 0 & -\sqrt{5} & -2 & 0 \\
0 & 0 & 0 & 0 & 0 & 0 & 0 & \sqrt{\frac{5}{2}} & \sqrt{2} & 0 \\
0 & 0 & 0 & 0 & 0 & 0 & 0 & 0 & 0 & \sqrt{5} \\
0 & 0 & 0 & 0 & 0 & 0 & 0 & 0 & 0 & 2 \\
\sqrt{3} & 2 & \frac{1}{\sqrt{2}} & -\sqrt{5} & \sqrt{\frac{5}{2}} & 0 & 0 & 0 & 0 & 0 \\
0 & \sqrt{5} & \sqrt{10} & -2 & \sqrt{2} & 0 & 0 & 0 & 0 & 0 \\
0 & 0 & 0 & 0 & 0 & \sqrt{5} & 2 & 0 & 0 & 0
\end{bmatrix}. 
$$

Let $L=(x-\lambda) \cdot \boldsymbol{X}^{(2)}$.
We are interested in the function 
$$F_2(x_1, x_2, x_3) \bydef \langle \exp(-L) \exp(-\lambda \cdot x /2), \exp(-\lambda \cdot x /2) \rangle = 2\pi \langle \exp(-L) e_1, e_1 \rangle.$$
Expand $\exp(-L)$ into its Taylor series 
$$\sum_{n=0}^\infty \frac{(-1)^n}{n!} ((x_1-1)X^{(2)}_1 + x_2 X^{(2)}_2 + x_3 X^{(3)}_3 )^n.$$ Then the coefficient of $(x_1-1)^j x^k_2 x^l_3$ in the Taylor expansion of the $(1,1)$ entry of $\exp(-L)$ at $\lambda$ is $\frac{(-1)^{j+k+l}}{(j+k+l)!}\langle S e_1, e_1 \rangle$, where $S$ is the sum of words of length $j+k+l$ with $j$ many $X^{(2)}_1$, $k$ many $X^{(2)}_2$ and $l$ many $X^{(2)}_3$ as letter. In this way, we obtain the Taylor expansion of $\langle \exp(-L) e_1, e_1 \rangle$ at $\lambda$ as 
$$
 1 - 3(x_1-1) + 6(x_1-1)^2 + \frac{3}{2}x_2^2 + \frac{3}{2}x^2_3 - 10(x_1-1)^3 - \frac{15}{2}(x_1-1)x^2_2 - \frac{15}{2}(x_1-1)x_3^2 + \dots.
$$
Again, we compute the derivatives of $\varphi$ at $\lambda$ to get
$\varphi(\lambda) = 2\pi$, $\frac{\partial \varphi}{\partial x_1}(\lambda) = -6\pi$, $\frac{\partial \varphi}{\partial x_2}(\lambda) = 0$, $\frac{\partial \varphi}{\partial x_3}(\lambda) = 0$, $\frac{\partial^2 \varphi}{\partial x_1^2}(\lambda) = 24\pi$, $\frac{\partial^2 \varphi}{\partial x_1 x_2}(\lambda) = 0 $, $\frac{\partial^2 \varphi}{\partial x_2 x_3}(\lambda) = 0$ and so on. Thus, we see that the Taylor coefficients of $F_2$ and $\varphi$ at $\lambda$ are the same up to order $5$.

In this case the interpolant $2\pi \langle \exp(-L)\ e_1, e_1 \rangle$ is obtained by multiplying $2\pi$ with the Wigner distribution $\mathcal{W}_{e_1}(\boldsymbol{X}^{(2)})$ associated with $\boldsymbol{X}^{(2)}$.

\vspace{2mm}

\subsubsection{Via an appropriate Wigner distribution}\label{Subsub:Lorentz}
Here we find heuristically a triple of  symmetric matrices constrained by a prescribed  Taylor polynomial of degree $2$ of the Laplace transform evaluated at $\lambda$.
Specifically we define:
$$
A_1= \left[\begin{matrix}
    3 & 0 & 0 & \sqrt{3} \\ 0 & 2 & 0 & 0\\ 0 & 0 & 2 & 0 \\ \sqrt{3} & 0 & 0 & 2
\end{matrix} \right], \, 
A_2=\left[\begin{matrix}
    0 & \sqrt{3} & 0 & 0 \\ \sqrt{3} & 0 & 0 & 0\\ 0 & 0 & 0 & 0 \\ 0 & 0 & 0 & 0
\end{matrix} \right], 
A_3=\left[\begin{matrix}
   0 & 0 & \sqrt{3} & 0 \\ 0 & 0 & 0 & 0\\ \sqrt{3} & 0 & 0 & 0 \\ 0 & 0 & 0 & 0
\end{matrix} \right].
$$
An elementary algebraic argument shows that that the joint numerical range of the triple $A=(A_1, A_2, A_3)$ lies inside $\Gamma$. Let $\mathcal{W}_{e_1}$ be the Wigner distribution associated to $A$ and the standard unit vector $e_1$ of $\mathbb{R}^4$. Then we infer as before:
$$
(\mathcal{W}_{e_1}(y), \  \exp(- (x-\lambda)\cdot y)) ) = \langle \exp(-L) e_1, e_1 \rangle,
$$
where $L = (x_1 -1)A_1 + x_2 A_2 + x_3 A_3.$
We are looking for the Taylor polynomial of $2\pi (\exp(-L))_{11}$, the $(1, 1)$-entry of $2\pi \exp(-L)$. Thus, it is enough to look at the Taylor series for $2\pi \exp(-L)$ around the origin up to degree $2$ which is 
$$2\pi (I - L + \frac{1}{2}L^2).$$
Calculating the above sum we see that the Taylor polynomial of $2\pi [\exp(-L)]_{11}$ at the origin is 
$$2\pi - 6\pi(x_1-1) + 12\pi (x_1-1)^2 +3\pi x_2^2 + 3\pi x_3^2,$$
and it perfectly matches with the degree $2$ Taylor polynomial of the characteristic function $\varphi$ at $\lambda$.

\section{Concluding remarks}

\begin{remark}
The Borel transform and its inverse allow one to transition back and forth between the Laplace and Stieltjes–Fantappi\`e transforms. Therefore, it is natural to ask whether there is a passage between the corresponding interpolants. Due to the different nature of the wandering vectors in the Krylov subspaces associated to the jets of an interior point $\lambda$ of $\Gamma$, or $(\lambda, 1)$ of $\Gamma \times \mathbb{R}_+$, we do not expect a straightforward solution to this question.

There is, however, one notable exception: if the jet data is prescribed at the vertex of the cone (i.e., $\lambda = 0$), implying that the original measure admits all moments, then the real Hilbert space $H_n$ reduces to the space spanned by the monomials $\{x^\alpha : |\alpha| \leq n \}$ in both Subsections~\ref{Subsec:Laplace-interpolation} and \ref{Subsec:Fantappie-interpolation}.
Consequently, the symmetric linear $d$-tuples $\mathbf{X}^{(n)}$ are identical in both cases, and the vectors simplify to $u = \exp(-\lambda \cdot x / 2) = 1 = v$. The Wigner distribution $\mathcal{W}_u(\mathbf{X}^{(n)})$ then yields the expected approximants:
$$
F_n(z) = \langle \mathcal{W}_u(\mathbf{X}^{(n)}), e^{-z\cdot x} \rangle = \langle \exp(-z\cdot \mathbf{X}^{(n)}) u, u \rangle, \quad z \in T_\Gamma,
$$
and 
\begin{align*}
    \Phi_n(z, z_0) &= \int_{0}^\infty e^{-\tau z_0} F_n(\tau z) \, d\tau \\
    &= \left\langle \left( \int_{0}^\infty \exp\bigl(-\tau (z_0 I + z\cdot \mathbf{X}^{(n)})\bigr) \, d\tau \right) u , u \right\rangle \\
    &= \langle (z_0 I + z\cdot \mathbf{X}^{(n)})^{-1} u, u \rangle, \quad (z, z_0) \in T_{\Gamma \times \mathbb{R}_+}.
\end{align*}

\end{remark}
 
\begin{remark} The commutative tuple dilation appearing at the end of the first toy example is ad-hoc. In general, the size of the matching commutative dilation is much bigger, cf. Section~\ref{Sec:Comm-dilation}. Similarly, the triple of $4 \times 4$ non-commuting matrices matching via their Wigner transform the data in the Lorentz case example is also ad-hoc.

This suggests that factoring the Hankel kernel can be bypassed by other techniques aimed at constructing an interpolation represented by Laplace transforms of Wigner distributions. Specifically, an alternative way to extract the non-commuting tuple $A=(A_1, \dots, A_d)$ from a given Taylor coefficient data of their Laplace transform in the sense of Weyl's calculus:
$$ \langle \exp ( - x\cdot A) u, u\rangle = \sum_{|\alpha| \leq 2n+1} c_\alpha x^\alpha + O(|x|^{2n+2})$$
is not excluded, at least in particular situations. Proposition~\ref{unique-tuple} isolates the key sought properties of such a system of matrices.
\end{remark} 

\begin{remark} The present article has exclusively focused on function theory on tube domains, symmetric or not. A continuation of our analysis on bounded symmetric domains is natural and necessary. As an early step, an investigation of Fantappi\`e transform of positive measures supported by the unit ball in ${\mathbb C}^d$ is contained in \cite{McCP}. Some non-trivial uniform bounds of Weyl's operational calculus on the joint numerical range of a tuple of non-commuting operators are obtained there.

\end{remark}

\begin{remark} The additional spectral argument $z_0$ in the Stieltjes-Fantappi\`e transform invites to investigate a continued fraction expansion with variable coefficients, of the formal power series
$$ \int_{\Gamma^\ast}\frac{d\mu(x)}{z_0 + z\cdot x} \sim \sum_k \frac{c_k(z)}{z_0^{k+1}}.$$
To our knowledge such a study and comparison with known multivariable Pad\'e type approximation techniques, such as \cite{Guillaume-Huard}, is not known.

\end{remark}

\noindent \textbf{Acknowledgement:} Several AI platforms (Gemini, ChatGPT and Claude) have partially contributed to some symbolic computations and numerical examples contained in this article. We are grateful to Prof. Tirthankar Bhattacharyya for valuable comments leading to an improvement of the presentation of this work.

The research of the first author is supported by ANRF-NPDF (Grant No. PDF/2025/002836).
He is also thankful for the generous support provided by the SPARC (Project No. 1698) at the initial stage of this work at the Indian Institute of Science, Bengaluru. The research of the second author is supported by the institute post-doctoral fellowship of the Indian Institute of Technology Palakkad, and was partially supported at the initial stage of this work by the DST FIST program-2021 [TPN-700661] and the ANRF Grant WEA/2023/000017.

\end{document}